\documentclass[a4paper]{amsart}

\usepackage{amssymb}

\usepackage{mathrsfs}

\usepackage{tikz}

\usepackage{tikz-cd}

\usepackage{graphicx}

\usepackage{hyperref}

\usepackage{latexsym}

\usepackage[all]{xy}

\usepackage{yhmath}

\usepackage{amsmath}

\usepackage{amsthm}

\usepackage{hhline}

\usepackage{multirow}

\numberwithin{equation}{section}

\newtheorem*{thm}{Theorem}
\newtheorem{thmx}{Theorem}
\newtheorem*{prop}{Proposition}
\newtheorem*{lem}{Lemma}
\newtheorem*{cor}{Corollary}

\theoremstyle{definition}
\newtheorem*{definition}{Definition}
\newtheorem*{example}{Example}

\theoremstyle{remark}

\newtheorem*{remark}{Remark}
\newtheorem*{notation}{Notation}

\newcommand{\N}{\mathbb{N}}
\newcommand{\Z}{\mathbb{Z}}

\newcommand{\g}{\mathfrak{g}}
\newcommand{\bo}{\mathfrak{b}}
\newcommand{\h}{\mathfrak{h}}
\newcommand{\n}{\mathfrak{n}}
\newcommand{\Ca}{\mathscr{C}_R}
\newcommand{\Da}{\mathscr{D}}
\newcommand{\cohDn}{\text{coh}(\widehat{\Dflag})}
\newcommand{\A}{\mathcal{A}_q}
\newcommand{\Aw}{\mathcal{A}_{q,w}}
\newcommand{\B}{\mathcal{B}_q}
\newcommand{\C}{\mathcal{C}}
\newcommand{\D}{\mathcal{D}}
\newcommand{\E}{\mathcal{E}}
\newcommand{\F}{\mathcal{F}}
\newcommand{\I}{\mathcal{I}}
\newcommand{\LH}{\mathcal{LH}}
\newcommand{\M}{\mathcal{M}}
\newcommand{\K}{\mathcal{K}}
\newcommand{\Nn}{\mathcal{N}}

\newcommand{\flag}{\mathcal{M}_{B_q}(G_q)}
\newcommand{\Dflag}{\D_{B_q}^\lambda (G_q)}
\newcommand{\flaghat}{\widehat{\flag}}

\newcommand{\Ubar}{U_k}
\newcommand{\Ufbar}{U^{\text{fin}}_k}

\newcommand{\Of}{\mathcal{O}}
\newcommand{\Oq}{\mathcal{O}_q}
\newcommand{\Oqhat}{\widehat{\Oq}}
\newcommand{\Oqwhat}{\widehat{\Of_{q,w}}}

\newcommand{\OqB}{\mathcal{O}_q(B)}
\newcommand{\OqBhat}{\widehat{\OqB}}

\newcommand{\Uqnhat}{\widehat{U_q}}
\newcommand{\Uqbnhat}{\widehat{U_q^{\geq 0}}}
\newcommand{\Dqnhat}{\widehat{\D_q}}
\newcommand{\ban}{\mathbf{Ban}_L}

\newcommand{\Mod}[1]{\mathbf{Mod}(#1)}
\newcommand{\Comod}[1]{\mathbf{Comod}(#1)}
\newcommand{\norm}[1]{\left|\left| #1 \right|\right|}
\newcommand{\abs}[1]{\left| #1 \right|}
\newcommand{\ten}[2]{#1\widehat{\otimes}_L#2}
\newcommand{\tenR}[2]{#1\widehat{\otimes}_R#2}

\DeclareMathOperator{\Hom}{\text{Hom}}
\DeclareMathOperator{\gr}{\text{gr}}
\DeclareMathOperator{\id}{\text{id}}
\DeclareMathOperator{\coker}{\text{coker}}
\DeclareMathOperator{\im}{\text{Im}}
\DeclareMathOperator{\coim}{\text{Coim}}
\DeclareMathOperator{\htt}{\text{ht}}
\DeclareMathOperator{\ad}{\text{ad}}
\DeclareMathOperator{\loc}{\text{Loc}}
\DeclareMathOperator{\coh}{\text{coh}}
\DeclareMathOperator{\ind}{\text{Ind}}
\DeclareMathOperator{\res}{\text{Res}}

\begin{document}

\title{A Beilinson-Bernstein Theorem for analytic quantum groups}
\date{}
\author{Nicolas Dupr\'e}
\address{Department of Pure Mathematics and Mathematical Statistics\\
Centre for Mathematical Sciences, Wilberforce Road\\
Cambridge, CB3 0WB, United Kingdom}
\email{nd332@cam.ac.uk}

\begin{abstract}
We introduce a $p$-adic analytic analogue of Backelin and Kremnizer's construction of the quantum flag variety of a semisimple algebraic group, when $q$ is not a root of unity and $\vert q-1\vert<1$. We then define a category of $\lambda$-twisted $D$-modules on this analytic quantum flag variety. We show that when $\lambda$ is regular and dominant and when the characteristic of the residue field does not divide the order of the Weyl group, the global section functor gives an equivalence of categories between the coherent $\lambda$-twisted $D$-modules and the category of finitely generated modules over $\widehat{U_q^\lambda}$, where the latter is a completion of the ad-finite part of the quantum group with central character corresponding to $\lambda$. Along the way, we also show that Banach comodules over the Banach completion $\OqBhat$ of the quantum coordinate algebra of the Borel can be naturally identified with certain topologically integrable modules.
\end{abstract}

\maketitle
\thispagestyle{empty}

\setcounter{tocdepth}{1}
\tableofcontents

\section{Introduction}

\subsection{Background and motivation}\label{motiv}

Let $L$ be a complete discrete valuation field of mixed characteristic $(0,p)$, with discrete valuation ring $R$, uniformizer $\pi$ and residue field $k$. We fix an element $q\in R^\times$ and assume that $q\equiv 1\pmod{\pi}$ and that $q$ is not a root of unity. Ardakov and Wadsley have recently started an ongoing program aiming to develop $p$-adic analytic analogues of $D$-modules in order to understand $p$-adic representation theory, see \cite{Wadsley1, Wadsley2, Wadsley3, Ardakov1}. Their aim is to use $p$-adic analytic localisation results analogous to the classical theorem of Beilinson-Bernstein \cite{BB} in order to better understand locally analytic representations of $p$-adic groups, which were introduced by Schneider and Teitelbaum in a series of papers including \cite{SchTeit01, SchTeit02, SchTeit03}. There have also been other approaches at using localisation techniques to understand locally analytic representations, notably by Schmidt \cite{Schmidt3} and Patel, Schmidt and Strauch \cite{PSS1, PSS2, PSS3}.

Let us  briefly recall one of Ardakov and Wadsley's main results. Let $\mathbf{G}$ be a simply connected split semisimple algebraic group over $R$ with $R$-Lie algebra $\g$ and let $X$ be its flag scheme $\mathbf{G}/\mathbf{B}$. In \cite{Wadsley1}, they defined a family $(\widehat{\mathcal{U}_{n,L}})_{n\geq 0}$ of Banach completions of the envloping algebra $U(\g_L)$ of the $L$-Lie algebra $\g_L:=\g\otimes_R L$. Moreover, for a weight $\lambda$, they introduced a family $(\widehat{\D_{n,L}^\lambda})_{n\geq 0}$ of sheaves of completed deformed twisted crystalline differential operators on $X$. Their theorem then states:

\begin{thm}[\cite{Wadsley1}] For any $n\geq 0$ and for $\lambda$ regular and dominant, the global section functor gives an equivalence of categories between coherent sheaves of $\widehat{\D_{n,L}^\lambda}$-modules and finitely generated $\widehat{\mathcal{U}_{n,L}}$-modules with central character corresponding to $\lambda$.
\end{thm}

Our aim is to prove an analogue of the above Theorem when working with quantum groups, where for simplicity we only treat the case $n=0$ in this paper. The study of quantum groups in a $p$-adic analytic setting was first proposed by Soibelman in \cite{Soibelman08}, where he introduced quantum deformations of the algebras of locally analytic functions on $p$-adic Lie groups and of the corresponding distribution algebra. His ideas were also heavily influenced by the aforementioned work of Schneider and Teitelbaum. This paper of Soibelman then inspired a short note of Lyubinin \cite{Lyubinin1} and also a different approach for GL$_2$ in \cite{wald}. Recently, there has also been a new approach at constructing $p$-adic analytic quantum groups using Nichols algebra in \cite{Craig2}. However, besides these, not much work has been done in this area. In \cite{Nico2}, we constructed quantum analogues of the Arens-Michael envelope of $\g_L$ and of the algebra of rigid analytic functions on the analytification of $G_L$, and proved that these were Fr\'echet-Stein algebras. We also constructed several Banach completions of those algebras, and some of these objects feature in this paper. Our hope is that more work will be done to pursue these efforts. The theory of quantum groups has strong links with the representation theory of algebraic groups in positive characteristic. We expect that a successful theory of $p$-adic analytic quantum groups would have similar links with the representation theory of $p$-adic groups, and we view our work as a first effort towards developing such a theory.

Recently, there has also been some work hinting at noncommutative analogues of rigid analytic geometry in \cite{KreBB}. In this light, we think that defining noncommutative analogues of analytic flag varieties as we do in this paper is interesting in its own right. It would be interesting to compare our constructions with their general framework.

\subsection{Quantum flag varieties and quantum $D$-modules}

The proof of Theorem \ref{motiv} relied on the classical Beilinson-Bernstein theorem, and similarly we will use a quantum group analogue of it due to Backelin and Kremnizer \cite{QDmod1}\footnote{We note that there exists a different approach to quantum $D$-modules and Beilinson-Bernstein by Tanisaki \cite{Tanisaki}.}. We briefly recall their constructions. Let $U_q$ be the quantized enveloping algebra of $\g_L$. Let $\Oq$ be the quantized coordinate algebra of $G_L$, and let $\OqB$ be the quotient Hopf algebra of $\Oq$ corresponding to a Borel subgroup of $G_L$. Backelin and Kremnizer then define the quantum flag variety to be the category $\flag$ of $\OqB$-equivariant $\Oq$-modules. Specifically, an object of this category is an $\Oq$-module equipped with a right $\OqB$-comodule structure such that $\Oq$-action map is a comodule homomorphism. In this language, the global section functor $\Gamma$ is the functor of taking $\OqB$-coinvariants. They then define the ring of quantum differential operators on $G_L$ to be the smash product algebra $\D_q=\Oq\# U_q$, and a $\lambda$-twisted $D$-module becomes an object $M$ of the quantum flag variety equipped with an additional $\D_q$-action such that the $\OqB$-coaction and the action of the quantum Borel subalgebra $U_q^{\geq 0}\subset U_q\subset \D_q$ `differ by $\lambda$' (here $\lambda$ is an element of the character group $T_P$ of the weight lattice). There is also a distinguished object $\D_q^\lambda$ which represents global sections in the category of $\lambda$-twisted $D$-modules. The precise definitions are made in Section 4. Their main theorem is that, when $\lambda$ is regular and dominant, the global section functor gives an equivalence of categories between $\lambda$-twisted $D$-modules and modules over $\Gamma(\D_q^\lambda)$.

Nothing stops us from making completely analogous definitions using certain Banach completions $\Oqhat$, $\OqBhat$ and $\widehat{\D_q}$ of these algebras (see section \ref{intro3} below). That allows us to define what we call the analytic quantum flag variety as the category $\flaghat$ of $\OqBhat$-equivariant Banach $\Oqhat$-modules, meaning that the objects of this category are Banach $\Oqhat$-modules which are also Banach $\OqBhat$-comodules such that the $\Oqhat$-action map is a comodule homomorphism (see section \ref{compflag}). We note that this category is not abelian. Instead it fits into Schneiders' framework of quasi-abelian categories \cite{qacs}. In particular it has a derived category and, under suitable conditions, we can right derive left exact functors. The global section functor $\Gamma$ here is also the functor of taking $\OqBhat$-coinvariants, and we use this framework of quasi-abelian categories to make sense of the cohomology of $\Gamma$. We can then define $\lambda$-twisted $D$-modules to be objects $\M$ in $\flaghat$ which are equipped with an additional $\widehat{\D_q}$-action such that the $\OqBhat$-coaction and the action of $\widehat{U_q^{\geq 0}}$ differ by $\lambda$. There is also a distinguished object $\widehat{\D_q^\lambda}$ which represents global sections. Again, the precise definitions are made in section \ref{defofdhat}.

\subsection{General strategy}\label{intro3}

Let us briefly outline the argument used by Ardakov and Wadsley in \cite{Wadsley1} to prove that one gets an equivalence of categories in Theorem \ref{motiv}. We will employ essentially the same strategy.

\begin{enumerate}
\item They first work with integral versions of classical algebraic $D$-modules and show that large enough twists of coherent $D$-modules are acyclic and generated by their global sections. Using this, they then show that the category of coherent $\widehat{\D_{n,L}^\lambda}$-modules has a family of generators obtained from taking certain twists of $\widehat{\D_{n,L}^\lambda}$. In particular those are $\pi$-adic completions of algebraic $D$-modules.
\item The first step essentially reduces the problem to working with those coherent $\widehat{\D_{n,L}^\lambda}$-modules which can be `uncompleted'. They then show that these are generated by their global sections. This uses the classical Beilinson-Bernstein theorem.
\item Finally, they show that completions of acyclic coherent $D$-modules are also acyclic. This uses technical facts about the cohomology of a projective limit of sheaves.
\item Once you know that coherent $\widehat{\D_{n,L}^\lambda}$-modules are acyclic and generated by their global sections, the result follows from standard general facts.
\end{enumerate}

In order to adapt this, we are first required to work with integral forms of quantum groups and the corresponding integral quantum flag variety, see sections \ref{Lusztig}, \ref{prelimonOq} \& \ref{Rflag}. Specifically, there is an integral form $\A$ of $\Oq$ which was first defined by Andersen, Polo and Wen \cite{Andersen}. By taking $\B$ to be its image in the quotient Hopf algebra $\OqB$, we are then able to define the category $\Ca$ of $\B$-equivariant $\A$-modules. We can also define an integral form $\D$ of the ring $\D_q$, and use it to define $\lambda$-twisted $D$-modules in $\Ca$ (here $\lambda$ is an element of $T_P^R$, the character group over $R$ of the weight lattice). These integral forms allow us to define the Banach completions we mentioned above by simply setting $\Oqhat:=\widehat{\A}\otimes_R L$, $\OqBhat:=\widehat{\B}\otimes_R L$ and $\widehat{\D_q}:=\widehat{\D}\otimes_R L$ respectively.

Unlike in the first step above, we are not able to show that large enough twists of coherent $\D$-modules are acyclic and generated by global sections, but we manage to show it for those which are annihilated by $\pi$. This turns out to be enough for the first two steps to work. Most of this paper is then spent developing the correct tools from noncommutative algebraic geometry in the category $\flaghat$ in order for the ideas used in the third step to even make sense.

\subsection{\v{C}ech complexes}

To have a version of step (iii) above, we need to work with the right sort of complexes, computing the cohomology of global sections, in order to apply the argument on the cohomology of a projective limit. To do so, it is convenient to work with proj categories. Indeed, the classical flag variety is isomorphic to Proj$(\mathcal{O}(G/N))$, and Backelin-Kremnizer showed that $\flag$ is equivalent to Proj$(\Oq(G/N))$ in the sense of Artin-Zhang \cite{ArtinZhang}. We show that the integral quantum flag variety enjoys the same property. To obtain this result, one problem we ran into is that, while it is well-known that the algebra $\Oq$ is Noetherian, it isn't known in general whether its integral form $\A$ is also Noetherian (in type $A$, it is known to be true from Polo's appendix in \cite{Andersen}). That makes it non-trivial to define the objects which should play the role of coherent modules. Thankfully, we were able to prove that the integral form of $\Oq(G/N)$ is Noetherian, and using this we showed that the Noetherian objects in $\Ca$ are precisely those which are finitely generated over $\A$, see Theorem \ref{coherent}. Once this obstacle is cleared, the proof that we have a noncommutative projective scheme is essentially identical to the one in \cite{QDmod1}.

This result is essential because it allows us to define our promised complex which computes the cohomology of global sections for these integral forms. We think of this as a \v{C}ech-like complex. Using the Proj description of $\Ca$, one can in a suitable sense cover the category with analogues of the Weyl group translates of the big cell, see sections \ref{ore} \& \ref{Cech}. The complexes are then obtained using general constructions from Rosenberg \cite{Rosenberg}. After taking $\pi$-adic completions, the objects of $\Ca$ are then naturally sent to another intermediate category, which we unoriginally call $\widehat{\Ca}$ and which is in some sense an integral form of $\flaghat$. We use the Weyl group localisations mentioned above to write down an analogue of our \v{C}ech-like complexes in this new integral category. After extending scalars, this gives us a \v{C}ech-like complex in the category $\flaghat$. This is the right object in order to apply the arguments from step (iii).

\subsection{Main results}

At several stages of this paper, we work with Banach comodules over $\OqBhat$. We first give a more explicit description of these objects. We begin by defining what we call topologically integrable modules over a certain completion $\widehat{U^{\text{res}}(\bo)}$ of $U_q^{\geq 0}$, see section \ref{integrable}. Roughly, these are modules where the torus acts topologically semisimply and the positive part acts locally topologically nilpotently. The definition is partly inspired from work of F\'eaux de Lacroix \cite{Lacroix}, who developed a notion of semisimplicity for topological Fr\'echet modules (note that we already used the notion of topological semisimplicity in our previous work \cite[Section 5]{Nico2}). Our first main result is then:

\begin{thmx}\label{thmA}
The category $\Comod{\OqBhat}$ of Banach right $\OqBhat$-comodules is canonically equivalent to the category of topologically integrable $\widehat{U^{\text{res}}(\bo)}_L$-modules.
\end{thmx}

This result allows for a more intuitive understanding of what these comodules are, and also draws further parallels between our constructions and standard notions that appear in $p$-adic representation theory. We note that Banach comodules over a Banach coalgebra have also been studied in a more general, categorical setting in \cite{Craig1}.

Our next result is that the cohomology of $\Gamma$ in $\flaghat$ can be computed using the \v{C}ech-like complexes described above:

\begin{thmx}\label{thmB}
For any $\M\in\flaghat$, the standard complex $\check{C}(\M)$ computes $R\Gamma(\M)$.
\end{thmx}

The precise definition of this complex is made in section \ref{ore2}. We note that as a consequence of this, we obtain in Corollary \ref{Cech2} that $\Gamma$ has finite cohomological dimension (something which wasn't obvious beforehand!). Both of these are essential in order to obtain a Beilinson-Bernstein theorem, but we also think of them as interesting results in their own right. We view our analytic quantum flag variety as being in some sense a noncommutative analytic space, and these results make it feasible to work with it.

Finally, with all the above at hand, we are able to run the strategy from section \ref{intro3} to obtain our version of Beilinson-Bernstein localisation. Before stating it, we need to introduce a few more notions. We call a $D$-module in $\flaghat$ coherent if it is finitely generated over $\widehat{\D_q}$. Moreover, $U_q$ contains a subalgebra $U_q^{\text{fin}}$, called its finite part, which is the subalgebra of elements on which the adjoint action of $U_q$ is locally finite. This has an integral form $U^{\text{fin}}\subseteq U$ which contains the centre of $U$, and given $\lambda$ we may form a quotient $U^\lambda=U^{\text{fin}}\otimes_{Z(U)}R_\lambda$. Completing, we obtain an algebra $\widehat{U_q^\lambda}=\widehat{U^\lambda}\otimes_R L$ which is a Noetherian Banach algebra. Our Beilinson-Bernstein localisation then states:

\begin{thmx}\label{thmC}
Suppose $\lambda\in T_P^k$ is regular and dominant, and assume that $p$ is a very good prime for the root system of $\g$. Then the functor $\Gamma$ of global sections and the localisation functor \emph{Loc}$_\lambda$ are quasi-inverse equivalences of categories between the category \emph{$\cohDn$} of $\lambda$-twisted coherent $\Dqnhat$-modules on the analytic quantum flag variety and the category of finitely generated modules over $D:=\Gamma(\widehat{\D^\lambda_q})$. Moreover, there is a surjective algebra homomorphism $\widehat{U_q^\lambda}\to D$ which is an isomorphism whenever $p$ does not divide the order of the Weyl group.
\end{thmx}

See section \ref{BB2} for the definitions of the localisation functor Loc$_\lambda$, and see sections \ref{cohomology} and \ref{condonp} for the meaning of the conditions on $p$ and for the definition of the set $T_P^k$. We simply note here that the condition that $p$ does not divide $|W|$ is automatically satisfied if $p$ is larger than the Coxeter number.

Thus, we may think of the above Theorem as saying that the category $\flaghat$ is $D$-affine, and moreover we can identify this category with the category of finitely generated $\widehat{U_q^\lambda}$-modules under some reasonable condition on $p$. We note that in order to just get $D$-affinity without any statement on global sections, the condition on $p$ can be weakened to say that it is a good prime, see Theorem \ref{BB2}.

\subsection{Computation of global sections} We were made aware that there may be gaps in the proof of the computation of global sections in \cite[Proposition 4.8]{QDmod1}, see \cite[Remark 5.4]{Tanisaki2}. We simply point out here that these potential issues do not affect our work as we never use their computation of global sections.

Firstly, in our proof of Theorem \ref{thmC}, we only use the $D$-affinity of quantum flag variety in order to obtain the corresponding result for Banach completions. And indeed, the equivalence of categories given by the global section functor from the category of $\lambda$-twisted $D$-modules on the quantum flag variety to the category of modules over the global sections of $\D^\lambda_q$ in \cite{QDmod1} does not require the full computation of global sections. It does rely on a quantum analogue of the Beilinson-Bernstein `key lemma', but that only needs for there to be a map $U_q^\lambda\to \Gamma(\D^\lambda_q)$ in order to interpret global sections of $D$-modules as modules over $U_q^{\text{fin}}$. That way one obtains a splitting of some particular maps at the level of global sections which can be used to prove that every $D$-module is acyclic and generated by its global sections, see the proof of \cite[Theorem 4.12]{QDmod1}. But we do not need to know that $U_q^\lambda\to \Gamma(\D^\lambda_q)$ is an isomorphism for that part of the Beilinson-Bernstein theorem to hold (this is also true classically).

Secondly, our computation of global sections via the homomorphism $\widehat{U_q^\lambda}\to D$ in Theorem \ref{global1} does not use the computation of global sections at the uncompleted level. Instead, our arguments go via reduction modulo $\pi$, where $q$ becomes 1 and the situation becomes non-quantum. Therefore, what we crucially need is instead the computation of global sections of the sheaves of twisted crystalline differential operators on the flag variety in positive characteristic, obtained in \cite[Proposition 3.4.1]{BMR}.

\subsection{Future work}

We only deal with the analogue of the case $n=0$ from Theorem \ref{motiv} in this paper. We plan to extend our results to an analogue of the $n>0$ case. This would hopefully allow us to then obtain a localisation theorem for the (finite part of) the quantum Arens-Michael envelope $\wideparen{U_q}$ defined in \cite{Nico2}, analogously to the non-quantum situation \cite{Ardakov1}.

\subsection{Structure of the paper}

In Section 2, we recall all the necessary facts about quantum groups and their integral forms. In particular, we give an explicit description of the categories of $\A$-comodules and $\B$-comodules as the categories of integrable modules over Lusztig's integral forms of $U_q$ and $U_q^{\geq 0}$ respectively. We believe this to be well-known, but we could not find any suitable reference for this, so we included a proof. This needed some general facts about Hopf $R$-algebras which we included in an appendix.

In Section 3, we give a detailed account of some of the theory regarding the finite part of the quantum group and the Harish-Chandra isomorphism. We specifically prove analogues of known results for the integral form of $U_q^{\text{fin}}$, partly as an attempt to demystify this finite part and also in order to understand its structure after reducing modulo $\pi$. We then use this analysis to prove that the quantum Harish-Chandra isomorphism restricts to an isomorphism at the level of the integral form, if $p$ does not divide $|W|$. This is very useful later in the paper, when we compute global sections.

The main body of our work starts in Section 4, where we recall all the main definitions and constructions from \cite{QDmod1} and replicate them for integral forms. We then prove that this integral category is a noncommutative projective scheme. In doing so, we make heavy use of results about the cohomology of the induction functor for quantum groups from Andersen, Polo and Wen \cite{Andersen}. Then we use this to construct a \v{C}ech-like complex which computes the cohomology of global sections. Finally, we define $D$-modules and we prove that if a coherent $D$-modules is annihilated by $\pi$, then large enough twists of it are acyclic and generated by their global sections. This requires for $p$ to be a good prime. The main tool required here is the identification of the cohomology of (twists of) $\D^\lambda/\pi \D^\lambda$ with the sheaf cohomology of sheaves of twisted differential operators on the flag variety

In Section 5, we recall facts about completed tensor products and introduce Banach comodules over a Banach coalgebra. We then introduce topologically integrable modules over the Banach completion of Lusztig's integral form for $U_q^{\geq 0}$, and show that these are equivalent to Banach $\OqBhat$-comodules. Using results on topological semisimplicity from our previous work \cite{Nico2}, it follows from the fact that any Banach $\OqBhat$-comodule $\M$ embeds topologically into $\ten{\M}{\OqBhat}$, equipped with the comodule structure $1\widehat{\otimes} \widehat{\Delta}$.

In Section 6, we then introduce the categories $\widehat{\Ca}$ and $\flaghat$, and recall all the necessary facts on quasi-abelian categories. We then construct a \v{C}ech-like complex and prove that it computes the cohomology of global sections. The main technical tool we need here is some flatness results for completed tensor products from \cite{Andreas1}. The theorem then follows essentially by using the fact that it holds for lattices modulo $\pi^n$ for every $n$. Finally, in Section 7 we put everything together to prove our Beilinson-Bernstein theorem. The arguments for the equivalence of categories are essentially those from \cite{Wadsley1}, with some adjustments. As explained earlier, the computation of global sections makes heavy use of the corresponding computation in \cite{BMR} and of the modulo $\pi$ computations of Section 3.

\subsection{Acknowledgements} Most of the material in this paper formed a significant part of the author's PhD thesis, which was being produced under the supervision of Simon Wadsley. We are very grateful to him for his continued support and encouragement throughout this research, without which writing this paper would not have been possible. We would also like to thank him for communicating privately a proof to us which inspired our arguments in Section 5. We are also thankful to Andreas Bode for his continued interest in our work, and for communicating  Proposition \ref{exactloc} to us before his work was written up. Finally, we wish to thank Kobi Kremnizer for a useful conversation on quantum groups and proj categories. The author's PhD was funded by EPSRC.

\subsection{Conventions and notation}

Unless explicitly stated otherwise, the term ``module'' will be used to mean \emph{left} module, and Noetherian rings are both left and right Noetherian. Also, all of our filtrations on modules or algebras will be positive and exhaustive unless specified otherwise. Following \cite[Def 2.7]{Wadsley1}, an $R$-submodule $W$ of an $L$-vector space will be called a \emph{lattice} if $V=LW$ and $W$ is $\pi$-adically separated, i.e $\bigcap_{n\geq 0}\pi^n W=0$. Given an $R$-module $M$, we denote by $\widehat{M}$ its $\pi$-adic completion and write $\widehat{M_L}:=\widehat{M}\otimes_R L$.

Given an $L$-normed vector space $X$, we denote by $X^\circ$ its unit ball. Given a Banach algebra $A$, a Banach $A$-module $M$ will always be assumed to have action map of norm at most 1, i.e $M^\circ$ will always be assumed to be an $A^\circ$-module.

In a Hopf algebra $H$, we use Sweedler's notation for the comultiplication, i.e we write $\Delta(h)=\sum h_1\otimes h_2$. All our comodules will be \emph{right} comodules unless stated otherwise.

Finally, while we talked about $R$-group schemes and their corresponding Lie algebras in this introduction, quantum groups are defined purely in terms of the root system and are traditionally defined starting from complex Lie algebras and algebraic groups, regardless of what the base field is. This is the convention we follow as well. Hence we let $\g$ be a complex semisimple Lie algebra. We fix a Cartan subalgebra $\h\subseteq \g$ contained in a Borel subalgebra. We choose a positive root system and we denote the simple roots by $\alpha_1, \ldots, \alpha_n$. Let $C=(a_{ij})$ denote the Cartan matrix. We let $G$ be the simply connected semisimple algebraic group corresponding to $\g$, and we let $B$ be the Borel subgroup corresponding to the positive root system, and let $N\subset B$ be its unipotent radical. Let $\bo=\text{Lie}(B)$ and $\n=\text{Lie}(N)$. Let $W$ be the Weyl group of $\g$, and let $\langle\, , \rangle$ denote the standard normalised $W$-invariant bilinear form on $\h^*$. Let $P\subset\h^*$ be the weight lattice and $Q\subset P$ be the root lattice. Let $T_P$ denote the abelian group $\Hom_\Z(P, L^\times)$. We will use the additive notation for this group. Let $d$ be the smallest natural number such that $\langle  \mu, P\rangle\subset \frac{1}{d}\Z$ for all $\mu\in P$. Let $d_i=\frac{\langle \alpha_i,\alpha_i\rangle}{2}\in\{1, 2, 3\}$ and write $q_i:=q^{d_i}$.

We make the following two assumptions. First, we assume that $q^\frac{1}{d}$ exists in $R$ and that $q^\frac{1}{d}\equiv 1\pmod{\pi}$. Then for each $\lambda\in P$, we have an associated element in $T_P$ sending a given $\mu\in P$ to $q^{\langle \lambda, \mu\rangle}$, which we will also denote by $\lambda$. Secondly, we assume that $p>2$ and, if $\g$ has a component of type $G_2$, we furthermore restrict to $p>3$. This ensures that $p$ does not divide any non-zero entry of the Cartan matrix.

All the above algebraic groups and Lie algebras have $k$-forms, and we write $G_k, \g_k,\ldots$ etc to denote them.

\section{Preliminaries on quantum groups and their integral forms}

\subsection{Quantized enveloping algebra}\label{PrelimonUq} We begin by recalling basic facts about quantized enveloping algebras (see eg \cite[Chapter I.6]{BroGoo02} for more details). For $n\in\Z$ and $t\in L$, we write $[n]_t:=\frac{t^n-t^{-n}}{t-t^{-1}}$. We then set the quantum factorial numbers to be given by $[0]_t!=1$ and $[n]_t!:=[n]_t[n-1]_t\cdots [1]_t$ for $n\geq 1$. Then we set
$$
{n \brack i}_t:=\frac{[n]_t!}{[i]_t![n-i]_t!}
$$
when $n\geq i\geq 1$.

\begin{definition}
The simply connected quantized enveloping algebra $U_q(\g)$ is defined to be the $L$-algebra with generators $E_{\alpha_1},\ldots, E_{\alpha_n}$, $F_{\alpha_1}, \ldots, F_{\alpha_n}$, $K_\lambda$, $\lambda\in P$, satisfying the following relations:
\begin{align*}
& K_\lambda K_\mu=K_{\lambda+\mu},\quad K_0=1,\\
& K_\lambda E_{\alpha_i} K_{-\lambda}=q^{\langle\lambda, \alpha_i\rangle}E_{\alpha_i},\quad K_\lambda F_{\alpha_i} K_{-\lambda}=q^{-\langle\lambda, \alpha_i\rangle}F_{\alpha_i},\\
& [E_{\alpha_i}, F_{\alpha_j}]=\delta_{ij}\frac{K_{\alpha_i}-K_{-\alpha_i}}{q_i-q_i^{-1}},\\
& \sum_{l=0}^{1-a_{ij}}(-1)^l {1-a_{ij}\brack l}_{q_i} E_{\alpha_i}^{1-a_{ij}-l}E_{\alpha_j}E_{\alpha_i}^l=0\quad (i\neq j),\\
& \sum_{l=0}^{1-a_{ij}}(-1)^l {1-a_{ij}\brack l}_{q_i} F_{\alpha_i}^{1-a_{ij}-l}F_{\alpha_j}F_{\alpha_i}^l=0\quad (i\neq j).
\end{align*}
\end{definition}

We will also abbreviate $U_q(\g)$ to $U_q$ when no confusion can arise as to the choice of Lie algebra $\g$. We can define Borel and nilpotent subalgebras, namely $U_q^{\geq 0}$ is the subalgebra generated by all the $K's$ and the $E's$, and $U_q^+$ is the subalgebra generated by all the $E's$. Similarly, $U_q^-$ is defined to be the subalgebra generated by all the $F's$. There is also a Cartan subalgebra given by $U_q^0:=L[K_\lambda: \lambda\in P]$, which is isomorphic to the group algebra $LP$. There is an algebra automorphism $\omega$ of $U_q$ defined by $\omega(E_{\alpha_i})=F_{\alpha_i}$, $\omega(F_{\alpha_i})=E_{\alpha_i}$ and $\omega(K_\lambda)=K_{-\lambda}$.

Recall that $U_q$ is a Hopf algebra with operations given by
$$
\begin{array}{l l l}
\Delta(K_\lambda)=K_\lambda\otimes K_\lambda\quad & \varepsilon(K_\lambda)=1\quad & S(K_\lambda)=K_{-\lambda}\\
\Delta(E_{\alpha_i})=E_{\alpha_i}\otimes 1 + K_{\alpha_i}\otimes E_{\alpha_i}\quad & \varepsilon(E_{\alpha_i})=0\quad & S(E_{\alpha_i})=-K_{-\alpha_i}E_{\alpha_i}\\
\Delta(F_{\alpha_i})=F_{\alpha_i}\otimes K_{-\alpha_i} + 1\otimes F_{\alpha_i}\quad & \varepsilon(F_{\alpha_i})=0\quad & S(F_{\alpha_i})=-F_{\alpha_i}K_{\alpha_i}
\end{array}
$$
for $i=1,\ldots, n$ and all $\lambda\in P$. Then $U_q^{\geq 0}$ is a sub-Hopf algebra of $U_q$.

Also recall that there is a triangular decomposition
$$
U_q\cong U_q^-\otimes_L U_q^0\otimes_L U_q^+
$$
and that $U_q^\pm$ have bases consisting of PBW type monomials. More specifically, if $\beta_1,\ldots, \beta_N$ are the positive roots, ordered in a particular way, then there are elements $E_{\beta_1}, \ldots, E_{\beta_N}$ of $U_q^+$ such that the set of all ordered monomials $E_{\beta_1}^{m_1}\cdots E_{\beta_N}^{m_N}$ forms a basis for $U_q^+$. We now let $F_{\beta_j}:=\omega(E_{\beta_j})$ and the corresponding monomials in the $F$'s will form a basis of $U_q^-$. The triangular decomposition immediately gives a PBW type basis for $U_q$, namely it consists of monomials of the form
$$
M_{\boldsymbol{r}, \boldsymbol{s}, \lambda}:=\boldsymbol{F^r}K_\lambda\boldsymbol{E^s}
$$
where $\boldsymbol{r}, \boldsymbol{s}\in\Z^N_{\geq 0}$. We recall that the \emph{height} of such a monomial is defined to be
$$
\htt(M_{\boldsymbol{r}, \boldsymbol{s}, \lambda}):=\sum_{j=1}^N(r_j+s_j)\htt(\beta_j)
$$
where $\htt(\beta):=\sum_{i=1}^n a_i$ for a positive root $\beta=\sum_i a_i \alpha_i$. This gives rise to a positive filtration on $U_q$  defined by
$$
F_iU_q:=L\text{-span}\{M_{\boldsymbol{r}, \boldsymbol{s}, \lambda} : \htt(M_{\boldsymbol{r}, \boldsymbol{s}, \lambda})\leq i\}.
$$
This filtration can actually be extended to a multifiltration as follows. The associated graded algebra $U^{(1)}=\gr U_q$ can be seen to have the same presentation as $U_q$, with the exception that now all the $E$'s commute with all the $F$'s. Moreover it is isomorphic to $U_q$ as a vector space. We can then make $U^{(1)}$ into a $\Z^{2N}_{\geq 0}$-filtered algebra, by assigning to each monomial $M_{\boldsymbol{r}, \boldsymbol{s}, \lambda}$ the degree $(r_1,\ldots,r_N, s_1,\ldots, s_N)$ where we impose the reverse lexicographic orderin ordering on $\Z^{2N}_{\geq 0}$. Denote the corresponding associated graded algebra of $U^{(1)}$ by $U^{(2N+1)}$. This algebra is known to be $q$-commutative over $L$ (see \cite[Proposition 10.1]{DeConProc}). Here we say that an $L$-algebra $A$ is $q$-commutative over a subalgebra $B$ if it is finitely generated over $B$, say by $x_1,\ldots, x_m$, such that the $x_i$ normalise $B$ and for all $1\leq i\leq j\leq m$ there are $n_{ij}\in \frac{1}{d}\Z$ such that $x_ix_j=q^{n_{ij}}x_jx_i$. We regord here a noncommutative analogue of Hilbert's basis theorem, which follows directly from \cite[Theorem 1.2.10]{noncomalg} and induction.

\begin{lem} If $A$ is $q$-commutative over $B$ and $B$ is Noetherian, then so is $A$.
\end{lem}

Hence we see that $U_q$ is a Noetherian $L$-algebra.

\subsection{Integral forms of $U_q$}\label{Lusztig} We now recall details about two integral forms that we will work with. First recall the notation:
$$
E_{\alpha_i}^{(s)}:=\frac{E_{\alpha_i}^s}{[s]_{q_i}!},\quad F_{\alpha_i}^{(s)}:=\frac{F_{\alpha_i}^s}{[s]_{q_i}!},
$$
for any integer $s\geq 0$. Then Lusztig's integral form $U^{\text{res}}$ is defined to be the $R$-subalgebra of $U_q$ generated by $K_\lambda$ ($\lambda\in P$) and all $E_{\alpha_i}^{(r)}$ and $F_{\alpha_i}^{(r)}$ for $r\geq 0$ and $1\leq i\leq n$. Recall that for $1\leq i\leq n$, $c,t\in\Z$ with $t\geq 0$ we define
$$
{K_{\alpha_i}; c \brack t}=\prod_{j=1}^t\frac{K_{\alpha_i}q_i^{c-j+1}-K_{\alpha_i}^{-1}q_i^{-(c-j+1)}}{q_i^j-q_i^{-j}}.
$$
Then by \cite[11.1, p.238]{Jantzen} we have that all such ${K_{\alpha_i}; c \brack t}$ lie in $U^{\text{res}}$. Also note that by \cite[Theorem 6.7]{Lusztig1} $U^{\text{res}}$ has a triangular decomposition and a PBW type basis, so that $U^{\text{res}}$ is free over $R$.

There is an $R$-subalgebra $(U^{\text{res}})^0$ generated by all $K_\lambda$ and all ${K_{\alpha_i}; c \brack t}$. We let $U^{\text{res}}(\bo)$ denote the $R$-subalgebra of $U^{\text{res}}$ generated by $(U^{\text{res}})^0$ and all $E_{\alpha_i}^{(r)}$ for $r\geq 0$ and $1\leq i\leq n$. By \cite[Lemma 1.1]{Andersen}, for each $\lambda\in P$ there is a unique character $\psi_\lambda:(U^{\text{res}})^0\to R$ defined by
\begin{equation}\label{character}
\psi_\lambda(K_\mu)=q^{\langle \lambda, \mu\rangle}\quad\text{and}\quad \psi_\lambda\left( {K_{\alpha_i}; c \brack t} \right)= {\langle \lambda, \alpha_i^\vee\rangle+ c \brack t}_{q_i}.
\end{equation}
We will say these characters are of type $\boldsymbol{1}$.

Given a $U^{\text{res}}$-module $M$ and a character $\psi$ as above of $(U^{\text{res}})^0$, we write $M_\psi$ for the elements $m\in M$ such that $um=\psi(u)m$ for all $u\in(U^{\text{res}})^0$. We now recall the notion of integrable module from \cite[1.6]{Andersen}:

\begin{definition} A $U^{\text{res}}$-module $M$ is said to be \emph{integrable of type $\boldsymbol{1}$} if it is a sum of weight spaces which all correspond to a character of type $\boldsymbol{1}$ as described above and if in addition, for every $m\in M$, there is $r>>0$ such that $m$ is killed by $E^{(r)}$ and $F^{(r)}$. Similarly we define a $U^{\text{res}}(\bo)$-module to be integrable of type $\boldsymbol{1}$ if it is the sum of its weight spaces corresponding to type $\boldsymbol{1}$ characters and for every $m\in M$, $E^{(r)}m=0$ for $r>>0$.
\end{definition}

Since all our characters will always be of type $\boldsymbol{1}$ we will often just say `integrable' to mean `integrable of type $\boldsymbol{1}$'.

The second integral form we will need is the De Concini-Kac integral form $U$. This is defined to be the $R$-subalgebra of $U_q$ generated by $E_{\alpha_i}, F_{\alpha_i} (1\leq i\leq n), K_\lambda (\lambda\in P)$. This algebra has a similar presentation to $U_q$. If we write $[K_{\alpha_i};m]:={K_{\alpha_i}; m \brack 1}$ for $m\in\Z$ and $1\leq i\leq n$, then $U$ is generated as an $R$-algebra by $E_{\alpha_i}, F_{\alpha_i}, [K_{\alpha_i};0] (1\leq i\leq n), K_\lambda (\lambda\in P)$ with the same relations as $U_q$ except that the commutator relation between $E_{\alpha_i}$ and $F_{\alpha_j}$ is replaced by the two relations
\begin{align*}
&[E_{\alpha_i}, F_{\alpha_j}]=\delta_{ij} [K_{\alpha_i};0],\\
&(q_i-q_i^{-1})[K_{\alpha_i};0]=K_{\alpha_i}-K_{\alpha_i}^{-1}.
\end{align*}
Note that $U$ is a Hopf $R$-algebra. For example we have the identity
$$
\Delta([K_{\alpha_i};0])=[K_{\alpha_i};0]\otimes K_{\alpha_i} +K_{\alpha_i}^{-1}\otimes [K_{\alpha_i};0].
$$
Note that we also have the equality
$$
[K_{\alpha_i};m]=[K_{\alpha_i};0]q_i^{-m}+K_{\alpha_i}[m]_{q_i}
$$
for all $m\in \Z$, and so $U$ contains all $[K_{\alpha_i};m]$. 

We showed in \cite[Section 4]{Nico2} that $U$ has a triangular decomposition $U\cong U^-\otimes_R U^0\otimes_R U^+$ where $U^{\pm}$ is the $R$-subalgebra generated by the $E_{\alpha_i}$'s, respectively $F_{\alpha_i}$'s, and $U^0$ is the $R$-subalgebra generated by $[K_{\alpha_i};0] (1\leq i\leq n), K_\lambda (\lambda\in P)$. Moreover $[K_{\alpha_i};m]\in U^0$ for all $m\in\Z$ by the above. We also showed that $U^\pm$ has a PBW basis, more specifically that the PBW monomials which form an $L$-basis of $U_q^\pm$ are also an $R$-basis of $U^\pm$.

Note that both of these integral forms are $\pi$-adically separated since $U\subset U^{\text{res}}$ and $U^{\text{res}}$ is free over $R$. We finish by describing the relationship between the reduction modulo $\pi$ of $U$ and $U^{\text{res}}$ and classical objects. We write $U_k:=U/\pi U$ and $U^{\text{res}}_k=U^{\text{res}}/\pi U^{\text{res}}$.

\begin{prop}\begin{enumerate}
\item \emph{(\cite[Proposition 9.2.3]{ChaPre})} The quotient $k$-algebra $U_k/(K_{\varpi_1}-1,\ldots, K_{\varpi_n}-1)$ is isomorphic to $U(\g_k)$.
\item \emph{(\cite[8.15]{Lusztig1})} The quotient $U^{\text{res}}_k/(K_{\varpi_1}-1,\ldots, K_{\varpi_n}-1)$ is isomorphic to the hyperalgebra of the group $G_k$.
\end{enumerate} 
\end{prop}

\subsection{Quantized coordinate rings and their integral forms} \label{prelimonOq} We now recall the construction of the quantized coordinate algebra $\Oq$. For any module $M$ over an $L$-Hopf algebra $H$, and for any $f\in H^*$ and $v\in M$, the matrix coefficient $c^M_{f,v}\in H^*$ is defined by
$$
c^M_{f,v}(x):=f(xv)\quad\quad\text{for } x\in H.
$$
Also recall from \cite[Theorem 5.10]{Jantzen} that for each $\lambda\in P$ there is a unique irreducible representation of type $\boldsymbol{1}$, $V(\lambda)$, of $U_q$ and that these form a complete list of such representations. The quantized coordinate ring $\Oq$ is then defined to be the $L$-subalgebra of the Hopf dual $U_q^\circ$ generated by the matrix coefficients of the modules $V(\lambda)$ for $\lambda\in P^+$. In fact, from \cite[I.7-I.8]{BroGoo02}, it is a finitely generated, Noetherian $L$-algebra, and it is a sub-Hopf algebra of $U_q^\circ$. There is also a quantized coordinate algebra of the Borel $\OqB$. Since $U_q^{\geq 0}$ is a Hopf-subalgebra of $U_q$, the restriction maps yields a Hopf algebra homomorphism $\Oq\to (U_q^{\geq 0})^\circ$ and we let $\OqB$ denote its image.

We now recall how the integral forms of $\Oq$ and $\OqB$ are defined. Let $U^{\text{res}}$ be Lusztig's integral form defined in above. Let $\mathcal{J}$ denote the set of ideals $I$ in $U^{\text{res}}$ such that $U^{\text{res}}/I$ is a finite free $R$-module. We now consider the set $\mathscr{I}$ consisting of ideals $I\in\mathcal{J}$ such that $I\cap(U^{\text{res}})^0$ contains a finite intersection of ideals $\ker(\psi_\lambda)$. Note that for any $R$-module $M$, we may view $\Hom_R(U^{\text{res}},M)$ as a $U^{\text{res}}$-module via $(x\cdot f)(y)=f(yx)$ for all $x,y\in U^{\text{res}}$. In \cite[Definition 1.10]{Andersen}, a so-called induction functor from the trivial subalgebra was defined. It takes any $R$-module $M$ to the subrepresentation $H(M)$ of $\Hom_R(U^{\text{res}},M)$ given by all elements in the sum the weight spaces in $\Hom_R(U^{\text{res}},M)$ which are killed by all $E_{\alpha_i}^{(r)}$ and $F_{\alpha_i}^{(r)}$ for $r>>0$. In other words $H(M)$ is the largest integrable subrepresentation of $\Hom_R(U^{\text{res}},M)$. We then define the integral form of the quantized coordinate algebra to be $\A:=H(R)$. By \cite[Corollary 1.30]{Andersen}, we have $f\in H(M)$ if and only if $f$ kills an ideal $I\in \mathscr{I}$. In particular,
$$
\A=\{f\in (U^{\text{res}})^* : f|_I=0 \text{ for some } I\in\mathscr{I}\}.
$$
So $\A$ is a sub-Hopf algebra of $(U^{\text{res}})^\circ$ (see Definition \ref{hopfdual}) and it may be viewed as the algebra of matrix coefficients of finite free $U^{\text{res}}$-modules of type $\boldsymbol{1}$. In particular the comultiplication on it makes it into a $(U^{\text{res}})^\circ$-comodule and hence we may view it as a $U^{\text{res}}$-module by Proposition \ref{comod} (and that agrees with the definition of the $U^{\text{res}}$-action on $H(R)$). Moreover by \cite[Theorem 1.33]{Andersen}, $\A$ is free over $R$.

Next, we look at the Borel subalgebra $U^{\text{res}}(\bo)$ of $U^{\text{res}}$. Let $\I$ be the set of $f\in\A$ such that $f|_{U^{\text{res}}(\bo)}=0$. The Hopf algebra homomorphism $\A\to U^{\text{res}}(\bo)^\circ$ given by restriction has kernel precisely $\I$ and so we see that $\I$ is a Hopf ideal and that $\B:=\A/\I\subseteq U^{\text{res}}(\bo)^\circ$ is a Hopf algebra. Similarly to the above, \cite{Andersen} defined an induction functor from the trivial subalgebra to $U^{\text{res}}(\bo)$ in a completely analogous way: if $M$ is an $R$-module, we define $\mathcal{H}(M)$ to be the largest integrable submodule of $\Hom_R(U^{\text{res}}(\bo),M)$. By \cite[Proposition 2.7(ii) and (iii)]{Andersen} we have that $\B=\mathcal{H}(R)$ and so it is integrable, and it is free as an $R$-module.

\subsection{The categories of comodules}\label{comodcat} We now recall how the category of $\A$-comodules (respectively $\B$-comodules) can be identified with integrable $U^{\text{res}}$-modules (respectively $U^{\text{res}}(\bo)$-modules). We expect this to be well-known but we did not find a suitable reference for it, so we provide proofs. To that end, we use general results about $R$-Hopf algebras which we've written in the appendix.

Since $\A=H(R)$, it is integrable with the $U^{\text{res}}$-module structure described above. Note that for any $R$-module $M$ there is a natural map $M\otimes_R \A\to H(M)\subseteq \Hom_R(U^{\text{res}},M)$ which is the composite of the map $M\otimes_R \A\to M\otimes_R (U^{\text{res}})^*$, coming from the inclusion $\A\subseteq (U^{\text{res}})^*$, and the map $\theta_M$ from Corollary \ref{comod}. By abuse of notation we also denote this map by $\theta_M$. For the Borel, we have again a map $\theta_M: M\otimes_R \B\to \mathcal{H}(M)$ for any $R$-module $M$. Moreover we have again that $f\in \mathcal{H}(M)$ if and only if $f$ kills an ideal $I$ of $U^{\text{res}}(\bo)$ such that $U^{\text{res}}(\bo)/I$ is finitely generated and $I\cap(U^{\text{res}})^0$ contains a finite intersection of ideals $\ker(\psi_\lambda)$.

The next result immediately follows from the above:

\begin{lem} If $M$ is torsion-free as an $R$-module then $\Hom_R(U^{\text{res}},M)/H(M)$ and $\Hom_R(U^{\text{res}}(\bo),M)/\mathcal{H}(M)$ are torsion-free. In particular $(U^{\text{res}})^*/\A$ and $U^{\text{res}}(\bo)^*/\B$ are torsion free.
\end{lem}

\begin{proof}
If $\pi^n f\in\Hom_R(U^{\text{res}},M)$ kills an ideal in $\mathscr{I}$, then so does $f$ as $M$ is torsion-free. An analogous argument applies to $\mathcal{H}(M)$. The last part follows by putting $M=R$.
\end{proof}

Since $\A$ and $\B$ are sub Hopf algebras of $(U^{\text{res}})^\circ$ and $U^{\text{res}}(\bo)^\circ$ respectively, it follows that any comodule over $\A$ (respectively $\B$) is a comodule over $(U^{\text{res}})^\circ$ (respectively $U^{\text{res}}(\bo)^\circ$). Thus we may view comodules over $\A$ and $\B$ as locally finite modules over $U^{\text{res}}$ and $U^{\text{res}}(\bo)$ respectively. This defines functors from the categories of $\A$-comodules and $\B$-comodules to the categories of locally finite $U^{\text{res}}$-modules and $U^{\text{res}}(\bo)$-modules respectively.

\begin{remark} The following observations will be useful in the next proof and also at several points later on. Suppose that $M$ is a $\B$-comodule, with coaction $\rho_M:M\to M\otimes_R \B$. Note that by the axioms of comodules, the composite
$$
(1\otimes \varepsilon)\circ \rho_M=1_M
$$
so that the map $\rho$ splits and $M$ is a direct summand of $M\otimes_R \B$ as an $R$-module. Moreover, the diagram
\[
\begin{tikzcd}
M \arrow{r}{\rho_M} \arrow[swap]{d}{\rho_M} & M\otimes_R\B \arrow{d}{\text{id}\otimes\Delta} \\
M\otimes_R\B\arrow{r}{\rho_M\otimes\text{id}} & M\otimes_R\B\otimes_R\B
\end{tikzcd}
\]
commutes. But note that the map $1\otimes\Delta$ makes $M\otimes_R\B$ into a $\B$-comodule, so that the above diagram and the splitting says that $M$ identifies via $\rho_M$ with a subcomodule of $M\otimes_R\B$ where the latter is given the comodule structure $1\otimes\Delta$. Of course all of the above applies more generally to a comodule over an arbitrary coalgebra.
\end{remark}

\begin{thm} The category of $\A$-comodules, respectively $\B$-comodules, is isomorphic to the category of integrable $U^{\text{res}}$-modules, respectively $U^{\text{res}}(\bo)$-modules.
\end{thm}

\begin{proof}
We first show that the above functors are fully faithful. This is the exact same argument as in Proposition \ref{fullyfaithful}, using Lemma \ref{comod} with $A=U^{\text{res}}$, $B=R$ and $C=\A$ for $\A$-comodules and with $A=U^{\text{res}}(\bo)$, $B=R$ and $C=\B$ for $\B$-comodules. For these to apply we need to show that $(U^{\text{res}})^*/\A$ and $U^{\text{res}}(\bo)^*/\B$ are torsion-free, but this is just the previous Lemma.

Next, the key fact we use is \cite[Theorem 1.31(iii)]{Andersen}: for any $R$-module $M$ the natural map $\theta_M:M\otimes_R \A\to \Hom_R(U^{\text{res}},M)$ is an isomorphism onto $H(M)$. Now suppose that $M$ is an integrable $U^{\text{res}}$-module. Then for all $m\in M$, the action map $\varphi_M(m):x\mapsto x\cdot m$ belongs to $H(M)$. So by the above facts the maps $\varphi_M(m)$ all belong to the image of $\theta_M$. By Lemma \ref{key_hopf_fact} with $C=\A$ we conclude that $M$ must be an $\A$-comodule. An analogous argument shows that integrable $U^{\text{res}}(\bo)$-modules are $\B$-comodules using \cite[Proposition 2.7(iv)]{Andersen}, which states that the natural map $\theta_M:M\otimes_R \B\to \Hom_R(U^{\text{res}}(\bo),M)$ is an isomorphism onto $\mathcal{H}(M)$.

Thus since the functors are fully faithful we are now reduced to showing that any $\A$-comodule (respectively $\B$-comodule) is integrable when viewed as a $U^{\text{res}}$-module (respectively $U^{\text{res}}(\bo)$-module). We prove it for $\B$, the proof for $\A$ being entirely analogous. Suppose $M$ is a $\B$-comodule. Then by the above remark the map $\rho:M\to M\otimes_R \B$ is an injective comodule homomorphism where the right hand side is given the coaction map $1\otimes \Delta$. In other words, in the language of $U^{\text{res}}(\bo)$-modules, this is saying the action on $M\otimes_R \B$ is the tensor product of the trivial action on $M$ with the usual action on $\B$, i.e for $u\in U^{\text{res}}(\bo)$ we have $u(m\otimes f)=m\otimes uf$ for all $m\in M$ and $f\in \B$. Thus, since $\B$ is integrable, so is $M\otimes_R \B$ with that structure. But now the result follows since integrable modules are closed under taking submodules by \cite[Note added in proof p.59]{Andersen}.
\end{proof}

\subsection{Some Noetherianity conditions}\label{Noetherian} We record here some conditions under which we can lift the Noetherian property from the reduction mod $\pi$ of a ring to the ring itself. These will be useful later in the paper.

\begin{prop} \begin{enumerate}
\item Suppose that $A$ is an $R$-algebra such that $A/\pi A$ is Noetherian. Then the $\pi$-adic completion $\widehat{A}$ is also Noetherian.
\item Let $n\geq1$ and suppose that we have $\Z^n$-graded $R$-algebra $\mathcal{R}=\oplus_{\mathbf{m}\in \Z^n} \mathcal{R}_{\mathbf{m}}$ such that each graded piece $\mathcal{R}_{\mathbf{m}}$ is finitely generated over $R$. If $\mathcal{R}/\pi\mathcal{R}$ is Noetherian, then $\mathcal{R}$ is graded Noetherian.
\end{enumerate}
\end{prop}

\begin{proof}
(i) is just \cite[Lemma 3.2.2]{Berthelot}. For (ii) we use the same argument as in \cite[Proposition II.2.3]{LuntsRos}. Specifically, consider the $\pi$-adic filtration on $\mathcal{R}$. The associated graded ring is a quotient of the polynomial algebra $(\mathcal{R}/\pi \mathcal{R})[t]$ (where $t$ corresponds to the symbol of $\pi$), and so is Noetherian. We will consider several graded $R$-submodules of $\mathcal{R}$, equipped with the subspace filtration of the $\pi$-adic filtration.

Suppose we are given two graded ideals $I\subset J$ with $I\neq J$. Then we have $\gr I\subset \gr J$ and it will suffice to show that $\gr I\neq \gr J$. Pick $\mathbf{m}\in \Z^n$ such that $I_{\mathbf{m}}\neq J_{\mathbf{m}}$, and assume that $\gr I_{\mathbf{m}}=\gr J_{\mathbf{m}}$. Since $I_{\mathbf{m}}$ and $J_{\mathbf{m}}$ are finitely generated over $R$, we will get a contradiction by Nakayama if we show that $J_{\mathbf{m}}=I_{\mathbf{m}}+\pi J_{\mathbf{m}}$.

By the Artin-Rees Lemma (\cite[Theorem 10.11]{AtiyahMac}) applied to $J_{\mathbf{m}}$ viewed as a submodule of $\mathcal{R}_{\mathbf{m}}$, the subspace filtration of the $\pi$-adic filtration on $\mathcal{R}_{\mathbf{m}}$ and the $\pi$-adic filtration on $J_{\mathbf{m}}$ have finite difference. So there exists a $d\in\Z_{<0}$ such that for all $j\in J_{\mathbf{m}}$ with degree $d(j)<d$ in the subspace filtration, $j\in\pi J_{\mathbf{m}}$. Now let $j\in J_{\mathbf{m}}$ be arbitrary. We show by induction on $d(j)$ that $j\in I_{\mathbf{m}}+\pi J_{\mathbf{m}}$, the cases $d(j)<d$ being already dealt with. Since $\gr I_{\mathbf{m}}=\gr J_{\mathbf{m}}$, there exists $i\in I_{\mathbf{m}}$ such that $d(i-j)<d(j)$. But by induction hypothesis this implies $i-j\in I_{\mathbf{m}}+\pi J_{\mathbf{m}}$, and hence we get $j=i-(i-j)\in I_{\mathbf{m}}+\pi J_{\mathbf{m}}$ as required.
\end{proof}

\begin{cor} The ring $\widehat{\A}$ is Noetherian.
\end{cor}

\begin{proof}
Since $q \equiv 1\pmod{\pi}$, the ring $\A/\pi \A$ coincides with the ring of regular functions on the group $G_k$ and hence is Noetherian. Therefore the result follows from part (i) of the Proposition.
\end{proof}

\section{The ad-finite part of $U_q$ and the Harish-Chandra isomorphism}

In this Section we investigate the finite part of the quantum group with respect to the adjoint action, its integral form and the reduction of it modulo $\pi$. We also investigate the Harish-Chandra isomorphism for $U$. Some of these results might be known, since reduction modulo $\pi$ amounts to specialising at $q=1$, but we couldn't find a suitable reference in the literature beyond \cite{adjoint} which studies the finite part. That paper is written using a different Hopf algebra presentation for $U_q$, and so we believe it to be worthwhile to repeat their results in some details using our presentation for the quantum group. Furthermore, they never work explicitly with integral forms and so we also explain how their results hold for them.

\subsection{The ad-finite part of $U_q$}\label{Joseph}

Recall that since it is a Hopf algebra, $U_q$ has a left adjoint action on itself given by $\ad(u)(v)=\sum u_1 v S(u_2)$. This action is not in general locally finite, so we define the \emph{finite part} of $U_q$ to be
$$
U_q^{\text{fin}}=\{u\in U_q : \dim_L\ad(U_q)(u)<\infty\}.
$$
$U_q^{\text{fin}}$ is then the largest integrable submodule of $U_q$ with respect to the adjoint action. It is a subalgebra of $U_q$ (see \cite[Corollary 2.3]{adjoint}). We now work towards seeing that it is in fact quite large.

The following computations and results are all due to Joseph and Letzter, see \cite[6.1-6.5]{adjoint}. We reproduce them using our presentation of $U_q$ for the convenience of the reader. For each $1\leq i\leq n$, let $U_i$ be the subalgebra of $U_q$ generated by $E_{\alpha_i}$, $F_{\alpha_i}$ and $K_{\alpha_i}$. Note that as a Hopf algebra, $U_i\cong U_{q_i}^{\text{ad}}(\mathfrak{sl}_2)$ is the adjoint form of the quantum group of $\mathfrak{sl}_2$.

\begin{remark}
The adjoint form $U_q^{\text{ad}}$ of $U_q$ is the algebra with the same generators and relations except we only take those $K_\lambda$ with $\lambda\in Q$ in the generators. So it is a very large Hopf subalgebra of $U_q$. Considering the adjoint action of $U_q^{\text{ad}}$ on $U_q$ rather than that of $U_q$ on itself changes nothing for our purposes. Indeed, the $K_\lambda$ with $\lambda\notin Q$ normalise $U_q^{\text{ad}}$ since they $q$-commute with all the generators. Hence, using the triangular decompositions for $U_q$ and $U_q^{\text{ad}}$, we see that $U_q$ is free of finite rank over $U_q^{\text{ad}}$, with basis given by the $K_\lambda$ where $\lambda$ ranges through a set of coset representatives for $Q$ in $P$. Thus we see that
$$
U_q^{\text{fin}}=\{u\in U_q : \dim_L\ad(U_q^{\text{ad}})(u)<\infty\}.
$$
and hence there is no harm in considering adjoint forms as we do above.
\end{remark}

Now, using the Hopf algebra structure of $U_q$, we see that the adjoint action is described as follows: for $u\in U_q$, we have
\begin{align*}
&\ad(K_\lambda)(u)=K_\lambda uK_\lambda^{-1}\\
&\ad(E_{\alpha_i})(u)=E_{\alpha_i}u-K_{\alpha_i}uK_{\alpha_i}^{-1}E_{\alpha_i}\\
&\ad(F_{\alpha_i})(u)=(F_{\alpha_i}u-uF_{\alpha_i})K_{\alpha_i}.
\end{align*}
Then we have:

\begin{lem}[{\cite[Lemma 6.1]{adjoint}}] Let $\lambda\in P$ and fix $1\leq i\leq n$. Then $\ad(E_{\alpha_i})$ (or $\ad(F_{\alpha_i})$) acts locally nilpotently on $K_\lambda$ if and only if $\langle \lambda, \alpha_i^\vee\rangle\leq 0$ and is a multiple of 2.
\end{lem}

\begin{proof} Using the above formulae, we see that for every $m\geq 0$, we have
$$
\ad(E_{\alpha_i})(E_{\alpha_i}^m K_\lambda)=(1-q_i^{2m+\langle \lambda, \alpha_i^\vee\rangle})E_{\alpha_i}^{m+1}K_\lambda,
$$
and
$$
\ad(F_{\alpha_i})(F_{\alpha_i}^m K_{\alpha_i}^m K_\lambda)=(1-q_i^{-2m-\langle \lambda, \alpha_i^\vee\rangle})F_{\alpha_i}^{m+1} K_{\alpha_i}^{m+1} K_\lambda.
$$
In both cases the right hand side vanishes if and only if $\langle \lambda, \alpha_i^\vee\rangle=-2m$.
\end{proof}

Local finiteness follows from this. Indeed, if $M$ is any $U_i$-module and $m\in M$ is a weight vector such that $E_{\alpha_i}$ and $F_{\alpha_i}$ act locally nilpotently on $m$, then $U_im$ is finite dimensional. This follows from the triangular decomposition $U_i\cong U_i^-\otimes_LU_i^0\otimes_LU_i^+$, where $U_i^-$, $U_i^0$ and $U_i^+$ are the $L$-algebras generated by $F_{\alpha_i}$, $K_{\alpha_i}$ and $E_{\alpha_i}$ respectively. Indeed, it suffices then to show that if $F_{\alpha_i}^nm'=0$ for all $n>>0$ and $m'\in M$ is a weight vector, then $F_{\alpha_i}^nE_{\alpha_i}m'=0$ for all $n>>0$ also. But this follows from the relation
$$
F_{\alpha_i}^nE_{\alpha_i}=E_{\alpha_i}F_{\alpha_i}^n-[n]_{q_i}F_{\alpha_i}^{n-1}\cdot \frac{K_{\alpha_i}q_i^{1-n}-K_{\alpha_i}^{-1}q_i^{n-1}}{q_i-q_i^{-1}}.
$$
Now if we let $U_q^{i,\text{fin}}=\{u\in U_q : \dim_L\ad(U_i)(u)<\infty\}$, then we have shown that $K_\lambda\in U_q^{i,\text{fin}}$ if and only if $\langle \lambda, \alpha_i^\vee\rangle=-2m$ for some $m\geq 0$. Next, we crucially have the following result:

\begin{prop}[{\cite[Proposition 6.5]{adjoint}}] $U_q^{\text{fin}}=\bigcap_{i=1}^n U_q^{i,\text{fin}}$.
\end{prop}

\begin{cor} Let $\lambda\in P$. Then $K_\lambda\in U_q^{\text{fin}}$ if and only if $\lambda\in -2P^+$. Moreover, for any $1\leq i\leq n$, $E_{\alpha_i}K_{-2\varpi_i}, F_{\alpha_i}K_{\alpha_i-2\varpi_i}\in U_q^{\text{fin}}$.
\end{cor}

\begin{proof}
That $K_\lambda\in U_q^{\text{fin}}$ if and only if $\lambda\in -2P^+$ follows immediately from the Proposition and the above discussion. But $U_q^{\text{fin}}$ is evidently stable under the adjoint action. So, since $\ad(E_{\alpha_i})(K_{-2\varpi_i})=(1-q_i^{-2})E_{\alpha_i}K_{-2\varpi_i}$ and $\ad(F_{\alpha_i})(K_{-2\varpi_i})=(1-q_i^2)F_{\alpha_i}K_{\alpha_i}K_{-2\varpi_i}$, the last part follows.
\end{proof}

Thus we see that $U_q$ is generated over $U_q^{\text{fin}}$ by all $K_\lambda$'s. However, more can be said. Let $S=\{K_{-2\lambda} : \lambda\in P^+\}$. By the above $S\subseteq U_q^{\text{fin}}$, and in fact it is an Ore set in it since it is normal: for $u\in U_q^{\text{fin}}$, we have
$$
K_{-2\lambda}u=K_{-2\lambda}uK_{-2\lambda}^{-1}K_{-2\lambda}=\ad(K_{-2\lambda})(u)K_{-2\lambda}\in U_q^{\text{fin}}K_{-2\lambda}
$$
since $U_q^{\text{fin}}$ is stable under the adjoint action (and similarly $uK_{-2\lambda}\in K_{-2\lambda}U_q^{\text{fin}}$). Thus we may form the localisation $S^{-1}U_q^{\text{fin}}$ which is of course contained in $U_q$, and which by the above contains $E_{\alpha_i}$ and $F_{\alpha_i}K_{\alpha_i}$ for every $1\leq i\leq n$. Finally, let $0=r_1,\ldots, r_m$ be a set of coset representatives for $P/2P$. Then we have:

\begin{thm}[{\cite[Theorem 6.4]{adjoint}}] The $L$-algebra $S^{-1}U_q^{\text{fin}}$ is generated by all $E_{\alpha_i}$, $F_{\alpha_i}K_{\alpha_i}$ and $K_{\pm 2\varpi_i}$ ($1\leq i\leq n$). Moreover $U_q$ is free over $S^{-1}U_q^{\text{fin}}$ with basis $1=K_{r_1},\ldots, K_{r_m}$.
\end{thm}

\subsection{Integral form of $U_q^{\text{fin}}$}\label{intfinite1} Now working with integral forms, the $R$-Hopf algebra $U^{\text{res}}$ acts on itself via the adjoint action and, moreover, this action preserves $U$ by \cite[Lemma 1.2]{Tanisaki2}. Hence we may define 
$$
U^{\text{fin}}=\{u\in U : \ad(U^{\text{res}})(u) \text{ is finitely generated over } R\}.
$$
This will be the integral form of $U_q^{\text{fin}}$. Note that since $R$ is Noetherian, $U^{\text{fin}}$ is stable under the adjoint action of $U^{\text{res}}$. Moreover, since the adjoint action makes $U$ into a $U^{\text{res}}$-module algebra, it follows that $U^{\text{fin}}$ is an $R$-subalgebra of $U$. This fact is also implied by the next result.

\begin{lem} We have
$$
U^{\text{fin}}=U_q^{\text{fin}}\cap U=\{u\in U : \ad(U)(u) \text{ is finitely generated over } R\}.
$$
In particular, for $\lambda\in P^+$ and $1\leq i\leq n$, $K_{-2\lambda}, E_{\alpha_i}K_{-2\varpi_i}, F_{\alpha_i}K_{\alpha_i-2\varpi_i}\in U^{\text{fin}}$.
\end{lem}

\begin{proof} We first show that $U^{\text{fin}}=U_q^{\text{fin}}\cap U$. Clearly the right hand side contains $U^{\text{fin}}$. Conversely, if $u\in U_q^{\text{fin}}\cap U$ then $\ad(U^{\text{res}})(u)\subset U$ is a lattice inside $\ad(U_q)(u)$ which by definition is finite dimensional over $L$. Hence this lattice is finitely generated over $R$ by \cite[Proposition 2.7]{Wadsley1} (this requires for $\ad(U^{\text{res}})(u)$ to be $\pi$-adically separated, which holds because it is contained inside $U^{\text{res}}$ which is itself free over $R$ and hence $\pi$-adically separated). A completely analogous argument shows that
$$
U_q^{\text{fin}}\cap U=\{u\in U : \ad(U)(u) \text{ is finitely generated over } R\}.
$$
The last part follows immediately from our earlier calculations.
\end{proof}

Thus we see that the set $S$ is contained in $U^{\text{fin}}$ and is Ore in $U^{\text{fin}}$ by the same argument as for $U_q^{\text{fin}}$. Note that it follows from the Lemma that $S^{-1}U^{\text{fin}}=S^{-1}U_q^{\text{fin}}\cap U$. Indeed, the left hand side is clearly contained in the right hand side, and moreover if $u\in S^{-1}U_q^{\text{fin}}\cap U$, then there is $s\in S$ such that $su\in U_q^{\text{fin}}$. On the other hand, $s\in U$ so $su\in U_q^{\text{fin}}\cap U=U^{\text{fin}}$. Thus $u\in S^{-1}U^{\text{fin}}$.

%\begin{remark} Note that the intersection between the $L$-algebra generated by the $K_{\pm 2\varpi_i}$, i.e. the group algebra $L(2P)$, and $U^{\text{fin}}$ is bigger than just $R(2P)$. Indeed, $U$ contains the elements $\frac{K_{\alpha_i}-K_{\alpha_i}^{-1}}{q_i-q_i^{-1}}$ for $1\leq i\leq n$. Fix such an $i$, and let $\beta_i$ be a coset representative for $\alpha_i$ in $P/2P$, so that $\alpha_i-\beta_i\in 2P$. Note that we also have $\alpha_i+\beta_i\in 2P$. Then $K_{\alpha_i+\beta_i}, K_{\alpha_i-\beta_i}\in S^{-1}U^{\text{fin}}$, and so we see that
%$$
%T_{\alpha_i}:=K_{\beta_i}\cdot \frac{K_{\alpha_i}-K_{\alpha_i}^{-1}}{q_i-q_i^{-1}}=\frac{K_{\alpha_i+\beta_i}-K_{\alpha_i-\beta_i}^{-1}}{q_i-q_i^{-1}}\in S^{-1}U_q^{\text{fin}}\cap U=S^{-1}U^{\text{fin}}.
%$$
%On the other hand $T_{\alpha_i}$ is clearly not in $R(2P)$. Note however that we have the relation
%\begin{align*}
%T_{\alpha_i}&=K_{\beta_i-\alpha_i}K_{\alpha_i}(E_{\alpha_i}F_{\alpha_i}-F_{\alpha_i}E_{\alpha_i})\\
%&=K_{\beta_i-\alpha_i}(E_{\alpha_i}\cdot F_{\alpha_i}K_{\alpha_i}-q_i^{-2}F_{\alpha_i}K_{\alpha_i}\cdot E_{\alpha_i})
%\end{align*}
%and so for each $1\leq i\leq n$, $T_{\alpha_i}$ belongs to the $R$-algebra generated by all $E_{\alpha_j}$, $F_{\alpha_j}K_{\alpha_j}$ and $K_{\pm 2\varpi_j}$ since $\beta_i-\alpha_i\in 2P$.
%\end{remark}

\begin{prop} The $R$-algebra $S^{-1}U^{\text{fin}}$ is generated by all $E_{\alpha_i}$, $F_{\alpha_i}K_{\alpha_i}$ and $K_{\pm 2\varpi_i}$ ($1\leq i\leq n$). Moreover $U$ is free over $S^{-1}U^{\text{fin}}$ with basis $1=K_{r_1},\ldots, K_{r_m}$.
\end{prop}

\begin{proof}
Let $U'$ be the $R$-algebra generated by all $E_{\alpha_i}$, $F_{\alpha_i}K_{\alpha_i}$ and $K_{\pm 2\varpi_i}$ ($1\leq i\leq n$). Then $U'\subseteq S^{-1}U^{\text{fin}}$. We claim that $U$ is free over $U'$ with basis $K_{r_1},\ldots, K_{r_m}$. By Theorem \ref{Joseph} this set is linearly independent over $U'$, so we just need to show that it spans $U$. We first observe that $K_{r_1},\ldots, K_{r_m}$ generate $U$ as a ring over $U'$. Indeed, any $K_{\lambda}$ is in the $R(2P)$-span of $K_{r_1},\ldots, K_{r_m}$. In particular all $K_{-\alpha_i}$ lie in it, and so we see that all the generators of $U$ may be obtained from the generators of $U'$ and $K_{r_1},\ldots, K_{r_m}$ as required.

Now using the fact that $K_{r_1},\ldots, K_{r_m}$ $q$-commute with the generators of $U'$, we see that any element of $u$ is a sum of terms of the form $u\cdot K_{r_{j_1}+\ldots+r_{j_s}}$ where $u\in U'$ and $1\leq j_1,\ldots, j_s\leq m$. But now, if $r_j$ is a coset representative for $r_{j_1}+\ldots+r_{j_s}$, then the above equals
$$
(u\cdot K_{(r_{j_1}+\ldots+r_{j_s})-r_j}) K_{r_j}
$$
where now $u\cdot K_{(r_{j_1}+\ldots+r_{j_s})-r_j}\in U'$. Thus we get that $K_{r_1},\ldots, K_{r_m}$ spans $U$ over $U'$ as required.

We can now see that $S^{-1}U^{\text{fin}}\subseteq U'$. Indeed, if $u\in S^{-1}U^{\text{fin}}$ then we may write
$$
u=\sum_{j=1}^m u_j K_{r_j}
$$
with $u_1, \ldots, u_m\in U'$. But $u\in S^{-1}U_q^{\text{fin}}$ so by Theorem \ref{Joseph} we must have $u_2=\ldots=u_m=0$, i.e. $u\in U'$.
\end{proof}

\subsection{Reduction modulo $\pi$}\label{intfinite2} We now use the previous section to investigate $\Ufbar:=U^{\text{fin}}/\pi U^{\text{fin}}$. We first recall a few things about $\Ubar$. Because $q\equiv 1\pmod{\pi}$ we have, for any $1\leq i\leq n$,
$$
K_{\alpha_i}-K_{\alpha_i}^{-1}=(q_i-q_i^{-1})[E_{\alpha_i},F_{\alpha_i}]\equiv 0\pmod{\pi}
$$
and so $K_{\alpha_i}^2=1$ in $\Ubar$. This implies that $K_\mu=1$ in $\Ubar$ for all $\mu\in 2Q$.

Now if $d$ denotes the index of $Q$ inside $P$, then for each $\lambda\in -2P^+$ we have $d\lambda\in -2Q^+$. Hence $K_\lambda^d\equiv 1\pmod{\pi}$. It follows that the image of $S^{-1}U^{\text{fin}}$ in $\Ubar$ is the same as the image of $U^{\text{fin}}$. Note that by Lemma \ref{intfinite1} the image of $U^{\text{fin}}$ in $\Ubar$ is isomorphic to $\Ufbar$. From now on we will always view $\Ufbar$ as a $k$-subalgebra of $\Ubar$. By Proposition \ref{intfinite1}, we now immediately deduce the following:

\begin{lem} $\Ufbar$ is generated as a $k$-algebra by all $E_{\alpha_i}$, $F_{\alpha_i}K_{\alpha_i}$ and $K_{2\varpi_i}$ ($1\leq i\leq n$).
\end{lem}

Recall that $U(\g_k)$ has the following presentation: it is the $k$-algebra with generators $e_i$, $f_i$ and $h_i$ ($1\leq i\leq n$) and relations:
\begin{align}
&[e_i, f_j]=\delta_{ij}h_i,\quad [h_i, h_j]=0\label{rel1}\\
&[h_i, e_j]=a_{ij}e_j, \quad [h_i, f_j]=-a_{ij}f_j\label{rel2}\\
& \ad(e_i)^{1-a_{ij}}(e_j)=\ad(f_i)^{1-a_{ij}}(f_j)=0\quad (i\neq j)\label{rel3}
\end{align}
where $a_{ij}:=\langle \alpha_j, \alpha_i^\vee\rangle$.

To simplify notation, let $H_i:=\frac{K_{\alpha_i}-K_{\alpha_i}^{-1}}{q_i-q_i^{-1}}$ ($1\leq i\leq n$). By Proposition \ref{Lusztig}(i), there is a surjective algebra homomorphism $\theta:\Ubar\to U(\g_k)$ given by $E_{\alpha_i}\mapsto e_i$, $F_{\alpha_i}\mapsto f_i$, $H_i\mapsto h_i$ and $K_\lambda\mapsto 1$. We now show that this map has a splitting.

\begin{prop} There is an injective algebra homomorphism $\phi: U(\g_k)\to \Ufbar$ given by $e_i\mapsto E_{\alpha_i}$, $f_i\mapsto F_{\alpha_i}K_{\alpha_i}$ and $h_i\mapsto H_iK_{\alpha_i}$.
\end{prop}

\begin{proof}
We first check that the map is well-defined, i.e. that the relations (\ref{rel1}), (\ref{rel2}) and (\ref{rel3}) hold in $\Ufbar$ after substituting $E_{\alpha_i}$, $F_{\alpha_i}K_{\alpha_i}$ and $H_iK_{\alpha_i}$ for $e_i$, $f_i$ and $h_i$ respectively.

The relations (\ref{rel1}) hold because $[E_{\alpha_i},F_{\alpha_j}]=\delta_{ij}H_i$ in $U$ and $K_{\alpha_i}$ is central modulo $\pi$, so that
$$
[E_{\alpha_i},F_{\alpha_j}K_{\alpha_j}]\equiv [E_{\alpha_i},F_{\alpha_j}] K_{\alpha_j}=\delta_{ij}H_iK_{\alpha_i} \pmod{\pi}.
$$
Also the fact that $H_i$ and $H_j$ commute is clear.

Next, in $U$ we have $K_{\alpha_i}E_{\alpha_j}=q_i^{a_{ij}}E_{\alpha_j}K_{\alpha_i}$ and $K_{\alpha_i}F_{\alpha_j}=q_i^{-a_{ij}}F_{\alpha_j}K_{\alpha_i}$. Hence we get
\begin{align*}
H_iE_{\alpha_j}-E_{\alpha_j}H_i&=H_iE_{\alpha_j}-\frac{K_{\alpha_i}q_i^{-a_{ij}}-K_{\alpha_i}^{-1}q_i^{a_{ij}}}{q_i-q_i^{-1}}E_{\alpha_j}\\
&=H_iE_{\alpha_j}-\left(q_i^{-a_{ij}}\frac{K_{\alpha_i}-K_{\alpha_i}^{-1}}{q_i-q_i^{-1}}-\frac{q_i^{a_{ij}}-q_i^{-a_{ij}}}{q_i-q_i^{-1}} K_{\alpha_i}^{-1} \right)\cdot E_{\alpha_j}\\
&=H_iE_{\alpha_j}-q_i^{-a_{ij}}H_iE_{\alpha_j}+[a_{ij}]_{q_i}K_{\alpha_i}^{-1}E_{\alpha_j}\\
&=(1-q_i^{-a_{ij}})H_iE_{\alpha_j}+[a_{ij}]_{q_i}K_{\alpha_i}^{-1}E_{\alpha_j}.
\end{align*}
A completely analogous calculation gives
$$
H_iF_{\alpha_j}-F_{\alpha_j}H_i=(1-q_i^{a_{ij}})H_iF_{\alpha_j}-[a_{ij}]_{q_i}K_{\alpha_i}^{-1}F_{\alpha_j}.
$$
Multiplying these two identities by $K_{\alpha_i}$ (and the second one by $K_{\alpha_j}$) and reducing modulo $\pi$, we obtain
$$
[H_iK_{\alpha_i}, E_{\alpha_j}]\equiv a_{ij} E_{\alpha_j}\quad\text{and}\quad [H_iK_{\alpha_i}, F_{\alpha_j}K_{\alpha_j}]\equiv -a_{ij} F_{\alpha_j}K_{\alpha_j}\pmod{\pi}
$$
by using the fact that $q\equiv 1\pmod{\pi}$ and that $K_{\alpha_i}$ is central in $\Ubar$. Hence the relations (\ref{rel2}) hold as required.

Finally, by \cite[Lemma 4.18]{Jantzen}, we have
\begin{equation}\label{rel4}
\ad(E_{\alpha_i})^{1-a_{ij}}(E_{\alpha_j})=0=\ad(F_{\alpha_i})^{1-a_{ij}}(F_{\alpha_j}K_{\alpha_j})
\end{equation}
in $U$. Now, for any $u\in U$, we have
$$
\ad(E_{\alpha_i})(u)=E_{\alpha_i}u-K_{\alpha_i}uK_{\alpha_i}^{-1}E_{\alpha_i}\equiv E_{\alpha_i}u-uE_{\alpha_i}\pmod{\pi}
$$
and
$$
\ad(F_{\alpha_i})(u)=(F_{\alpha_i}u-uF_{\alpha_i})K_{\alpha_i}\equiv (F_{\alpha_i}K_{\alpha_i})u-u(F_{\alpha_i}K_{\alpha_i})\pmod{\pi}
$$
since $K_{\alpha_i}$ is central modulo $\pi$. Reducing (\ref{rel4}) modulo $\pi$, we therefore obtain
$$
\ad'(E_{\alpha_i})^{1-a_{ij}}(E_{\alpha_j})\equiv 0\equiv \ad'(F_{\alpha_i}K_{\alpha_i})^{1-a_{ij}}(F_{\alpha_j}K_{\alpha_j})\pmod{\pi}
$$
where $\ad'(a)(b):=ab-ba$. Thus the relations (\ref{rel3}) hold as well.

Therefore the map $\phi$ is well-defined. But by considering the action on the generators, we see that the composite $\theta\circ\phi:U(\g_k)\to U(\g_k)$ is the identity. Hence $\phi$ is injective.
\end{proof}

Thus $U(\g_k)$ is a subalgebra of $\Ufbar$ and $\Ubar$, which are generated over it by the $K_\lambda$ with $\lambda$ in $2P$ and $P$ respectively. To finish establishing the structure of $\Ufbar$ and $\Ubar$, we need to investigate a corresponding algebra over $R$.

To this end, it will be useful to consider an alternative to $U_q^-$ in the quantum group. Let $U_q'^-$ be the $L$-subalgebra of $U_q$ generated by all $F_{\alpha_i}K_{\alpha_i}$ ($1\leq i\leq n$). By \cite[Corollary 4.5]{Nico2}, we in fact get that $F_\beta K_\beta\in U_q'^-$ for any positive root $\beta$. So from the PBW theorem for $U_q$ we see that $U_q'^-$ has an $L$-basis given by the monomials
\begin{equation}\label{eqn}
(F_{\beta_1}K_{\beta_1})^{a_1}\cdots (F_{\beta_N}K_{\beta_N})^{a_N}.
\end{equation}
Indeed, these are linearly independent, and the fact that they span follows from the fact that their span contains 1 and is preserved under left multiplication by any $F_{\alpha_i}K_{\alpha_i}$. That itself holds because if $u=(F_{\beta_1}K_{\beta_1})^{a_1}\cdots (F_{\beta_N}K_{\beta_N})^{a_N}$, then up to multiplying by a power of $q$, $u$ can be written as $F_{\beta_1}^{a_1}\cdots F_{\beta_N}^{a_N}K_\gamma$ where $\gamma=\sum a_i\beta_i$, and so we have
$$
F_{\alpha_i}K_{\alpha_i}u=q^b vK_{\gamma+\alpha_i}
$$
for some $v\in U_q^-$ of height $\gamma+\alpha_i$ and some $b\in\Z$. Hence this may be rearranged, using the PBW basis for $U_q^-$, into an element in the span of elements of the form (\ref{eqn}).

Thus we obtain $U_q\cong U_q'^-\otimes_L U_q^0\otimes_L U_q^+$ since the basis monomials
$$
F_{\beta_1}^{a_1}\cdots F_{\beta_N}^{a_N}K_\lambda E_{\beta_1}^{b_1}\cdots E_{\beta_N}^{b_N}
$$
can be rewritten as
$$
(F_{\beta_1}K_{\beta_1})^{a_1}\cdots (F_{\beta_N}K_{\beta_N})^{a_N}K_{\lambda-\gamma} E_{\beta_1}^{b_1}\cdots E_{\beta_N}^{b_N}
$$
with $\gamma$ as above.

We may now prove:

\begin{cor} The multiplication in $\Ubar$ gives rise to isomorphisms $U(\g_k)\otimes_k k(P/2Q)\cong \Ubar$ and $U(\g_k)\otimes_k k(2P/2Q)\cong \Ufbar$ respectively. In particular, $\Ufbar$ is Noetherian.
\end{cor}

\begin{proof}
By the previous Proposition and Lemma, $\Ubar$ and $\Ufbar$ are generated over $U(\g_k)$ by the $K's$ in $P$ and $2P$ respectively. Moreover, $K_\mu\equiv 1\pmod{\pi}$ for any $\mu\in 2Q$. Hence the maps are well defined and surjective. In particular the Noetherianity claim follows immediately by Hilbert's basis theorem.

So we just need to check injectivity. Let $0=r_1,\ldots, r_m$ be left coset representatives for $2Q$ in $P$. Let $A$ be the $R$-subalgebra of $U$ generated by $E_{\alpha_i}$, $F_{\alpha_i}K_{\alpha_i}$ ($1\leq i\leq n$) and $R(2Q)$. Note that the image of $A$ in $\Ubar$ is $U(\g_k)$ by the previous Proposition. Moreover, using the relations
\begin{align*}
E_{\alpha_j}(F_{\alpha_i}K_{\alpha_i})&=q_j^{-2}(F_{\alpha_i}K_{\alpha_i})E_{\alpha_j}+\delta_{ij} H_iK_{\alpha_i}\\
E_{\alpha_j}(H_iK_{\alpha_i})&=q_i^{-2a_{ij}}(H_iK_{\alpha_i})E_{\alpha_j}-q_i^{a_{ij}}[a_{ij}]_{q_i}E_{\alpha_j}\\
(H_iK_{\alpha_i})(F_{\alpha_j}K_{\alpha_j})&=q_i^{-2a_{ij}}(F_{\alpha_j}K_{\alpha_j})(H_iK_{\alpha_i})-[a_{ij}]_{q_i}(F_{\alpha_j}K_{\alpha_j})
\end{align*}
obtained from the proof of the previous Proposition, one easily obtains that $A$ is contained in $U_q'^-\otimes_L L(2Q)\otimes_L U_q^+$ since $H_iK_{\alpha_i}\in L(2Q)$. So by the above triangular decomposition $U_q\cong U_q'^-\otimes_L U_q^0\otimes_L U_q^+$, one gets that $K_{r_1},\ldots, K_{r_m}$ are linearly independent over $A$. But $U$ is generated over $A$ by $K_{r_1},\ldots, K_{r_m}$ as a ring, and for each $i, j$ there is an $l$ such that $K_{r_i}K_{r_j}\in A\cdot K_{r_l}$. Hence $U$ is a free $A$-module with basis $K_{r_1},\ldots, K_{r_m}$ since $AK_{r_i}=K_{r_i}A$ for all $i$. We then get that $U(\g_k)\otimes_k k(P/2Q)\cong \Ubar$ after reducing modulo $\pi$.

Now the result for $\Ufbar$ follows as well because the diagram
\[
\begin{tikzcd}
U(\g_k)\otimes_k k(2P/2Q)\arrow{r}\arrow{d} & \Ufbar\arrow{d}\\
U(\g_k)\otimes_k k(P/2Q)\arrow{r}{\cong} &\Ubar
\end{tikzcd}
\]
commutes and the vertical arrows are injections.
\end{proof}

\begin{remark} Of course we also obtain that $\Ubar$ is Noetherian, but we knew this to be true much more directly because $U$ is Noetherian. On the other hand, we do not know whether $U^{\text{fin}}$ is Noetherian. It is known that $U_q^{\text{fin}}$ is Noetherian (see \cite[Proposition 6.5]{joseph}), however it's unclear whether the techniques used to prove it can be adapted to integral forms.
\end{remark}

\subsection{The $G_k$-module structure on $\Ufbar$}\label{Gkmod} By \cite[Lemma 2.4]{adjoint}, the centre of any Hopf algebra is given as the algebra of invariants under the adjoint action, so that $Z(U)=\{u\in U : \ad(v)(u)=\varepsilon(v)u \text{ for all } v\in U\}$. Since $U$ and $U^{\text{res}}$ are lattices inside $U_q$, it follows from $Z(U)=Z(U_q)\cap U$ that in fact
$$
Z(U)=\{u\in U : \ad(v)(u)=\varepsilon(v)u \text{ for all } v\in U^{\text{res}}\}.
$$
Hence $Z(U)\subset U^{\text{fin}}$ and in fact $Z(U)=(U^{\text{fin}})^{\A}$ is the algebra of coinvariants for the $\A$-comodule structure on $U^{\text{fin}}$ coming from the adjoint action of $U^{\text{res}}$.

It will be useful to study the reduction modulo $\pi$ of the centre. To that end, we first look at the adjoint action of $U^{\text{res}}$ on $\Ufbar$. This gives rise to an $\A$-comodule structure, and so an $\A/\pi\A\cong \Of(G_k)$-comodule structure, i.e. a $G_k$-module structure. On the other hand, $G_k$ acts naturally on $\Ufbar\cong U(\g_k)\otimes_k k(2P/2Q)$ by tensoring the usual adjoint action on $U(\g_k)$ with the trivial action on $k(2P/2Q)$.

\begin{lem} These two $G_k$-actions on $\Ufbar$ coincide.
\end{lem}

\begin{proof}
By Proposition \ref{Lusztig}(ii), we just need to check that the corresponding $U_k^{\text{res}}$-actions coincide. Since in both cases the $K$'s act trivially, it suffices to check that the actions of $E_{\alpha_i}^{(m)}$ and $F_{\alpha_i}^{(m)}$ are the same in both cases for any $1\leq i\leq n$ and any $m\geq 0$.

Let $\lambda\in 2P$. By the proof of \cite[Lemma 1.2]{Tanisaki2}, we have:
\begin{align*}
\ad(E_{\alpha_i}^{(m)})(K_\lambda)&=\frac{(-1)^m q_i^{m(m-1)}}{[m]_{q_i}!}\left(\prod_{j=0}^{m-1} \left(q_i^{\langle\lambda, \alpha_i^\vee\rangle}-q_i^{-2j}\right)\right)E_{\alpha_i}^m K_\lambda\\
\ad(E_{\alpha_i}^{(m)})(E_{\alpha_i})&=q_i^{-m(m+1)/2}(q_i-q_i^{-1})^m E_{\alpha_i}^{m+1}\\
\ad(E_{\alpha_i}^{(m)})(E_{\alpha_j})&=
\begin{cases}
\sum_{r=0}^m (-1)^r q_i^{r(n-1+a_{ij})}E_{\alpha_i}^{(m-r)}E_{\alpha_j}E_{\alpha_i}^{(r)} & \text{if $n<1-a_{ij}$}\\
0 & \text{otherwise}
\end{cases}\\
\ad(E_{\alpha_i}^{(m)})(F_{\alpha_i}K_{\alpha_i})&=
\begin{cases}
H_iK_{\alpha_i} & \text{if $m=1$}\\
(-1)^{m-1}q_i^{(m-1)(m+2)/2}(q_i-q_i)^{m-2}E_{\alpha_i}^{m-1}K_{\alpha_i}^2 & \text{if $m\geq 2$}
\end{cases}\\
\ad(E_{\alpha_i}^{(m)})(F_{\alpha_j}K_{\alpha_j})&=0
\end{align*}
where $i \neq j$. Recall that $E_{\alpha_i}^{(r)}\in U$ whenever $r\leq -a_{ij}$ by our assumptions on the residue characteristic $p$. So the right hand side of these are all clearly in $U$, possibly with the exception of the first equation. For this one, one just need to verify that
$$
\frac{1}{[m]_{q_i}!}\left(\prod_{j=0}^{m-1} \left(q_i^{\langle\lambda, \alpha_i^\vee\rangle}-q_i^{-2j}\right)\right)\in R.
$$
Say $\langle \lambda, \alpha_i^\vee\rangle=2l$ for some $l\in \Z$. Then the above becomes
$$
q_i^{ml-m(m-1)/2}(q_i-q_i^{-1})^m\frac{[l]_{q_i}\cdot [l+1]_{q_i}\cdot \ldots \cdot [l+m-1]_{q_i}}{[m]_{q_i}!}
$$
which is in $R$ as required. Reducing all the above formulae modulo $\pi$, this then gives in $\Ufbar$:
\begin{align*}
\ad(E_{\alpha_i}^{(m)})(K_\lambda)&=0\\
\ad(E_{\alpha_i}^{(m)})(E_{\alpha_i})&=0\\
\ad(E_{\alpha_i}^{(m)})(E_{\alpha_j})&=
\begin{cases}
\sum_{r=0}^m (-1)^r E_{\alpha_i}^{(m-r)}E_{\alpha_j}E_{\alpha_i}^{(r)} & \text{if $n<1-a_{ij}$}\\
0 & \text{otherwise}
\end{cases}\\
\ad(E_{\alpha_i}^{(m)})(F_{\alpha_i}K_{\alpha_i})&=
\begin{cases}
H_iK_{\alpha_i} & \text{if $m=1$}\\
-E_{\alpha_i} & \text{if $m= 2$}\\
0 & \text{if $m\geq 3$}
\end{cases}\\
\ad(E_{\alpha_i}^{(m)})(F_{\alpha_j}K_{\alpha_j})&=0
\end{align*}
which agrees with the usual action of the hyperalgebra of $G_k$ as required. The calculations for the $F_{\alpha_i}^{(m)}$'s are similar and omitted.
\end{proof}

\begin{remark} It is also true that the $G_k$-actions agree on the whole of $U_k$ with essentially the same calculations, but we won't need it.
\end{remark}

\begin{cor} There is a canonical isomorphism $(\Ufbar)^{G_k}\cong U(\g_k)^{G_k}\otimes_k k(2P/2Q)$, where $U(\g_k)^{G_k}$ is the Harish-Chandra centre of $U(\g_k)$.
\end{cor}

Note that this implies that the image of $Z(U)$ in $U(\g_k)$ is contained in the Harish-Chandra centre $U(\g_k)^{G_k}$.

\subsection{Weyl group invariants}\label{Weyl} Recall that the Weyl group acts on $U_q^0=LP$ by $w\cdot K_\lambda=K_{w\lambda}$ for any $w\in W$ and $\lambda\in P$, where the action of $W$ on $P$ is the usual one (not the dot action). This action can be checked to preserve $U^0$ (see the main calculation in the proof of the Proposition below). Because the centre $Z(U_q)$ is isomorphic to $L(2P)^W$ by the Harish-Chandra homomorphism (see the next section for details), it will be useful to consider the action of $W$ on $U^0\cap L(2P)$ as well.

As in the previous section, we will need to consider what happens when we reduce modulo $\pi$. From Proposition \ref{intfinite2}, $S(\h_k)$ can be identified with the $k$-subalgebra of $U_k$ generated by $H_iK_{\alpha_i}$ ($1\leq i\leq n$).

\begin{lem} The multiplication in $U_k^0$ gives rise to isomorphisms
$$
U^0_k\cong S(\h_k)\otimes_k k(P/2Q) \quad \text{and}\quad (U^0\cap L(2P))_k\cong S(\h_k)\otimes_k k(2P/2Q).
$$
\end{lem}

\begin{proof}
Let $0=r_1,\ldots, r_m$ be left coset representatives for $2Q$ in $P$ such that, for some $m'<m$, $r_1,\ldots, r_{m'}$ are left coset representatives for $2Q$ in $2P$, and let $B$ be the $R$-subalgebra of $U^0$ generated by $R(2Q)$ and the $H_iK_{\alpha_i}$ ($1\leq i\leq n$). Note that $B$ is contained in $U^0\cap L(2P)$ and that the image of $B$ in $U^0_k$ is $S(\h_k)$. Moreover, $K_{r_1}, \ldots, K_{r_m}$ generate $U^0$ as a $B$-algebra and are linearly independent over $B$. As in the proof of Corollary \ref{intfinite2}, we have that $U^0$ is in fact free over $B$ with basis $K_{r_1}, \ldots, K_{r_m}$. Reducing modulo $\pi$ gives the first isomorphism.

Let $C=BK_{r_1}\oplus\ldots\oplus BK_{r_{m'}}\subseteq U^0\cap L(2P)$. By definition of $B$, $C$ is actually an $R$-subalgebra of $U^0$. Furthermore, by the above basis for $U^0$ over $B$, we see that $U^0$ is free over $C$ with basis given by a set of coset representatives for $2P$ in $P$. But these are linearly independent over $U^0\cap L(2P)$, hence it follows that $U^0$ is also free over $U^0\cap L(2P)$ with the same basis and thus that $U^0\cap L(2P)=C$. Reducing modulo $\pi$ now gives the second isomorphism.
\end{proof}

As in the previous section, there are two natural actions of the Weyl group on $(U^0\cap L(2P))_k\cong S(\h_k)\otimes_k k(2P/2Q)$. Namely one can tensor the usual Weyl group action on $S(\h_k)$ with the trivial action on $k(2P/2Q)$, or consider the above Weyl group action on $U^0\cap L(2P)$ and reduce modulo $\pi$.

\begin{prop} These two actions of $W$ coincide.
\end{prop}

\begin{proof}
We first show that the Weyl group action on $R(2P/2Q)$ is trivial. Let $\mu\in 2P$ and $w\in W$. We show by induction on $\ell(w)$ that there exists $\lambda\in 2Q$ such that $K_{w\mu}=K_\mu K_\lambda$. This is trivial if $\ell(w)=0$ and the induction step reduces to $w=s_\alpha$ for some simple root $\alpha$. Then we have $\langle \mu, \alpha^\vee\rangle=2m$ for some $m\in\Z$ and so
$$
K_{s_\alpha(\mu)}=K_{\mu-2m\alpha}=K_\mu K_{-2m\alpha}
$$
as required.

Now since $(U^0\cap L(2P))_k\cong S(\h_k)\otimes_k k(2P/2Q)$ by the Lemma, since the image of $S(\h_k)$ in $U^0_k$ is generated by the reduction modulo $\pi$ of the $H_iK_{\alpha_i}$, and since $W$ is generated by simple reflections, it suffices to check that the reduction modulo $\pi$ of the action of $s_{\alpha_i}$ on $H_jK_{\alpha_j}$ agrees with the classical action on $h_j\in S(\h_k)$, for any $1\leq i,j\leq n$. Classically, we have
$$
s_{\alpha_i}(h_j)=h_j-\langle \alpha_i, \alpha_j^\vee\rangle h_i=h_j-a_{ji}h_i.
$$
Meanwhile, in $U^0$, $s_{\alpha_i}$ sends $H_jK_{\alpha_j}=\frac{K_{2\alpha_j}-1}{q_j-q_j^{-1}}$ to
\begin{align*}
\frac{K_{2(\alpha_j-a_{ij}\alpha_i)}}{q_j-q_j^{-1}}&=\frac{K_{2\alpha_j}-1}{q_j-q_j^{-1}}K_{-2a_{ij}\alpha_i}+\frac{K_{-2a_{ij}\alpha_i}-1}{q_j-q_j^{-1}}\\
&=H_jK_{\alpha_j}K_{2\alpha_i}^{-a_{ij}}+\frac{K_{2\alpha_i}^{-a_{ij}}-1}{q_j-q_j^{-1}}
\end{align*}
From now on assume $a_{ij}\neq 0$ else the result is now clear. Let $l=-a_{ij}>0$. Then the above equals
\begin{align*}
&H_jK_{\alpha_j}K_{2\alpha_i}^{l}+\frac{K_{2\alpha_i}-1}{q_j-q_j^{-1}}(1+K_{2\alpha_i}+\ldots+K_{2\alpha_i}^{l-1})\\
=&H_jK_{\alpha_j}K_{2\alpha_i}^{l}+\frac{q_i-q_i^{-1}}{q_j-q_j^{-1}}H_iK_{\alpha_i}(1+K_{2\alpha_i}+\ldots+K_{2\alpha_i}^{l-1}).
\end{align*}
Note that $q_i=q_j^{\frac{\langle \alpha_i, \alpha_j^\vee\rangle}{\langle \alpha_j, \alpha_i^\vee\rangle}}=q_j^{\frac{a_{ji}}{a_{ij}}}$, so that
$$
\frac{q_i-q_i^{-1}}{q_j-q_j^{-1}}=\frac{q_j^{\frac{a_{ji}}{a_{ij}}}-q_j^{-\frac{a_{ji}}{a_{ij}}}}{q_j-q_j^{-1}}\in R
$$
and reduces to $\frac{a_{ji}}{a_{ij}}$ modulo $\pi$ (recalling that $p$ doesn't divide any non-zero entries of the Cartan matrix). Using the fact that $K_{2\alpha_i}\equiv 1\pmod{\pi}$, we get that the above thus reduces to
$$
H_jK_{\alpha_j}-a_{ji}H_iK_{\alpha_i}
$$
as required.
\end{proof}

From the Lemma and the Proposition, we immediately get:

\begin{cor} There is a canonical isomorphism $(U^0\cap L(2P))_k)^W\cong S(\h_k)^W\otimes_k k(2P/2Q)$.
\end{cor}

\subsection{The Harish-Chandra isomorphism}\label{HC} Recall from \cite[Theorem 6.25 \& 6.26]{Jantzen} that there is an isomorphism $\varphi:Z(U_q)\to L(2P)^W$ which is defined as follows. First, from the triangular decomposition, there is a projection map $p:U_q\to U_q^0$ whose restriction to $Z(U_q)$ is an algebra homomorphism. There is also a an algebra automorphism $\gamma_{-\rho}$ of $U_q^0$ given by $\gamma_{-\rho}(K_\lambda)=q^{-\langle \lambda, \rho\rangle}K_\lambda$. Then the composite $\varphi=\gamma_{-\rho}\circ p$ gives an isomorphism between $Z(U_q)$ and $L(2P)^W$.

By the triangular decomposition for $U$ the restriction of $p$ to $U$ gives a projection $U\to U^0$. Moreover, the automorphism $\gamma_{-\rho}$ preserves $U^0$ since we have
\begin{equation}\label{neweqn2}
\gamma_{-\rho}(H_i)=\frac{q_i^{-1}K_{\alpha_i}-q_iK_{\alpha_i}^{-1}}{q_i-q_i^{-1}}=q_i^{-1}H_i-K_{\alpha_i}^{-1}
\end{equation}
for any $1\leq i\leq n$. Hence $\varphi$ gives an injective $R$-algebra homomorphism $\varphi:Z(U)\to (U^0\cap L(2P))^W$.

The projection $p:U\to U^0$ induces modulo $\pi$ the map
$$
U_k\cong U(\g_k)\otimes_k k(P/2Q)\to S(\h_k)\otimes_k k(P/2Q)\cong U^0_k
$$
obtained by tensoring the usual projection $U(\g_k)\to S(\h_k)$ with the identity map. Moreover, by definition, $\gamma_{-\rho}$ reduces to the identity map on $k(P/2Q)$ and from (\ref{neweqn2}) we have
$$
\gamma_{-\rho}(H_iK_{\alpha_i})\equiv H_iK_{\alpha_i}-1\pmod{\pi}.
$$
Therefore $\gamma_{-\rho}$ restricts to the usual shift by $-\rho$ on $S(\h_k)$.

\begin{thm} There is a commutative square
\[
\begin{tikzcd}
Z(U)\arrow{r}{\varphi} \arrow[swap]{d}{q_1} & (U^0\cap L(2P))^W \arrow{d}{q_2}\\
U(\g_k)^{G_k}\otimes_k k(2P/2Q)\arrow{r}{\psi\otimes\text{id}} & S(\h_k)^W\otimes_k k(2P/2Q)
\end{tikzcd}
\]
where the vertical arrows are given by reduction modulo $\pi$, $\varphi$ is the Harish-Chandra homomorphism for $U$, and $\psi$ is the Harish-Chandra homomorphism for $U(\g_k)$. Furthermore, if $p=\text{char}(k)$ does not divide $|W|$, then $q_1$ and $q_2$ are both surjective, and $\varphi$ is an isomorphism.
\end{thm}

\begin{proof}
By Proposition \ref{Weyl} and Corollary \ref{Weyl}, the map $U^0\cap L(2P)\to (U^0\cap L(2P))_k$ is $W$-equivariant and thus induces a map $q_2:(U^0\cap L(2P))^W\to S(\h_k)^W\otimes_k k(2P/2Q)$. Moreover, we already saw that $Z(U)= (U^{\text{fin}})^{\A}$ gets sent in $U_k$ to $(\Ufbar)^{G_k}\cong U(\g_k)^{G_k}\otimes_k k(2P/2Q)$ by Corollary \ref{Gkmod}. From the above discussion and by definition of $\psi$, we therefore see that the diagram does commute.

Now under the assumption on $p$, it's immediate from Lemma \ref{Weyl} that $q_2$ is surjective. Indeed, if $v\in S(\h_k)^W\otimes_k k(2P/2Q)$ and $u\in U^0\cap L(2P)$ such that $u\equiv v\pmod{\pi}$, then
$$
v=q_2\left(\frac{1}{|W\cdot u|}\sum_{x\in W\cdot u} x\right)
$$
as required.

Now $\varphi$ extends to an isomorphism $Z(U_q)\to L(2P)^W$, so it is injective and its cokernel is $\pi$-torsion. Pick $u\in U^0\cap L(2P)$ and let $m\geq 0$ be minimal such that $\pi^m u\in\im(\varphi)$. Assume for contradiction that $m>0$, and write $\pi^m u=\varphi(z)$ where $z\in Z(U)$. By minimality of $m$, note that $z\notin \pi Z(U)$ since $U^0\cap L(2P)$ has no $\pi$-torsion. Since $q_2(\pi^m u)=0$, we get that $\psi\otimes\text{id}(q_1(z))=0$ by commutativity of the diagram. But since $\psi\otimes\text{id}$ is an isomorphism under our assumption on $p$ by \cite[Theorem 9.3]{Jantzen3}, it follows that $q_1(z)=0$, i.e. that $z\in \pi Z(U)$, a contradiction.

Hence we have shown that $\varphi$ is an isomorphism, and moreover it now follows that $q_1=(\psi\otimes\text{id})^{-1}\circ q_2\circ \varphi$ is surjective since it is a composite of surjections.
\end{proof}

\begin{remark} Our argument was inspired by \cite[section 9.6]{Jantzen3}.
\end{remark}

\subsection{The condition on $p$}\label{condonp} The condition on the prime $p$ we imposed in Theorem \ref{HC} will appear again later in the paper, and is a fairly reasonable condition. It is in particular implied if $p>h$ where $h$ is the Coxeter number of the root system. Indeed, for an irreducible root system, the order of the Weyl group and the Coxeter number are given in \cite[2.11, Table 2 \& 3.18, Table 2]{Humphreys2}, which we display below along with the condition that $p$ must satisfy in order to not divide $|W|$:

\begin{table}[h]
%\centering
\hskip-1.5cm
\begin{tabular}{c | c | c | c | c | c | c | c | c}
Type & $A_n$ & $B_n/C_n$ & $D_n$ & $G_2$ & $F_4$ & $E_6$ & $E_7$ & $E_8$\\
\hline
$|W|$ & $(n+1)!$ & $2^n n!$ & $2^{n-1} n!$ & 12 & $2^7\cdot 3^2$ & $2^7\cdot 3^4\cdot 5$ & $2^{10}\cdot 3^4\cdot 5\cdot 7$ & $2^{14}\cdot 3^5\cdot 5^2\cdot 7$\\
$p>$ & $n+1$ & $n$ & $n$ & 3 & 3 & 5 & 7 & 7\\
$h$ & $n+1$ & $2n$ & $2(n-1)$ & 6 & 12 & 12 & 18 & 30
\end{tabular}
\end{table}

Thus the condition that $p$ does not divide $|W|$ is indeed implied by $p>h$, and moreover it itself implies that $p$ is a very good prime (see section \ref{cohomology}). Since we assumed our group to be simply-connected, this implies that the Harish-Chandra homomorphism $\psi: U(\g_k)^{G_k}\to S(\h_k)^W$ is an isomorphism for such a $p$ by \cite[section 6.4 \& Theorem 9.3]{Jantzen3}.

\section{The quantum flag variety and its integral form}

In this Section we review definitions and results from \cite{QDmod1} and then adapt them to integral forms. 

\subsection{The category $\flag$}\label{recap1} We first begin by recalling the definition of the quantum flag variety.

\begin{definition} (\cite[Definition 3.1]{QDmod1}) A \emph{$B_q$-equivariant sheaf on $G_q$} is a triple $(F, \alpha, \beta)$ where $F$ is an $L$-vector space, $\alpha:\Oq\otimes F\to F$ is a left $\Oq$-module action and $\beta: F\to F\otimes \OqB$ is a right $\OqB$-comodule action, such that $\alpha$ is an $\OqB$-comodule homomorphism where $\Of_q\otimes F$ is given the tensor $\OqB$-comodule structure. We denote by $\flag$ the category of $B_q$-equivariant sheaves on $G_q$.
\end{definition}

\begin{remark} In the classical case $q=1$, this category is equivalent to the category of $B$-equivariant sheaves of $\Of_G$-modules, which in turn is equivalent to the category of quasi-coherent sheaves of $\Of_{G/B}$-modules. So the category $\flag$ can be thought of as the quantum analogue of the flag variety.
\end{remark}

Obviously $\Oq$ is an object of this category. More generally we have a notion of line bundles. Any element $\lambda\in T_P$ may be thought of as a character of the group algebra $LP\cong L[K_\mu : \mu\in P]$, and we may extend it to a character of $U_q^{\geq 0}$ by setting it to kill the $E$'s. This defines a one dimensional $U_q^{\geq 0}$-module $L_\lambda$. The ones among these which are integrable, and so $\OqB$-comodules, correspond to $\lambda\in P$, and the coaction is $1\mapsto 1\otimes \lambda$.

\begin{definition}(\cite[Definition 3.3]{QDmod1}) 
We define a line bundle in $\flag$ to be an object of the form $\Oq(\lambda):=\Oq\otimes_L L_{-\lambda}$ for $\lambda\in P$, where the $\Oq$-action is on the left factor and the $\OqB$-coaction is the tensor one (or in the modules language this means we give it the tensor product $U_q^{\geq 0}$-module structure). More generally for a finite dimensional $\OqB$-comodule $V$ we get that $\Oq\otimes_L V$, with an analogous structure as above, is an element of $\flag$ and we may think of it as a vector bundle.
\end{definition}

Now that we have a flag variety, we turn to the notion of taking global sections.

\begin{definition} (\cite[Definition 3.4]{QDmod1}) The \emph{global section functor} $\Gamma:\flag\to L\text{-mod}$ is defined to be
$$
\Gamma(M):=\Hom_{\flag}(\Oq, M)=\{ m\in M : \beta(m)=m\otimes 1\}=: M^{B_q},
$$
which we call the \emph{$B_q$-invariants} of $M$.
\end{definition}

By \cite[Lemma 3.8]{QDmod1}, the category $\flag$ has enough injectives, and so we can right derive the global section functor. It was shown in \cite[Section 3]{QDmod1} that the category $\flag$ is equivalent to a Proj category in the sense of Artin-Zhang \cite{ArtinZhang}. That includes \cite[Proposition 3.5]{QDmod1} which states that the line bundles are very ample in the sense that for any coherent module $M$, the twist $M(\lambda)$ is $\Gamma$-acyclic and generated by its global sections for $\lambda >>0$.

\subsection{Quantum $D$-modules}\label{recap2} Let $H$ be a Hopf algebra over a commutative ring $S$, and let $A$ be an $S$-algebra equipped with a left $H$-module structure. We say that $A$ is an \emph{$H$-module algebra} if for all $u\in H$ and all $a, b\in A$, $u(ab)=\sum u_1(a)u_2(b)$. In that case, we may form the \emph{smash product algebra} $A\# H$. As an $S$-module, this is just $A\otimes_S H$, but with multiplication given by
$$ 
(a\otimes u)\cdot(b\otimes v)=\sum au_1(b)\otimes u_2v.
$$
From now on we drop the tensor signs, and write $au$ for $a\otimes u$. Note that the action $u(a)$ of $u\in H$ on $a\in A$ coincides with the adjoint action $\sum u_1aS(u_2)$ in $A\# H$.

Now, recall that there is a left $U_q$-module algebra structure on $\Oq$ given by
$$
u(a)=\sum a_2(u)\cdot a_1,
$$
for $u\in U_q$ and $a\in\Oq$. By viewing $\Oq\subseteq U_q^*$, this action amounts to the action $u(a)(x)=a(xu)$ for $a\in\Oq$ and $u,x\in U_q$. Following \cite[Definition 4.1]{QDmod1}, we define the ring of quantum differential operators on $G_q$ to be the smash product algebra $\D_q=\Oq\# U_q$.  We will need the following result which was not proved in \cite{QDmod1}:

\begin{prop} The ring $\D_q$ is Noetherian.
\end{prop}

\begin{proof} Since $\D_q=\Oq\otimes_L U_q$ as a vector space, and since $x\cdot(yu)=(xy)u$ for all $x, y\in \Oq$ and all $u\in U_q$, it follows that $\D_q$ is generated as an $\Oq$-module by $U_q$. Recall our PBW filtration on $U_q$. We now define an analogous filtration on $\D_q$ given by
$$
F_i\D_q=\Oq\cdot F_iU_q.
$$
We claim this defines an algebra filtration. Indeed, suppose that for some $i, j\geq 0$, we are given $u\in F_iU_q$ and $v\in F_jU_q$, and take $x,y\in \Oq$. By definition of the Hopf algebra structure on $U_q$, we have that $\Delta(u)\in F_i(U_q\otimes_L U_q)\subset F_iU_q\otimes_L F_iU_q$ where we give $U_q\otimes_L U_q$ the tensor filtration. Therefore, it follows that
$$
(xu)(yv)=\sum \big(xu_1(y)\big)\big(u_2v\big)\in F_{i+j}\D_q
$$
since the filtration on $U_q$ is an algebra filtration. Hence $\D_q$ is a positively filtered $L$-algebra.

It will therefore be enough to show that $\gr \D_q$ is Noetherian. First, observe that $F_0\D_q$ is generated over $\Oq$ by the $K_\mu$ for $\mu\in P$, which all commute. Moreover, for each generator $x_i$ of $\Oq$, we have that
$$
K_\mu x_i K_{-\mu}=K_\mu(x_i)\in q^{\frac{1}{d}\Z}x_i
$$
by definition of the $U_q$-action on $\Oq$ and since the $x_i$'s are matrix coefficients with respect to weight bases. Thus we see that the generators of $F_0U_q$ normalise $\Oq$. Hence it follows from Lemma \ref{PrelimonUq} that $F_0\D_q$ is Noetherian since $\Oq$ is Noetherian.

Next, we claim that the symbols $\overline{E_{\alpha_i}}$ and $\overline{F_{\alpha_j}}$ normalise $F_0 \D_q$ in $\gr\D_q$ for all $i,j$. Indeed, we have that they $q$-commute with the $K$'s and for $x\in\Oq$, we have
$$
E_{\alpha_i}x-(K_{\alpha_i}xK_{-\alpha_i})E_{\alpha_i}=E_{\alpha_i}(x)\in \Oq\subseteq F_0 \D_q,
$$
where $K_{\alpha_i}xK_{-\alpha_i}\in \Oq$ by the above. Thus in $\gr\D_q$ we have
$$
\overline{E_{\alpha_i}}x=(K_{\alpha_i}xK_{-\alpha_i})\overline{E_{\alpha_i}}\in F_0 \D_q\cdot \overline{E_{\alpha_i}}
$$
Similarly for the $F$'s.

Finally we give to $\gr \D_q$ an analogue of the $\Z^{2N}_{\geq 0}$-filtration on $\gr U_q$ from section \ref{PrelimonUq}. More precisely, we make $\gr\D_q$ into a $\Z^{2N}_{\geq 0}$-filtered $F_0\D_q$-algebra. First we impose the reverse lexicographic total ordering on $\Z^{2N}_{\geq 0}$, and give a $\Z^{2N}_{\geq 0}$-filtration on $\gr \D_q$ by stating that a monomial
$$
F_{\beta_1}^{r_1}\cdots F_{\beta_N}^{r_N}K_\lambda E_{\beta_1}^{s_1}\cdots E_{\beta_N}^{s_N}
$$
has degree $(r_1,\ldots, r_N,s_1,\ldots,s_N)$. Then it follows that the corresponding associated multigraded algebra is a $q$-commutative $F_0\D$-algebra. Hence the associated graded algebra of $\gr\D_q$ is Noetherian by Lemma \ref{PrelimonUq}, and so it must be that $\gr\D_q$ is Noetherian.
\end{proof}

Note that $\D_q$ is a $U_q$-module algebra via the adjoint action in $\D_q$, or alternatively by tensoring the above action on $\Oq$ with the adjoint action on $U_q$. Explicitly,
\begin{equation}\label{actiononDq}
u\cdot(a\otimes v)=\sum u_1.a\otimes u_2 v S(u_3).
\end{equation}
We now are ready to define $D$-modules on the quantum flag variety:

\begin{definition} (\cite[Definition 4.2]{QDmod1}) Let $\lambda\in T_P$. A \emph{$(B_q, \lambda)$-equivariant $\D_q$-module} is a triple $(M, \alpha, \beta)$ where $M$ is an $L$-vector space, $\alpha:\D_q\otimes M\to M$ is a left $\D_q$-module action and $\beta: M\to M\otimes \OqB$ is a right $\OqB$-comodule action. The map $\beta$ induces a left $U_q^{\geq 0}$-action on $M$ which we also denote by $\beta$. These actions must satisfy:
\begin{itemize}
\item[(i)] The $U_q^{\geq 0}$-actions on $M\otimes L_\lambda$ given by $\beta\otimes \lambda$ and $\alpha\vert_{U_q^{\geq 0}}\otimes 1$ are equal.
\item[(ii)] The map $\alpha$ is $U_q^{\geq 0}$-linear with respect to the $\beta$-action on $M$ and the action (\ref{actiononDq}) on $\D_q$.
\end{itemize}
In other words $M$ is an object of $\flag$ equipped with a $U_q^{\geq 0}$-equivariant $\D_q$-action with in addition the condition (i).
\end{definition}

We denote by $\Dflag$ the category of such $\D_q$-modules. We have a forgetful functor $\Dflag\to\flag$, which allows us to define a global section functor on $\Dflag$ given by $\Gamma\circ\text{forget}$. We also denote this functor by $\Gamma$.

Note that condition (i) above can be rephrased into saying that for $M\in\Dflag$ and $m\in M$, we have $E_\alpha m=\beta(E_\alpha)m$ and $K_\mu m=\lambda(\mu) \beta(K_\mu)m$ for all simple roots $\alpha$ and $\mu\in P$. In particular if $m$ is a global section then by $B_q$-invariance we must have $E_\alpha m=0$ and $K_\mu m=\lambda(\mu)m$. In other words global sections consist of the highest weight vectors of weight $\lambda$. So we see that the $\D_q$-module homomorphisms $\D_q\to M$ corresponding to global sections factor through the ideal $\D_qI$ where $I=\{E_{\alpha_i}, K_\mu-\lambda(K_\mu) : 1\leq i\leq n, \mu\in P\}$.

Based on the above, we define $\D_q^\lambda$ to be the quotient
$$
\D_q^\lambda=\D_q/\D_qI
$$
where $I$ is as above. We can see that $\D_q^\lambda=\Oq\otimes_L M_\lambda$ where $M_\lambda$ is the Verma module of highest weight $\lambda$. We saw that there is a surjection $U_q^{\text{fin}}\to M_\lambda$. Using this, we can view $M_\lambda$ as an $\OqB$-comodule, or an integrable $U_q^{\geq 0}$-module, via the quotient of the adjoint action. This action is just the usual action twisted by $-\lambda$ and so with this $U_q^{\geq 0}$-module structure it is isomorphic to $M_\lambda\otimes L_{-\lambda}$ and has trivial highest weight. Then, as an object of $\flag$, $\D_q^\lambda=\Oq\otimes_L M_\lambda$ with the tensor $\OqB$-coaction and with the action of $\Oq$ on the left factor, where we view $M_\lambda$ is an $\OqB$-comodule just as now. It's moreover in $\Dflag$: (i) follows from our discussion above of the fact that $M_\lambda$ has trivial highest weight as an $\OqB$-comodule, and (ii) simply follows from the fact that $\D_q$ is a $U_q^{\geq 0}$-module algebra.

Then $\D_q^\lambda$ represents the global section functor, i.e. $\Gamma(M)=\Hom_{\Dflag}(\D_q^\lambda, M)$ by the above. In particular, $\Gamma(\D_q^\lambda)$ is a ring. Also one can easily check that $\D_q^\lambda$ is the maximal quotient of $\D_q$ that lies in $\Dflag$, where we take the quotient $\D_q$-action and the quotient of the $U_q^{\geq 0}$-action (\ref{actiononDq}) on $\D_q$.

\begin{definition} Let $M_\lambda$ be the Verma module with highest weight $\lambda$. Let $J_\lambda=\text{Ann}_{U_q^{\text{fin}}}(M_\lambda)$. We write $U_q^\lambda=U_q^{\text{fin}}/J_\lambda$.
\end{definition}

We finally recall the notion of regular and dominant weights in this context. By \cite[Lemma 6.3]{Jantzen} the centre $Z$ of $U_q$ acts on any Verma module $M_\lambda$ by a character $\chi_\lambda$. Following \cite[2.1]{QDmod1} we say that $\lambda\in T_P$ is \emph{dominant} if $\chi_\lambda\neq\chi_{\lambda+\mu}$ for any $0\neq\mu\in Q^+$, and that $\lambda$ is regular dominant if for all $\mu\in P^+$ and all weight $\gamma\neq\mu$ of $V(\mu)$, where $V(\mu)$ denotes the simple $U_q$-module of highest weight $\mu$, then we have $\chi_{\lambda+\mu}\neq\chi_{\lambda+\gamma}$. When $\lambda\in P$ this is equivalent to saying that it's dominant, respectively regular dominant in the classical sense.

\begin{thm}[{\cite[Theorem 4.12]{QDmod1}}]\label{BB} Suppose that $\lambda\in T_P$ is regular and dominant. Then there is an equivalence of categories
$$
\Gamma: \D_{B_q}^\lambda (G_q)\to \Gamma(\D_q^\lambda)\text{-mod}.
$$
whose quasi-inverse is given by the localisation functor \emph{$\loc$}$(M)=\D_q^\lambda\otimes_{\Gamma(\D_q^\lambda)}M$.
\end{thm}

The proof of this Theorem uses an analogue of the Beilinson-Bernstein `key lemma' \cite[Lemma 4.14]{QDmod1} and the fact that there is a natural map $U_q^\lambda\to \Gamma(\D_q^\lambda)$ (see the proof of \cite[Proposition 4.8]{QDmod1}).

\subsection{An $R$-form of $\flag$}\label{Rflag} We now return to our integral forms $\A$ and $\B$ and make completely analogous definitions to the previous section. Many of our constructions are similar to those of \cite[Section 3]{QDmod1}

\begin{definition} The integral quantum flag variety is the category $\Ca$ whose objects consist of $\A$-modules $M$ which are equipped with a right $\B$-comodule action $M\to M\otimes_R\B$ such that the $\A$-action map $\A\otimes_R M\to M$ is a comodule homomorphism where we give $\A\otimes_R M$ the tensor comodule structure. The morphisms are just the $\A$-linear maps which are also comodule homomorphisms.
\end{definition}

There is an obvious functor
\begin{align*}
\Ca&\longrightarrow \flag\\
M&\longmapsto M_L:=M\otimes_R L
\end{align*}
to our quantum flag variety. Given $M\in\Ca$, we will write $\rho_M$ (respectively $\rho_{M_L}$) to denote the comodule map on $M$ (respectively $M_L$).

Next, there are several adjunctions we need to describe. Namely we have
\[
\begin{tikzcd}
\A\text{-mod} \arrow{r}{q}\arrow{d}{\theta} & R\text{-mod} \arrow{d}{\phi}\\
\Ca\arrow{r}{p}& \B\text{-comod}
\end{tikzcd}
\]
where each arrow denotes a pair of functors. We write $(\theta^*, \theta_*)$, $(p^*, p_*)$, $(q^*, q_*)$ and $(\phi^*, \phi_*)$ where each time the `lower star' functors are the right adjoints and go in the direction of the arrows. The functor $\theta_*:\A\text{-mod}\to\Ca$ is given by $N\mapsto N\otimes_R \B$ where $\A$ acts on $\theta_*(N)$ via the tensor action and the $\B$-coaction comes from the second factor, while $\theta^*:\Ca\to\A\text{-mod}$ is just the forgetful functor. The bijection making this an adjunction is as follows: let $M\in \Ca$ and $N\in\A\text{-mod}$, and let $\rho: M\to M\otimes_R\B$ and $\varepsilon:\B\to R$ be the comodule map and the counit of $\B$ respectively; given a module homomorphism $f:M\to N$, we construct a morphism $g:M\to N\otimes_R \B$ in $\Ca$ by taking the composite $(f\otimes\id)\circ \rho$. Conversely, given a morphism $g:M\to N\otimes_R\B$ in $\Ca$, we construct a module homomorphism $f:M\to N$ by taking the composite $(\id\otimes\varepsilon)\circ g$.

Moreover the adjunction between $\Ca$ and $\B\text{-comod}$ is given by $p_*=\text{forgetful}$ one way and the functor $p^*:M\mapsto \A\otimes_R M$ the other way, where $\A$ acts on the first factor and the $\B$-coaction is the tensor coaction. The bijection is as follows: given a map $f: \A\otimes_R M\to N$ in $\Ca$ we get a comodule map $M\to N$ by taking $m\mapsto f(1\otimes m)$, and conversely given a comodule map $g:M\to N$ we get a map $\A\otimes_R M\to N$ by post-composing $1\otimes g:\A\otimes_R M\to \A\otimes_R N$ with the action map $\A\otimes_R N\to N$.

Similarly $q_*=\text{forgetful}$, $q^*:M\mapsto \A\otimes_R M$, $\phi^*=\text{forgetful}$ and $\phi_*:M\to M\otimes_R\B$ where the coaction is on the second factor, all with similar bijections as in the above.

In particular, the maps $M\to \theta_*\theta^*(M)$ and $M\to\phi_*\phi^*(M)$ are both just the comodule map and so are injective, since the comodule map has left inverse $1\otimes\varepsilon$. Also note that since $\A$ and $\B$ are torsion-free and so flat over $R$, all the functors are exact and so $\theta_*, p_*, q_*$ and $\phi_*$ all map injective objects to injective objects.

\begin{lem}
The categories $\Ca$ and $\B$-comod have enough injectives.
\end{lem}

\begin{proof}
Let $M\in\Ca$ and let $I$ be an injective $\A$-module such that there is an $\A$-linear injection $M\to I$. By the above, the adjunction map $M\to\theta_*\theta^*(M)$ is injective, and so there is an injection
$$
M\to \theta_*\theta^*(M)\to \theta_*(I).
$$
But since $\theta_*$ is the right adjoint of an exact functor we see that $\theta_*(I)=I\otimes_R\B$ is injective and we're done for $\Ca$. The proof for $\B\text{-comod}$ is entirely analogous working with $\phi$ instead of $\theta$.
\end{proof}

Now we can define the global sections functor $\Gamma:\Ca\to R\text{-mod}$ to be
$$
\Gamma(M):=\Hom_{\Ca}(\A,M)=\{m\in M : \rho(m)=m\otimes 1\}:=M^{\B}.
$$
So in particular the above lemma shows that we can right derive this functor.

\subsection{Proj categories}\label{proj} Our first main aim is to show that our category $\Ca$ is a noncommutative projective scheme in the sense of Artin-Zhang \cite[2.3-2.4]{ArtinZhang}. We quickly recall the definitions. Given a $\Z^n$-graded ring $\mathcal{R}=\oplus_{\mathbf{m}\in \Z^n} \mathcal{R}_{\mathbf{m}}$, we say that a graded (left or right) $\mathcal{R}$-module $M$ is \emph{torsion} if, for every $m\in M$, there exists some $k$ such that $m$ is killed by $\mathcal{R}_{\geq k}:=\oplus_{m_1,\ldots,m_n\geq k} \mathcal{R}_{\mathbf{m}}$. Write $\mathcal{R}$-mod to denote the category of \emph{graded} (left or right) $\mathcal{R}$-modules. The full subcategory $\mathcal{T}(\mathcal{R})$ of torsion modules is a Serre subcategory of $\mathcal{R}$-mod, and we let Proj$(\mathcal{R}):=\mathcal{R}\text{-mod}/\mathcal{T}(\mathcal{R})$ denote the quotient category. Similarly we denote by proj$(\mathcal{R})$ the quotient category of the category of finitely generated graded modules by the full subcategory of finitely generated torsion modules.

Now suppose that we are equipped with a tuple $(\C, \Of, s_1, \ldots, s_n)$ where $\C$ is an abelian category, $\Of$ is an object of $\C$ and $s_1, \ldots, s_n$ are pairwise commuting autoequivalences of $\C$. For $\mathbf{m}\in\Z^n$ and an object $M$ of $\C$, we define twisting functors on $\C$ by
$$
M(\mathbf{m})=s_1^{m_1}\cdots s_n^{m_n}(M).
$$
We let $\Gamma$ denote the functor $\Hom_{\C}(\Of,-)$ and we set $\underline{\Gamma}(M)=\oplus_{\mathbf{m}\in \N^n} \Gamma(M(\mathbf{m}))$. Note that $\underline{\Gamma}(\Of)$ is a graded ring where the multiplication is defined as follows: for $a\in \Gamma(\Of(\mathbf{m}))$ and $b\in \Gamma(\Of(\mathbf{m'}))$, we set
$$
a\cdot b:=s_1^{m'_1}\cdots s_n^{m'_n}(a)\circ b.
$$
Similarly for each $M$ in $\C$, $\underline{\Gamma}(M)$ is a graded right $\underline{\Gamma}(\Of)$-module. Finally, let $\C^0$ denote the full subcategory of Noetherian objects in $\C$. Then we have the following multigraded version of a result of Artin and Zhang (see also \cite[Proposition 2.1]{QDmod1}):

\begin{prop} \emph{(\cite[Theorem 4.5]{ArtinZhang}, \cite[Remark 2.2]{QDmod1})}  Let $(\C, \Of, s_1, \ldots, s_n)$ be a tuple as above, such that the following hold:
\begin{enumerate}
\item $\Of$ belongs to $\C^0$;
\item $\Gamma(\Of)$ is a right Noetherian ring and $\Gamma(M)$ is a finitely generated $\Gamma(\Of)$-module for each object $M$ of $\C^0$;
\item for each $M\in\C^0$ there is an epimorphism $\oplus_{i=1}^l \Of(\mathbf{-m_i})\to M$ for some $l\geq 1$ and $\mathbf{m_1},\ldots, \mathbf{m_l}\in \N^n$; and
\item given $M,N\in\C^0$ and an epimorphism $M\to N$ in $\C$, the associated map $\Gamma(M(\mathbf{m}))\to \Gamma(N(\mathbf{m}))$ is surjective for $\mathbf{m}>>0$.
\end{enumerate}
Then $\underline{\Gamma}(\Of)$ is right Noetherian and $\C^0$ is equivalent to \emph{proj}$(\underline{\Gamma}(\Of))$ (working with graded right modules). If, moreover, we assume that every object of $\C$ is a direct limit of objects in $\C^0$, then $\C$ is equivalent to \emph{Proj}$(\underline{\Gamma}(\Of))$.
\end{prop}

Note that in general the assignement $M\mapsto \underline{\Gamma}(M)$ defines a left exact functor from $\C$ to the category of graded $\underline{\Gamma}(\Of)$-modules. Now we return to the setting of the quantum $R$-flag variety.

\begin{definition}
We define the \emph{representation ring} to be $R_q:=\oplus_{\lambda\in P^+}\Gamma(\A(\lambda))$ with the induced ring structure from the multiplication in $\A$.
\end{definition}

\begin{remark}
We will apply the above setup to the category $\Ca$. Specifically we will set the autoequivalences to be $s_i(M):=M(\varpi_i)$. The above mentioned ring structure for $\underline{\Gamma}(\A)$ is then just the ring structure of $R_q^{\text{op}}$. So we will apply the above results, working with $R_q$, by replacing every instance of the word `right' by `left'.
\end{remark}

\begin{thm} The category $\Ca$ is equivalent to \emph{Proj}$(R_q)$ (this time working with left modules).
\end{thm}

We now start preparing for the proof of this theorem.

\subsection{Line bundles}\label{line} We begin by proving results in $\Ca$ analogous to standard facts about line bundles on the flag variety. We will mostly just adapt arguments from \cite[Section 3.4]{QDmod1}. They apply essentially identically but we repeat them nevertheless.

Note that we have a functor of taking $\B$-invariants in $\B$-comod, which we denote by $\tilde{\Gamma}$. This functor is also left exact. Let $\ind$ be the functor $\Gamma\circ p^*: M\mapsto (\A\otimes_R M)^{\B}$. Since $\B$-comod is isomorphic to the category of integrable $U^{\text{res}}(\bo)$-modules of type $\boldsymbol{1}$ by Theorem \ref{comodcat}, $\ind$ is just the induction functor $M\mapsto (\A\otimes_RM)^{U^{\text{res}}(\bo)}$ studied in \cite{Andersen}. This will be useful in the next result.

\begin{prop} \begin{enumerate}
\item If $I\in\B$-comod is injective, then $p^*(I)$ is $\Gamma$-acyclic.
\item For any $M\in\B$-comod and any $i\geq 0$, $R^i\text{\emph{Ind}}(M)=R^i\Gamma(p^*(M))$.
\item For any $M\in \Ca$ and any $i\geq 0$, $R^i\tilde{\Gamma}(p_*(M))\cong R^i\Gamma(M)$.
\item The functor $\Gamma$ has cohomological dimension at most $N=\dim G/B$.
\end{enumerate}
\end{prop}

\begin{proof} (i) The adjunction map $I\to\phi_*\phi^*(I)=I\otimes_R\B$ is injective. So as $I$ is injective, this embedding splits. Therefore, as $p^*$ is additive and since the derived functors $R^i\Gamma$ commute with finite direct sums, it suffices to show that $p^*(I\otimes_R\B)$ is acyclic. To simplify notation a bit, we write $J=\phi^*(I)$. We claim that we have an isomorphism $p^*(\phi_*(J)) \overset{\cong}{\to}\theta_*(q^*(J))$. Indeed, as $R$-modules they both equal $\A\otimes_R I\otimes_R \B$ and the isomorphism is given by $a\otimes i\otimes b\mapsto \sum a_1\otimes i\otimes a_2b$, with inverse $a\otimes i\otimes b\mapsto \sum a_1\otimes i\otimes S(a_2)b$. These maps are easily checked to be both module and comodule homomorphisms. Hence we have that 
\begin{align*}
R^i\Gamma(p^*(I\otimes_R\B))&\cong R^i\Gamma(\theta_*(q^*(J)))\\
&=\text{Ext}^i_{\Ca}(\A, \theta_*(q^*(J)))\\
&\cong \text{Ext}^i_{\A}(\theta^*(\A), q^*(J))\\
&=\text{Ext}^i_{\A}(\A, q^*(J))=0
\end{align*}
for $i>0$, as $\A$ is projective as an $\A$-module. Here we used the fact that $\theta_*$ is exact and preserves injectives in the second isomorphism.

(ii) Pick an injective resolution $M\to I^\bullet$. Then, by (i), $p^*(M)\to p^*(I^\bullet)$ is a $\Gamma$-acyclic resolution of $p^*(M)$, hence it computes the cohomology of $\Gamma$. The result now follows.

(iii) Pick an injective resolution $M\to I^\bullet$ in $\Ca$. Since $p_*$ preserves injectives, it follows that $M\to I^\bullet$ is also an injective resolution in $\B$-comod. The result follows.

(iv) Let $M\in\Ca$. Since $p_*$ maps injectives to injectives, any injective resolution of $M$ in $\Ca$ is also an injective resolution of $M$ in the category of $\B$-comodules. Thus we see that $R^i\Gamma(M)\cong R^i\tilde{\Gamma}(p_*(M))$ for all $i\geq 0$ and it suffices to show that the right hand side vanishes for $i\geq N$. To simplify notation we will drop the $p_*$ when referring to an element of $\Ca$ viewed only as a comodule.

Now, note that there is a $\B$-comodule map $M\to p^*(M)=\A\otimes_R M$ given by $m\mapsto 1\otimes m$. This map has a splitting given by the $\A$-action map, which is a comodule homomorphism by definition of $\Ca$. So, as $\B$-comodules, $M$ is a direct summand of $\A\otimes_R M$. This in turn implies that $R^i\tilde{\Gamma}(M)$ is a direct summand of $R^i\tilde{\Gamma}(p^*(M))$. By (ii) the latter equals $R^i\text{Ind}(M)$. But it was proved in \cite[Theorem 5.8]{Andersen} that this induction functor has cohomological dimension at most $N$. So the result follows.
\end{proof}

\begin{definition} We let $T_P^R=\{\lambda\in T_P : \lambda((U^{\text{res}})^0)\subseteq R^\times\}$, which is a subgroup of $T_P$. Note that for $\lambda\in P$, the associated element of $T_P$ belongs in $T_P^R$. For each $\lambda\in T_P^R$ we have a rank 1 $U^{\text{res}}(\bo)$-module $R_\lambda$.

When $\lambda\in P$ we may view it as a comodule with coaction $1\mapsto 1\otimes\lambda$. In that case, we let $\A(\lambda):=p^*(R_{-\lambda})$, which we call a line bundle. More generally, for $M\in\Ca$, we will write $M(\lambda)$ for $M\otimes_R R_{-\lambda}$. By letting $\A$ act on the left factor and giving it the tensor $\B$-coaction, this is also an element of $\Ca$.
\end{definition}

\begin{cor} For all $\lambda\in P$ and all $i\geq 0$, $R^i\Gamma(\A(\lambda))$ is finitely generated as an $R$-module. Moreover if $\lambda\in P^+$ then $R^i\Gamma(\A(\lambda))=0$ for all $i>0$.
\end{cor}

\begin{proof}
By Proposition \ref{line}(ii), we have that $R^i\Gamma(\A(\lambda))=R^i\ind(R_{-\lambda})$. But it was proved in \cite[Theorem 5.8]{Andersen} that $R^i\ind$ sends finitely generated $R$-modules to finitely generated $R$-modules, and in \cite[Corollary 5.7]{Andersen} that $R^i\ind(R_{-\lambda})=0$ when $\lambda\in P^+$ and $i>0$.
\end{proof}

\subsection{Generators for $\Ca$}\label{gen_by_global_sections} We now show results analogous to \cite[Lemmas 3.13 \& 3.16, Proposition 3.5]{QDmod1}. The proofs are essentially identical with the exception of part (ii) of the Lemma below where a few small adjustments are necessary to deal with torsion.

Suppose that $M$ is a $\B$-comodule or in other words an integrable $U^{\text{res}}(\bo)$-module. We will write $V$ to denote the underlying $R$-module of $M$ equipped with the trivial $\B$-coaction.

\begin{lem} Let $M$ be as above.
\begin{enumerate}
\item If $M$ is in fact a $\A$-comodule, viewed as a $\B$-comodule via restriction, then $p^*(M)\cong p^*(V)$ in $\Ca$.
\item Suppose now that $M$ is finitely generated over $R$, and moreover suppose that all the weight spaces of $M$ have weight of the form $-\lambda$ where $\lambda\in P^+$. Then
\begin{enumerate}
\item $M$ is acyclic with respect to the induction functor;
\item there is an $\A$-comodule which surjects onto $M$ as a $\B$-comodule.
\end{enumerate}
\end{enumerate}
\end{lem}

\begin{proof}
(i) We have $p^*(M)=\A\otimes_R M$ and $p^*(V)=\A\otimes_R V$, which are the same as $R$-modules. The isomorphism is given by the map $a\otimes m\mapsto \sum am_2\otimes m_1$ where $m\mapsto \sum m_1\otimes m_2$ denotes the $\A$-coaction. It quite evidently is an $\A$-module map, and it is straightforward to check that it is also a $\B$-comodule map. Thus this is a morphism in $\Ca$. Quite similarly we have a map going the other way given by $a\otimes m\mapsto \sum a S(m_2)\otimes m_1$, which is also a morphism in $\Ca$ by the Hopf algebra axioms. It also follows from the Hopf algebra axioms that these two maps are inverse to each other, and so we have an isomorphism.

(ii) Write $M=\oplus_\lambda M_{-\lambda}$ for the weight space decomposition of $M$, where $\lambda\in P^+$ ranges through the weights of $M$. Since $M$ is finitely generated there are only finitely many weights, and we may list them as $-\lambda_1,-\lambda_2,\ldots,-\lambda_r$ so that $-\lambda_r$ is maximal among them. Hence $N:=M_{-\lambda_r}$ is a $U^{\text{res}}(\bo)$-submodule. We prove (1) by induction on $r$. Simply note that $N$ is acyclic by \cite[Corollary 5.7(ii)]{Andersen}, and by taking the long exact sequence associated to the short exact sequence
$$
0\to N\to M\to M/N\to 0
$$
we see that $M$ is also acyclic by induction hypothesis.

For (2), note that $\ind(M)$ is finitely generated over $R$ by \cite[Proposition 3.2]{Andersen}. Hence the result will follow if we show that the map $\res\ind(M)\to M$ coming from Frobenius reciprocity (see \cite[Proposition 2.12]{Andersen}) is surjective. We prove this by induction on $r$. Suppose that $r=1$ so that $M$ is isomorphic to a finite direct sum of modules all of the form $R_{-\lambda}$ or $R_{-\lambda}/\pi^n R_{-\lambda}$ for some $n\geq 1$. Then, it suffices to prove the claim for these summands. But it is true for $R_{-\lambda}$ by \cite[Proposition 3.3]{Andersen} and so it follows that it also true for any $R_{-\lambda}/\pi^n R_{-\lambda}$ since we have a commutative diagram
\[
\begin{tikzcd}
R_{-\lambda} \arrow[two heads]{r} & R_{-\lambda}/\pi^n R_{-\lambda}\\
\res\ind(R_{-\lambda})\arrow{r}\arrow[two heads]{u}& \res\ind(R_{-\lambda}/\pi^n R_{-\lambda})\arrow{u}
\end{tikzcd}
\]
Now for $r> 1$ we consider the commutative diagram
\[
\begin{tikzcd}
0\arrow{r} & N \arrow{r} & M\arrow{r} & M/N \arrow{r} & 0\\
0\arrow{r} & \res\ind(N) \arrow{u}\arrow{r} & \res\ind(M)\arrow{u}\arrow{r} & \res\ind(M/N)\arrow{u}\arrow{r} & 0
\end{tikzcd}
\]
in which both rows are exact by (1), and we conclude that $\res\ind(M)\to M$ is surjective by the induction hypothesis and the Five Lemma.
\end{proof}

Let coh$(\Ca)$ denote the full subcategory of $\Ca$ consisting of objects $M$ which are finitely generated as $\A$-modules. We call elements of coh$(\Ca)$ \emph{coherent modules}.

\begin{prop} Let $M\in\text{\emph{coh}}(\Ca)$. Then there exists $\lambda\in P^+$ such that for all $\mu\in \lambda+P^+$, $M(\mu)$ is generated by finitely many global sections. In particular there is finite direct sum of $\A(-\lambda)$ surjecting onto $M$ in $\Ca$.
\end{prop}

\begin{proof}
Suppose $m_1, \ldots, m_n$ generate $M$ over $\A$. Since $M$ is a $\B$-comodule i.e an integrable $U^{\text{res}}(\bo)$-module it is in particular locally finite. So if we let $W$ denote the $U^{\text{res}}(\bo)$-submodule they generate, then we have that $W$ is finitely generated over $R$. Moreover we have a surjection $p^*(W)\to M$ in $\Ca$. We may pick $\lambda\in P$ such that $W(\lambda)=W\otimes R_{-\lambda}$ satisfies the conditions of Lemma \ref{gen_by_global_sections}(ii) and let $N$ be an $R$-finite $\A$-comodule surjecting onto $W(\lambda)$. Then $p^*(N)$ surjects onto $p^*(V(\lambda))$ and hence onto $M(\lambda)$. By Lemma \ref{gen_by_global_sections}(i) and since $N$ is finite over $R$, we have that $p^*(N)$ is generated as an $\A$-module by finitely many global sections, and these define a surjection $\A^r\to p^*(N)$. Thus we have a surjection $\A^r\to M(\lambda)$ and twisting by $-\lambda$ we get a surjection $\oplus_{i=1}^r \A(-\lambda)\to M$ as claimed. Of course the same argument shows that $M(\mu)$ is generated by its global sections for any $\mu\in \lambda+P^+$.
\end{proof}

In the next section we repeatedly use a general construction, which we record here:

\subsection{Lemma}\label{useful_lemma}\emph{Let $M\in \Ca$ and let $m_1,\ldots, m_i\in M$ for some $i\geq 1$. Then there is a unique minimal coherent submodule $P$ of $M$ such that $m_1,\ldots, m_i\in P$.}

\begin{proof}
Let $N$ be the $U^{\text{res}}(\bo)$-submodule of $M$ generated by $m_1,\ldots, m_i$. Then $N$ is $R$-finite and we let $P$ be the $\A$-submodule of $M$ generated by $N$. Since the $\A$-action on $M$ is a comodule homomorphism it follows that $P$ is a subcomodule of $M$ and it is in $\coh(\Ca)$ as $N$ is finite over $R$. Moreover, any coherent submodule of $M$ which contains $m_1, \ldots, m_i$ must also contain $N$, and so must contain $P$.
\end{proof}

\subsection{Coherent modules}\label{coherent} Since we do not know whether $\A$ is Noetherian or not, it is not clear yet that coh$(\Ca)$ is a well-behaved category. This is what we turn to next. We first need to establish:

\begin{lem} The ring $R_q$ is graded Noetherian.
\end{lem}

\begin{proof} Since $q^{\frac{1}{d}}\equiv 1\pmod{\pi}$, the $U^{\text{res}}\otimes_R k$-representation $\Gamma(\A(\lambda))\otimes_R k$ is just the global sections of the usual line bundle $\mathcal{L}_\lambda$ on the flag variety $G_k/B_k$ over $k$ for any $\lambda\in P^+$ by \cite[3.11]{Andersen}, noting that $\mathcal{L}_\lambda$ has no higher cohomology by the classical Kempf vanishing theorem (see e.g. \cite[Proposition II.4.5]{Jantzen2}). Hence we see that the ring $R_q/\pi R_q$ is isomorphic to the ring of regular functions on the basic affine space $G_k/N_k$, and so is Noetherian. Moreover the graded pieces $\Gamma(\A(\lambda))$ are all finitely generated over $R$ by Corollary \ref{line}. Thus the result follows from Proposition \ref{Noetherian}(ii).
\end{proof}

\begin{thm} The modules in $\Ca$ which are finitely generated as $\A$-modules coincide exactly with the Noetherian objects.
\end{thm}

\begin{proof}
We prove this result in several steps. First we claim that $\A$ satisfies ACC in the category coh$(\Ca)$. Indeed, assume we have a chain
$$
M_1\subseteq M_2\subseteq M_3\subseteq \cdots
$$
of coherent submodules of $\A$. Recall the functor $\underline{\Gamma}$ from \ref{proj}. By Noetherianity of $R_q=\underline{\Gamma}(\A)$ and by left exactness of $\underline{\Gamma}$, we get that there is some $m\geq 1$ such that for all $n\geq m$, $\underline{\Gamma}(M_n)=\underline{\Gamma}(M_m)$. In particular we get that $\Gamma(M_n(\lambda))=\Gamma(M_m(\lambda))$ for all $\lambda\in P^+$. Fix any $n\geq m$. Then by Proposition \ref{gen_by_global_sections}, we may pick $\lambda>>0$ such that both $M_n(\lambda)$ and $M_m(\lambda)$ are generated by their global sections. But then the above equality of global sections implies that $M_n(\lambda)=M_m(\lambda)$ and hence after untwisting that $M_n=M_m$.

Next, we claim that $\A$ satisfies ACC in $\Ca$. Indeed, suppose we have a chain
$$
M_1\subset M_2\subset M_3\subset \cdots
$$
of subobjects of $\A$ with $M_i\neq M_{i+1}$ for every $i\geq 1$. Then we may pick $m_1\in M_1$ and $m_i\in M_i\setminus M_{i-1}$ for every $i\geq 2$. By Lemma \ref{useful_lemma}, for each $i\geq 1$ we may consider the smallest coherent submodule $P_i$ of $M_i$ which contains $m_1,\ldots, m_i$. Note that $P_i\subset P_{i+1}$ by the proof of Lemma \ref{useful_lemma}. But $m_i\in P_i$ for every $i$, so that we get a strict ascending chain
$$
P_1\subset P_2\subset P_3\subset \cdots
$$
of coherent submodules of $\A$, which is a contradiction by our first step.

Thus we have proved that $\A$ is a Noetherian object. It is then immediate that every line bundle $\A(\lambda)$ is also a Noetherian object. But by Proposition \ref{gen_by_global_sections}, this implies that every coherent module is a Noetherian object. Finally, for the converse, the above argument that $\A$ satisfies ACC in $\Ca$ also shows that Noetherian objects are finitely generated over $\A$. Indeed, if $M$ is not finitely generated, pick $m_1\in M$ and let $P_1$ be the smallest coherent submodule of $M$ containing $m_1$, given by Lemma \ref{useful_lemma}. Since $M$ is not coherent we have that $M\neq P_1$. So we can pick $m_2\in M$ such that $m_2\notin P_1$. Then we may apply Lemma \ref{useful_lemma} again and set $P_2$ to be the smallest coherent submodule of $M$ containing $m_1,m_2$. By construction, $P_1\subset P_2$ is a strict inclusion. As $M$ is not coherent, we may pick $m_3\in M\setminus P_2$. Carrying on, we get a strict ascending chain
$$
P_1\subset P_2\subset P_3\subset \cdots
$$
so that $M$ is not a Noetherian object.
\end{proof}

This in particular shows that $\coh(\Ca)$ is an abelian category. This has a few consequences.

\begin{prop} Let $M\in\emph{\text{coh}}(\Ca)$. Then:
\begin{enumerate}
\item there exists $\lambda\in P^+$ such that for all $\mu\in \lambda+P^+$, $M(\mu)$ is acyclic; and
\item \emph{(Serre finiteness)} for all $i\geq 0$, $R^i\Gamma(M)$ is finitely generated as an $R$-module.
\end{enumerate}
\end{prop}

\begin{proof}
(i) By Proposition \ref{gen_by_global_sections} and the above Theorem, we may find a resolution of $M$ of the form
$$
F_{\bullet}: F_N\overset{f_N}{\rightarrow}\cdots\overset{f_2}{\rightarrow}F_1\overset{f_1}{\rightarrow} M\to 0
$$
where the $F_i$ are finite direct sums of line bundles. Pick $\lambda\in P$ sufficiently large such that all the line bundles in $F_{\bullet}(\lambda)$ are of the form $\A(\mu)$ for $\mu\in P^+$. Then by Corollary \ref{line}, all the $F_i(\lambda)$ are $\Gamma$-acyclic. Let $K_0=M(\lambda)$ and $K_j=\ker{f_j(\lambda)}$ for $1\leq j\leq N$. Then we have a short exact sequence
$$
0\to K_j\to F_j(\lambda)\overset{f_j(\lambda)}{\to} K_{j-1}\to 0
$$
for every $1\leq j\leq N$, and the long exact sequence yields isomorphisms $R^i\Gamma(K_{j-1})\cong R^{i+1}\Gamma(K_j)$ for all $i\geq 1$. Thus, by using Proposition \ref{line}(iv), we obtain
$$
R^i\Gamma(M(\lambda))\cong R^{i+1}\Gamma(K_1)\cong\cdots \cong R^{i+N}\Gamma(K_N)=0
$$
for all $i\geq 1$ as required. Again the same argument works by replacing $\lambda$ by any $\mu\in \lambda+P^+$.

(ii) The proof we give is completely analogous to the proof in \cite[Theorem III.5.2]{Hartshorne}. First note that by Proposition \ref{line}(iv) we have $R^i\Gamma(M)=0$ for all $i >N$ and so we may assume that $i\leq N$. We will prove the result by downwards induction on $i$, the cases $i>N$ being already covered.

By Proposition \ref{gen_by_global_sections} there is a surjection $f:\bigoplus_{j=1}^n \A(-\lambda_j)\to M$ in $\Ca$, where each $\lambda_j\in P^+$. This gives a short exact sequence
$$
0\to K\to \bigoplus_{j=1}^n \A(-\lambda_j)\to M\to 0
$$
Applying the long exact sequence, we obtain
$$
\cdots \to \bigoplus_{j=1}^n R^i\Gamma(\A(-\lambda_j))\to R^i\Gamma(M)\to R^{i+1}\Gamma(K)\to\cdots
$$
By the induction hypothesis applied to $K$ (which we may apply by the above Theorem), we get that $R^{i+1}\Gamma(K)$ is finitely generated. Now by Corollary \ref{line} and since $R$ is Noetherian, we see that $R^i\Gamma(M)$ is finitely generated over $R$ as well.
\end{proof}

One of our main aims will be to establish a $D$-modules version of part (i) of the Proposition. Before we get to that, we can now finally fulfill our promise:

\begin{proof}[Proof of Theorem \ref{proj}]
Note that every object of $\Ca$ is a direct limit of object of $\coh(\Ca)$. Indeed, it suffices to show that every element of any $M\in\Ca$ is contained in a coherent submodule. But this is given by Lemma \ref{useful_lemma}.

So we just have to check all conditions (i)-(iv) from Proposition \ref{proj}. Condition (i) is just Theorem \ref{coherent}, (ii) follows from the fact that $\Gamma(\A)=R$ and from Proposition \ref{coherent}(ii), and (iii) follows from Proposition \ref{gen_by_global_sections}. Finally, condition (iv) is easily deduced from Theorem \ref{coherent} and Proposition \ref{coherent}(i). Indeed, suppose $M\to N$ is a surjection between coherent modules in $\Ca$ and let $K$ denote its kernel. For $\lambda>>0$, we know that $K(\lambda)$ is $\Gamma$-acyclic, and so the corresponding map $\Gamma(M(\lambda))\to \Gamma(N(\lambda))$ is surjective.
\end{proof}

\subsection{Weyl group translates of the big cell}\label{ore} We now introduce certain localisations of $\A$ from Joseph (see \cite[3.1-3.3]{joseph} and \cite[9.1.10]{josephbook}). For each fundamental weight $\varpi_i$, consider the highest weight representation $V(\varpi_i)$ of $U_q$. It contains a free $R$-lattice $M:=\ind (R_{\varpi_i})^*$ that is a $U^{\text{res}}$-module. In fact $M$ is a cyclic module generated by a highest weight vector $v\in V(\varpi_i)$ (see \cite[Proposition 3.3]{Andersen}). Let $f\in M^*$ be the corresponding dual vector. Let $c_{\varpi_i}:=c^M_{f,v}\in \A$ be the corresponding matrix coefficient. Joseph showed in \emph{loc. cit.} that these commute and we may define for any $\mu=\sum_i n_i \varpi_i\in P^+$ the element $c_\mu=\prod_i c_{\varpi_i}^{n_i}\in \A$. Moreover, for any $\mu\in P^+$, $c_\mu=c^{V(\mu)}_{f_\mu,v_\mu}$ is the matrix coefficient of the highest weight representation $V(\mu)$ of $U_q$. In fact it is the matrix coefficient of a $U^{\text{res}}$-lattice inside $V(\mu)$, namely $\ind (R_{-\mu})^*$.

Recall that $\A$ is a $U^{\text{res}}$-module algebra via the action $u\cdot f=\sum f_2(u)f_1$. If we identify $\A$ with a submodule of $\Hom_R(U^{\text{res}},R)$, this action is given by
$$
(u\cdot f)(x)=f(xu)
$$
for all $u, x\in U^{\text{res}}$ and all $f\in \A$. Therefore, identifying $c_\mu$ with the matrix coefficient corresponding to a highest weight vector as above, we see that $u\cdot c_\mu=\mu(u)c_\mu$ for any $u\in(U^{\text{res}})^0$ and $E_{\alpha_i}^{(r)}\cdot c_\mu=0$ for any $i$ and any $r\geq 1$. Thus in the $\B$-comodule language, we have $\Delta(c_\mu)=c_\mu\otimes \mu\in \A\otimes \B$. So we see that $c_\mu\in \Gamma(\A(\mu))$.

Recall now that $\Gamma(\A(\mu))=\ind{R_{-\mu}}$ is an integrable $U^{\text{res}}$-module. The elements of it can all be identified as certain functions in $\Hom_R(U^{\text{res}},R)$, and the module structure is given by 
$$
(u\cdot f)(x)=f(S(u)x)
$$
for all $u, x\in U^{\text{res}}$ and all $f\in \Gamma(\A(\mu))$. With respect to this action, the element $c_\mu$ has weight $-\mu$ and so is a lowest weight vector, since the module $\Gamma(\A(\mu))$ is a free $R$-lattice inside $V(-w_0\mu)$ and satisfies the Weyl character formula by \cite[Corollary 3.3]{Andersen}. In particular we see that $\Gamma(\A(\mu))$ has a unique (up to scalars) extreme $w$-weight vector $c_{w\mu}$ of weight $-w\mu$ for any Weyl group element $w\in W$, which we may choose to equal
$$
c_{w\mu}=E_{\alpha_{i_1}}^{(r_1)}\cdots E_{\alpha_{i_s}}^{(r_s)}\cdot c_\mu
$$
where $w=s_{i_1}\cdots s_{i_s}$ and where the exponents $r_j$ are defined by $r_s=\langle \mu, \alpha_{i_s}^\vee \rangle$ and $r_j=\langle s_{i_{j+1}}\cdots s_{i_s}\mu, \alpha_{i_j}^\vee \rangle$ for $j\leq s-1$. Then Joseph \cite[9.1.10]{josephbook} showed that $c_{w\lambda}c_{w\mu}=c_{w(\lambda+\mu)}$ for every $w\in W$ and every $\lambda, \mu\in P^+$. Therefore, for every $w\in W$, the set
$$
S_w:=\{c_{w\mu} : \mu \in P^+\}
$$
is multiplicatively closed in $\A$. Moreover we still have $c_{w\mu}\in \Gamma(\A(\mu))$, so that we may view $S_w$ as a multiplicatively closed subset of $R_q$. Joseph showed in \emph{loc. cit.} that $S_w$ is an Ore set in both $\Oq$ and its representation ring, but in fact his proof works equally well with $\A$ and in $R_q$ (see also \cite[III.2]{LuntsRos}). Hence we have:

\begin{lem} For every $w\in W$, $S_w$ is an Ore set in $\A$ and in $R_q$.
\end{lem}

So we may define localisations $\Aw:=S_w^{-1}\A$ for each Weyl group element. By viewing $\A\otimes \B$ as a left $\A$-module via the comultiplication $\Delta$, the comodule map $\A\to \A\otimes \B$, which by abuse of notation we also denote by $\Delta$, is an $\A$-module map, and its localisation gives a map
$$
\Delta_w:\Aw\to \Aw\otimes_R \B
$$
which defines a $\B$-comodule structure: for $f\in\A$ and $s\in S_w$ such that $\Delta(s)=s\otimes \lambda$, $\Delta_w$ sends $s^{-1}f$ to $(s^{-1}\otimes -\lambda)\cdot \Delta(f)$. Moreover the $\Aw$-module structure on $\Aw\otimes_R \B$ is defined by $\Delta_w$.

More generally, if $M\in\Ca$ with comodule map $\rho:M\to M\otimes_R\B$ then, by the axioms for $\Ca$, $\rho$ is an $\A$-module map where we view $M\otimes_R\B$ as an $\A$-module via $\Delta$, and its localisation gives rise to a map
$$
\rho_w:S_w^{-1}M\to S_w^{-1}M\otimes_R \B
$$
which will be $\Aw$-linear where $\Aw$ acts on $S_w^{-1}M\otimes_R \B$ via the map $\Delta_w$.

\begin{definition} We define $\Ca^w$ to be the category of $B$-equivariant $\Aw$-modules. Specifically, the objects consist of $\Aw$-modules $M$ which are equipped with a right $\B$-comodule action $M\to M\otimes_R\B$ such that the $\Aw$-action map $\Aw\otimes_R M\to M$ is a comodule homomorphism where we give $\Aw\otimes_R M$ the tensor comodule structure. The morphisms are just the $\Aw$-linear maps which are also comodule homomorphisms.
\end{definition}

The above discussion shows that there is a localisation functor $f_w^*:\Ca\to\Ca^w$ which sends a module $M$ to its localisation $S_w^{-1}M$ as an $\A$-module, and it has a right adjoint $(f_w)_*$ given by the forgetful functor.  Both of these are exact and they make $\Ca^w$ into a localisation of $\Ca$ in the sense of Gabriel i.e a quotient of $\Ca$ by a localising subcategory (see \cite[Chapter III.2]{gabriel}).

\subsection{\v{C}ech complexes}\label{Cech} We saw that $\Ca$ is equivalent to Proj$(R_q)$ and that we may equally localise any graded $R_q$-module at the set $S_w$ for any $w\in W$. Since the set $S_w$ contains elements of arbitrarily large degree in $R_q$, we see that the localisation functor $R_q\text{-mod}\to S_w^{-1}R_q\text{-mod}$ factors through Proj$(R_q)$ and makes $S_w^{-1}R_q\text{-mod}$ into a localisation of Proj$(R_q)$.

We have a global section functor on $\Ca^w$ which corresponds to taking $\B$-invariants. This is of course the same as the composite $\Gamma\circ(f_w)_*$. Now via the proj construction we see that global sections on $\Ca$ correspond to projection onto the degree 0 in Proj$(R_q)$. So we see that the global section functor on $S_w^{-1}R_q\text{-mod}$ is the functor of taking the degree 0 part of the graded module, which is exact! We then get: 

\begin{lem} The categories $S_w^{-1}R_q\text{-mod}$ and $\Ca^w$ have enough injectives, and they are naturally equivalent to each other as localisations of $\Ca$. Hence the global section functor on $\Ca^w$ is exact and objects of $\Ca^w$ are acyclic when viewed in $\Ca$.
\end{lem}

\begin{proof} By \cite[Corollary III.3.2]{gabriel} the first part follows from Lemma \ref{Rflag} and the fact that both categories are localisations of $\Ca$. By the above discussion, if the two categories are equivalent then global sections is exact. To prove that $S_w^{-1}R_q\text{-mod}$ and $\Ca^w$ are equivalent, we just need to show that $M\in \Ca$ has localisation zero if and only if $\underline{\Gamma}(M)$ has localisation zero.

Clearly if $M\in\Ca$ has localisation zero, then so does $\underline{\Gamma}(M)$. Conversely if $\underline{\Gamma}(M)$ has localisation zero, we show that $\underline{\Gamma}(S_w^{-1}M)=0$, which implies that $S_w^{-1}M=0$. Indeed suppose $s\in S_w$, $m\in M$ such that $\Delta(s)=s\otimes\mu$ and $s^{-1}m\in \Gamma(S_w^{-1}M(\lambda))$ for some $\lambda$. Then
$$
\rho(m)=\rho(s(s^{-1}m))=(s\otimes \mu)\rho(s^{-1}m)=(s\otimes \mu)(s^{-1}m\otimes \lambda)=m\otimes (\lambda +\mu)
$$
so that $m\in \Gamma(M(\lambda+\mu))$. By assumption there exists $t\in S_w$ such that $tm=0$. But then that means that the image of $m$ in $S_w^{-1}M$ is zero and so $s^{-1}m=0$. Thus we see that $\underline{\Gamma}(S_w^{-1}M)=0$ as required.

Finally, let $M\in\Ca^w$ and $M\to I^{\bullet}$ be an injective resolution of $M$ in $\Ca^w$. Note that since $(f_w)_*$ preserves injectives as it is the right adjoint to an exact functor, we have that $(f_w)_*(I^{\bullet})$ is an injective resolution of $(f_w)_*(M)$ in $\Ca$, and applying global sections and taking cohomology we obtain $R^i\Gamma((f_w)_*(M))=0$ for all $i>0$ since $\Gamma\circ(f_w)_*$ is exact.
\end{proof}

We think of $\Ca^w$ as being an analogue of the $w$-translate of the big cell on the flag variety, and the above lemma tells us that it is in some sense affine. Now to such a situation Rosenberg \cite[Sections 1 \& 2]{Rosenberg} (see also \cite[section III.3]{LuntsRos}) explained how to write down an analogue of the \v{C}ech complex which allows us to compute the cohomology of the functor $\Gamma$. Write $W=\{w_1,\ldots, w_m\}$, let $J=\{1,\ldots, m\}$ and for each $i\in J$ let $\sigma_i:=(f_{w_i})_*\circ f_{w_i}^*$. Moreover for any $\mathbf{i}=(i_1,\ldots, i_n)\in J^n$, let $\sigma_{\mathbf{i}}=\sigma_{i_1}\circ\cdots\circ \sigma_{i_n}$. Then by \cite[1.2 \& 1.3]{Rosenberg} we may write down a complex
$$
C^{\text{aug}}: \id_{\Ca}\to \bigoplus_{i\in J} \sigma_i\to \bigoplus_{\mathbf{i}\in J^2} \sigma_{\mathbf{i}}\to \bigoplus_{\mathbf{i}\in J^3} \sigma_{\mathbf{i}}\to\cdots
$$
where the maps are given as follows. Denote the adjunction morphism $\id_{\Ca}\to \sigma_i$ by $\eta_i$. Then for any $\mathbf{i}\in J^n$ and any $1\leq j\leq n$, there is a natural transformation
$$
\xi_n^j:\sigma_{i_1}\circ\cdots\circ \sigma_{i_n}\to \oplus_{i\in J}\sigma_{i_1}\circ\cdots\circ \sigma_{i_{j-1}}\circ\sigma_i \circ \sigma_{i_j}\circ\cdots\circ\sigma_{i_n}
$$
given by $\xi_n^j=\oplus_{i\in J} \sigma_{i_1}\cdots\sigma_{i_{j-1}}\eta_i \sigma_{i_j}\cdots\sigma_{i_n}$. The differential in the complex is then given by taking the alternating sum (over all $j$) of these $\xi_n^j$.

We may post-compose $C^{\text{aug}}$ with the functor of taking global sections to obtain a complex $\check{C}^{\text{aug}}$ called the \emph{augmented standard complex} of $\Gamma$. We may also consider the complex
$$
C: \bigoplus_{i\in J} \sigma_i\to \bigoplus_{\mathbf{i}\in J^2} \sigma_{\mathbf{i}}\to \bigoplus_{\mathbf{i}\in J^3} \sigma_{\mathbf{i}}\to\cdots
$$
and $\check{C}=\Gamma\circ C$, which we call the \emph{standard complex}. We then have:

\begin{prop} For any $M\in \Ca$, the complex $C^{\text{aug}}(M)$ is exact. Moreover, for $i\geq 0$, the $i$-th cohomology of the complex $\check{C}(M)$ is isomorphic to $R^i\Gamma (M)$.
\end{prop}

\begin{proof} By \cite[Proposition 1.4 \& Theorem 2.2]{Rosenberg} and by the Lemma, it will follow if we prove that the categories $\Ca^{w_i}$ cover the category $\Ca$, meaning that a morphism $g$ in $\Ca$ is an isomorphism if and only if $f_{w_i}^*(g)$ is an isomorphism for all $i\in J$. This is equivalent to saying that $M\in\Ca$ is zero if and only if all its localisations are zero. Working with proj categories instead, suppose $M$ is a graded $R_q$-module such that $S_w^{-1}M=0$ for all $w$. Pick $m\in M$. Then for all $i\in J$, there exists $\mu_i\in P^+$ such that $c_{w_i\mu_i}m=0$. Let $\mu=\sum_i \mu_i$. Then for all $w\in W$, $c_{w\mu}m=0$. But then it follows from the Lemma below that $\Gamma(\A(\lambda+\mu))m=0$ for all $\lambda >>0$. Since $m$ was arbitrary this implies that $M$ is torsion and so zero in Proj$(R_q)$.
\end{proof}

\subsection{Lemma} \emph{Let $\mu\in P^+$. Then for $\lambda>>0$ we have
$$
\sum_{w\in W} \Gamma(\A(\lambda)) c_{w\mu}=\Gamma(\A(\lambda+\mu)).
$$}

\begin{proof} This is proved in \cite[Lemma III.3.3]{LuntsRos} but we reproduce it here. Clearly the left hand side is included in the right hand side, and both sides are finitely generated as $R$-modules, so by Nakayama it's enough to show that the equality holds modulo $\pi$, i.e. that
$$
\sum_{w\in W} H^0(\lambda) \overline{c_{w\mu}}=H^0(\lambda+\mu)
$$
where $H^0(\lambda)$ denotes the global sections of the line bundle $\mathcal{L}(\lambda)$ on the flag variety $G_k/B_k$. We are then in a classical situation, and the equality will follow from the classical fact that the Weyl group translates of the big cell cover the flag variety of $G_k$. The equality was proved over $\mathbb{C}$ in \cite[Lemma 11]{joseph2}. The argument is the same here in positive characteristic, but for completeness we sketch it.

Firstly, since both sides are finite dimensional over $k$, to show equality is to show that the dimensions are equal, and so it will suffice to prove that the equality holds after passing to the algebraic closure of $k$. So without loss of generality, we may assume that $k=\bar{k}$. Moreover, for any $\lambda'$ and $\mu'$, the natural map $H^0(\lambda')\otimes H^0(\mu')\to H^0(\lambda' +\mu')$ is surjective (see \cite[Proposition 14.20]{Jantzen2}). Thus we may assume that $\lambda=n\mu$ for $n>>0$.

Now, consider the Weyl module $V=V(-w_0\mu)=H^0(\mu)^*$ and let $v\in V$ have weight $-w_0\mu$. Then the flag variety $G_k/B_k$ maps onto the $G_k$-orbit of the line $kv$ in the projective space $\mathbb{P}(V)$. If we take the homogeneous cone above this, its algebra of regular functions is a quotient of $S(V^*)$, and in fact is the commutative graded ring $A=\oplus_{n\geq 0} H^0(n\mu)$ (see \cite[Proposition 14.22]{Jantzen2}). The fact that the Weyl group translates of the big cell cover the flag variety now implies that the radical of the ideal of $A$ generated by the elements $\overline{c_{w\mu}}$ is in fact the irrelevant ideal $A_{> 0}$. This says that the ideal they generate contains all $H^0(m\mu)$ for $m$ large enough, as required.
\end{proof}

\subsection{Base change}\label{base_change} As an immediate application of the \v{C}ech complex, we show how the cohomology of $\Gamma$ behaves under base change to the field $L$.

\begin{prop} For any $M\in\Ca$ and any $i\geq 0$, we have $R^i\Gamma(M_L)=R^i\Gamma(M)\otimes_R L$.
\end{prop}

\begin{proof} Let $M\in\Ca$ and first assume that $i=0$. By the universal property of tensor products, we have a commutative diagram
\[
\begin{tikzcd}
M \arrow{r}{g} & M\otimes_R L \\
M^{\B}\arrow{r}\arrow{u}& M^{\B}\otimes_R L\arrow{u}{f}
\end{tikzcd}
\]
of $R$-modules with injective vertical arrows, and we have to show that $\im(f)=(M_L)^{B_q}$. It will be enough show that $(M_L)^{B_q}\subseteq \im(f)$, the other inclusion being clear.

Pick $m\in (M_L)^{B_q}$. Then there is some $a\geq 0$ such that $\pi^am\in\im(g)$, i.e $\pi^am=m'\otimes 1$ for some $m'\in M$. Now, given that $\rho_{M_L}(m)=m\otimes 1$, and since $\rho_{M_L}=\rho_M\otimes_R L$, we see that $\rho_M(m')-m'\otimes 1\in M\otimes_R\B$ is $\pi$-torsion. Hence, there is some $b\geq 0$ such that $\rho_M(\pi^bm')=\pi^bm'\otimes 1$, and thus we get that $\pi^{a+b}m\in\im(f)$. The result now follows since $\im(f)$ is an $L$-vector space.

For $i>0$, using Proposition \ref{Cech}, the case $i=0$ and the fact that $-\otimes_R L$ is exact, we see that $R^i\Gamma(M)\otimes_R L$ is the $i$-th \v{C}ech cohomology group of $M\otimes_R L$. By \cite[Proposition 4.5]{QDmod1} this group is equal to $R^i\Gamma(M\otimes_R L)$.
\end{proof}

\subsection{The ring $\D$}\label{defn_of_Dq} Recall the notation and the definitions from \ref{recap2}. We now define an $R$-form of the ring of quantum differential operators. For $u\in U$, $a\in \A$, $i\geq 0$, we have $u(a)=\sum a_2(u)\cdot a_1\in \A$ since $U\subset U^{\text{res}}$. From this, we can immediately see that $\A$ is a left $U$-module algebra. Hence we may form the smash product $\D=\A\# U$. Note that $\D$ is $\pi$-torsion free as it is equal to $\A\otimes_R U$ as an $R$-module, thus it follows that it is a lattice in $\D_q$.

\begin{prop} The algebra $\D/\pi \D$ is Noetherian. Hence so is $\widehat{\D}$.
\end{prop}

\begin{proof} By the above remarks we see that $\D_k:=\D/\pi\D$ is the smash product algebra of $\A/\pi\A\cong \Of(G_k)$ and $U_k$. By Corollary \ref{intfinite2}, $U_k\cong U(\g_k)\otimes_k k(P/2Q)$, and so $\D_k$ is a finite module over the smash product $\Of(G_k)\# U(\g_k)$. The latter is isomorphic to the ring $\D(G_k)$ of crystalline differential operators on the affine variety $G_k$ and hence is Noetherian. Thus $\D_k$ is Noetherian as required. The last part follows from Proposition \ref{Noetherian}.
\end{proof}

\subsection{$\D$-modules} \label{Verma} We now turn to an $R$-version of the category $\Dflag$. We first introduce the following notation: we let $U^{\geq 0}=U\cap U_q^{\geq 0}$. It is the $R$-subalgebra of $U$ generated by all $E_{\alpha_i}$, all $K_\mu$ ($\mu\in P$) and all $[K_{\alpha_i};0]_{q_i}$. Note that  $U^{\geq 0}$ is a subalgebra of $U^{\text{res}}(\bo)$. Moreover, note that the action (\ref{actiononDq}) restricts to an action of $U^{\text{res}}$ on $\D$ making it into a $U^{\text{res}}$-module algebra. This is because the adjoint action of $U^{\text{res}}$ preserves $U$.

\begin{definition}
Let $\lambda\in T_P^R$. We let $\Da^\lambda$ be the category whose objects are triples $(M, \alpha, \beta)$ where $M$ is an $R$-module, $\alpha:\D\otimes_R M\to M$ is a left $\D$-module action and $\beta: M\to M\otimes_R \B$ is a right $\B$-comodule action. The map $\beta$ induces a left $U^{\text{res}}(\bo)$-action on $M$ which we also denote by $\beta$. These actions must satisfy:
\begin{itemize}
\item[(i)] The $U^{\geq 0}$-actions on $M\otimes_R R_\lambda$ given by $\beta\otimes \lambda$ and $\alpha\vert_{U^{\geq 0}}\otimes 1$ are equal.
\item[(ii)] The map $\alpha$ is $U^{\text{res}}(\bo)$-linear with respect to the $\beta$-action on $M$ and the action (\ref{actiononDq}) on $\D$.
\end{itemize}
We will write $\coh(\Da^\lambda)$ to denote the full subcategory of $\Da^\lambda$ consisting of finitely generated $\D$-modules. 
\end{definition}

There is of course a forgetful functor $\text{forget}:\Da^\lambda\to \Ca$, and given an object $M\in\Da^\lambda$ we let its global sections equal $\Gamma(M)$ where we view $M$ as an object of $\Ca$. By abuse of notation we also denote this global section functor by $\Gamma$. Also the functor $M\mapsto M_L$ described earlier restricts to a functor $\Da^\lambda\to\Dflag$.

Note again that condition (i) above can be rephrased into saying that for $M\in\Da^\lambda$ and $m\in M$, we have $E_\alpha m=\beta(E_\alpha)m$, $K_\mu m=\lambda(K_\mu) \beta(K_\mu)m$, and
$$
[K_{\alpha};0]m=(\lambda([K_{\alpha};0])\beta(K_\alpha)+\lambda(K_\alpha^{-1})\beta([K_{\alpha};0]))m
$$
for all simple roots $\alpha$ and $\mu\in P$. In particular if $m$ is a global section then by $B_q$-invariance we must have $E_\alpha m=0$, $[K_{\alpha};0]m=\lambda([K_{\alpha};0])m$ and $K_\mu m=\lambda(K_\mu)m$. In other words global sections consist of the highest weight vectors of weight $\lambda$. So we see that the $\D$-module homomorphisms $\D\to M$ corresponding to global sections factor through the quotient $\D^\lambda=\D/I$ where $I$ is the left ideal generated by
$$
\{E_{\alpha_i}, K_\mu-\lambda(K_\mu), [K_{\alpha_i};0]-\lambda([K_{\alpha_i};0]): 1\leq i\leq n, \mu\in P\}.
$$
Our aim now is to show that $\D^\lambda\in\Da^\lambda$.

Recall the notation from \ref{Lusztig}. Note that we may define a Verma module $\M_\lambda$ for $U$, namely it is the cyclic $U$-module with generator $v_\lambda$ and relations $E_{\alpha_i}v_\lambda=0$, $K_\mu v_\lambda=\lambda(\mu)v_\lambda$ and $[K_{\alpha_i};0] v_\lambda=\lambda([K_{\alpha_i};0])v_\lambda$. By the triangular decomposition for $U$ and the PBW basis for $U^-$ (see \cite[Sections 4.5-4.6]{Nico2}), we see that $\M_\lambda$ is a free $R$-module with basis given by the monomials
$$
F_{\beta_1}^{r_1}\cdots F_{\beta_N}^{r_N}v_\lambda
$$
and so we also see that it is a lattice in the Verma module $M_\lambda$ for $U_q$. In fact it is the image of $U$ under the canonical surjection $U_q\to M_\lambda$. Recall that the quotient of the adjoint action of $U_q^{\geq 0}$ gave rise to an integrable module structure on $M_\lambda$. Since the adjoint action of $U^{\text{res}}(\bo)$ preserves $U$ (see \cite[Lemma 1.2]{Tanisaki2}), we immediately get:

\begin{lem} The above adjoint $U^{\text{res}}(\bo)$-action on $M_\lambda$ preserves $\M_\lambda$, making it into a $\B$-comodule.
\end{lem}

Now since $\D^\lambda=\A\otimes_R \M_\lambda$ as an $R$-module, we identify it with $p^*(\M_\lambda)\in\Ca$. Just as for $\D_q^\lambda$, we then have that $\D^\lambda$ is in fact an object of $\Da^\lambda$, and our previous discussion shows that it represents the global section functor on $\Da^\lambda$, i.e
$$
\Gamma(M)=\Hom_{\Da^\lambda}(\D^\lambda,M)
$$
for all $M\in\Da^\lambda$.

\subsection{Cohomology of the induction functor mod $\pi$}\label{cohomology} We will need to investigate the cohomology of $M_k:=M/\pi M$ for $M\in\B\text{-comod}$ for the induction functor $\ind:\B\text{-comod}\to\A\text{-comod}$ defined in \cite{Andersen}. Note that $M_k$ is in fact a $\B/\pi\B\cong \Of(B_k)$-comodule, and the global section functor applied to $p^*(M_k)$ coincides with the functor of taking $\Of(B_k)$-coinvariants in $\Of(G_k)\otimes_k M_k$, i.e. with the classical induction functor $\ind_{B_k}^{G_k}M_k$ (c.f. \cite[Proposition 3.7]{Andersen}). We will compare the cohomology groups $R^i\ind(M_k)$ and $R^i\ind_{B_k}^{G_k}(M_k)$.

By \cite[2.17-2.19]{Andersen}, if $M$ is a $\B$-comodule that is free as an $R$-module then it has a resolution
$$
0\to M\to Q_0\to Q_1\to\cdots
$$
in the category of $\B$-comodules, which is $R$-split and such that each $Q_i$ is $R$-free and acyclic. This is called the \emph{standard resolution} of $M$. This construction is completely canonical, so that $M_k$ also has similarly such a resolution in the category of $\B/\pi\B$-comodules, which we also call the standard resolution of $M_k$.

\begin{lem} Suppose $M\in\Ca$ is free as an $R$-module. Then there is a canonical isomorphism $R^i\text{\emph{Ind}}(M_k)\cong R^i\text{\emph{Ind}}_{B_k}^{G_k}(M_k)$ for all $i\geq 0$.
\end{lem}

\begin{proof} Since each $Q_i$ is free, we have a short exact sequence
$$
0\longrightarrow Q_i\overset{\cdot\pi}{\longrightarrow} Q_i\longrightarrow Q_i\otimes_R k\longrightarrow 0.
$$
Applying the long exact sequence and using the fact that $Q_i$ is acyclic, we immediately obtain that $Q_i\otimes_R k$ is also acyclic. Now since the standard resolution is split exact, it follows from the above that
$$
0\to M_k\to Q_0\otimes_R k\to Q_1\otimes_R k\to\cdots
$$
is an acyclic resolution of $M_k$ whose cohomology therefore computes $R\ind(M_k)$. On the other hand this resolution coincides with the standard resolution of $M_k$ by \cite[page 24, after equation (6)]{Andersen}, so computes $R\ind_{B_k}^{G_k}(M_k)$ by \cite[Proposition 3.7]{Andersen}.
\end{proof}

This will be useful because $R\text{Ind}(M_k)\cong R\Gamma (p^*(M_k))$ by Proposition \ref{line}(ii). We now apply the above to Verma modules, viewed as $\B$-comodules via the adjoint action of $U^{\text{res}}(\bo)$. First we recall some well-known generalities.

Recall from \cite[I.5.8 \& Proposition I.5.12]{Jantzen} that if $N$ is a representation of $B_k$, then there is a corresponding $G_k$-equivariant sheaf $\mathcal{L}(N)$ on the flag variety $X_k:=G_k/B_k$ such that $R^i\ind_{B_k}^{G_k}N$ is canonically isomorphic to the sheaf cohomology $H^i(X_k, \mathcal{L}(N))$ for each $i\geq 0$. Moreover, the sheaf $\D_{X_k}^{\lambda}$ of crystalline $\lambda$-twisted differential operators on $X_k$, for $\lambda\in \mathfrak{h}_k^*$, is the $G_k$-equivariant sheaf corresponding to the Verma module $\M(\lambda)=U(\g_k)/(x-\lambda(x) : x\in\bo_k)$ viewed as a $B_k$-representation via the adjoint action (c.f. \cite[pages 12\& 20]{Milicic} in characteristic 0, see also \cite[3.1.3]{BMR} in positive characteristic). Thus we see that the sheaf cohomology of $\D_{X_k}^{\lambda}$ coincides with $R\ind_{B_k}^{G_k}(\M(\lambda))$.

\begin{definition} We set
$$
T_P^k:=\{\gamma\in T_P^R : \gamma(K_\mu)\equiv 1\pmod{\pi} \text{ for all } \mu\in P\}.
$$
This is a subgroup of $T_P^R$ containing $P$. Note that any $\lambda\in T_P^k$ induces an element of $\mathfrak{h}_k^*$ which we also denote by $\lambda$. Hence the corresponding Verma module $\M_\lambda$ satisfies $(\M_\lambda)_k\cong \M(\lambda)$.
\end{definition}

Finally recall that the object $\D^\lambda\in\Ca$ is given by $\D^\lambda=p^*(\M_\lambda)$ and so similarly $\D^\lambda_k=p^*((\M_\lambda)_k)$. Thus we have $R\Gamma (\D^\lambda_k)\cong R\ind((\M_\lambda)_k)$. Putting everything together, we get by the Lemma:

\begin{prop} Let $\lambda\in T_P^k$. Then the cohomology of $\D^\lambda_k$ with respect to $\Gamma$ coincides with the classical sheaf cohomology of twisted differential operators on the flag variety, i.e. there is a canonical isomorphism $R^i\Gamma (\D^\lambda_k)\cong H^i(X_k, \D_{X_k}^{\lambda})$ for every $i\geq 0$.
\end{prop}

The above will allow us to compute the global sections of $\D^\lambda_k$. But first we need to mention some restrictions on the prime $p$.

\begin{definition} Recall that the prime $p$ is said to be \emph{bad} for an irreducible root system $\Phi$ if
\begin{itemize}
\item $p=2$ when $\Phi= B_l, C_l$ or $D_l$;\\
\item $p=2$ or $3$ when $\Phi= E_6, E_7, F_4$ or $G_2$; and\\
\item $p=2, 3$ or $5$ when $\Phi= E_8$.
\end{itemize}
We say that $p$ is \emph{bad} for $U_q$ if it is bad for some irreducible component of the associated root system, and we say that $p$ is \emph{good} if it is not bad. Finally, we say that $p$ is \emph{very good} for $U_q$ if it is a good and no irredicuible component of the root system is of type $A_{mp-1}$ for some integer $m\geq 1$.
\end{definition}

\begin{cor} Let $\lambda\in T_P^k$ and assume that $p$ is a very good prime. Then $\D^\lambda_k$ is $\Gamma$-acyclic and $\Gamma(\D^\lambda_k)\cong U(\g_k)_{\chi_\lambda}$, where $\chi_\lambda$ is the corresponding character of the Harish-Chandra centre of $U(\g_k)$.
\end{cor}

\begin{proof}
This follows from the Proposition and \cite[Proposition 3.4.1]{BMR} (this is where the restrictions on $p$ are required).
\end{proof}

\subsection{Twists of coherent $\D$-modules}\label{twists} Observe that for $\mu\in T_P^R$ and $M\in\Da^\lambda$, the left $\D$-action on $M(\mu)$ makes $M(\mu)$ into an element of $\Da^{\lambda+\mu}$. We investigate those twists.

\begin{prop} Let $\mu\in P^+$ and $\lambda\in T_P^k$. Assume that $p$ is a good prime. Then
$$
R^i\Gamma(\D^\lambda(\mu)_k)=0
$$
for all $i>0$.
\end{prop}

\begin{proof}
We have $\D^\lambda(\mu)_k=p^*((\M_\lambda)_k\otimes_k k_{-\mu})$. Using \cite[II.4.1.(2)]{Jantzen2}, we see that the corresponding $G_k$-equivariant sheaf on the flag variety $X_k$ is $\D_{X_k}^{\lambda}\otimes_{\Of_{X_k}}\mathcal{L}(\mu)$. Thus we are reduced to showing that
\begin{equation}\label{neweqn}
H^i\left(X_k, \D_{X_k}^{\lambda}\otimes_{\Of_{X_k}}\mathcal{L}(\mu)\right)=0
\end{equation}
for all $i>0$.

Now consider the filtration on $\D_{X_k}^{\lambda}$ by degree of differential operators, which naturally induces a filtration on $\D_{X_k}^{\lambda}\otimes_{\Of_{X_k}}\mathcal{L}(\mu)$ by
$$
F_i  \left(\D_{X_k}^{\lambda}\otimes_{\Of_{X_k}}\mathcal{L}(\mu)\right)=F_i (\D_{X_k}^{\lambda})\otimes_{\Of_{X_k}}\mathcal{L}(\mu).
$$
Let $\tau:T^*X_k\to X_k$ be the cotangent bundle. Since $\mathcal{L}(\mu)$ is locally free, the corresponding associated graded is
$$
\gr \left(\D_{X_k}^{\lambda}\otimes_{\Of_{X_k}}\mathcal{L}(\mu)\right)\cong \gr (\D_{X_k}^{\lambda})\otimes_{\Of_{X_k}}\mathcal{L}(\mu)\cong \tau_*\Of_{T^*X_k}\otimes_{\Of_{X_k}}\mathcal{L}(\mu).
$$
Next, by \cite[0.5.4.10]{EGAI},
$$
\tau_*\Of_{T^*X_k}\otimes_{\Of_{X_k}}\mathcal{L}(\mu)\cong \tau_*\left(\Of_{T^*X_k}\otimes_{\Of_{T^*X_k}}\tau^*\mathcal{L}(\mu)\right)\cong \tau_*\tau^*\mathcal{L}(\mu).
$$
Moreover, because $\tau$ is an affine morphism, we get from \cite[Cor. I.3.3]{EGAIII.1} that
$$
H^i\left(X_k, \tau_*\tau^*\mathcal{L}(\mu)\right)\cong H^i\left(T^*X_k, \tau^*\mathcal{L}(\mu)\right)
$$
for all $i\geq 0$. Finally, under the assumption that $p$ is a good prime, it was shown in \cite[Theorem 2]{KLT} that $H^i\left(T^*X_k, \tau^*\mathcal{L}(\mu)\right)=0$ for all $i>0$.

Putting everything together, we have obtained that $\gr \left(\D_{X_k}^{\lambda}\otimes_{\Of_{X_k}}\mathcal{L}(\mu)\right)$ is $\Gamma$-acyclic. Now, since $X_k$ is Noetherian, cohomology commutes with direct limits by \cite[Proposition III.2.9]{Hartshorne}, and so each homogeneous component $\gr_i \left(\D_{X_k}^{\lambda}\otimes_{\Of_{X_k}}\mathcal{L}(\mu)\right)$ is $\Gamma$-acyclic, and hence each filtered piece $F_i  \left(\D_{X_k}^{\lambda}\otimes_{\Of_{X_k}}\mathcal{L}(\mu)\right)$ is $\Gamma$-acyclic as well. Therefore (\ref{neweqn}) holds as required since $\D_{X_k}^{\lambda}\otimes_{\Of_{X_k}}\mathcal{L}(\mu)$ is the direct limit of the $F_i  \left(\D_{X_k}^{\lambda}\otimes_{\Of_{X_k}}\mathcal{L}(\mu)\right)$. 
\end{proof}

\begin{cor} Assume $p$ is a good prime. Then for $\mu\in P^+$ and for any $n\geq 1$, we have that $\D^\lambda(\mu)/\pi^n \D^\lambda(\mu)$ is $\Gamma$-acyclic.
\end{cor}

\begin{proof} We proceed by induction on $n$. The case $n=1$ is just the previous Proposition. Now for $n\geq 1$, we have a short exact sequence
$$
0\to \D^\lambda(\mu)/\pi \D^\lambda(\mu)\to \D^\lambda(\mu)/\pi^{n+1} \D^\lambda(\mu)\to \D^\lambda(\mu)/\pi^n \D^\lambda(\mu)\to 0
$$
where by the Proposition and by induction hypothesis, the two side terms are acyclic. Hence by the long exact sequence the middle term is acyclic.
\end{proof}

As a consequence of this we can obtain a $D$-modules version of Proposition \ref{gen_by_global_sections} which will be useful to us later. We first need a lemma:

\begin{lem} Let $M\in\text{\emph{coh}}(\Da^\lambda)$. Then there is an $\A$-submodule $N$ of $M$ such that $N\in\text{\emph{coh}}(\Ca)$ and $N$ generates $M$ as a $\D$-module.
\end{lem}

\begin{proof}
Let $m_1,\ldots, m_n$ be a generating set for $M$ as a $\D$-module. Viewing $M$ as an object of $\Ca$, we simply let $N$ be the smallest coherent submodule of $M$ containing $m_1, \ldots, m_n$, as given by Lemma \ref{useful_lemma}.
\end{proof}

\begin{thm} Let $M\in\text{\emph{coh}}(\Da^\lambda)$. Then $M(\mu)$ is generated by finitely many global sections for $\mu>>0$. Moreover, if $\pi M=0$ and if $p$ is a good prime, then $M(\mu)$ is also $\Gamma$-acyclic for $\mu>>0$.
\end{thm}

\begin{proof}
Let $N\in \coh(\Ca)$ be as in the previous lemma. Note that $M(\mu)\in\coh(\Da_n^{\lambda+\mu})$ for any $\mu$. By Proposition \ref{gen_by_global_sections} we see that $N(\mu)$ is generated by finitely many global sections for $\mu>>0$. Since $M$ is generated by $N$ as a $\D$-module, the first claim follows.

Now assume $\pi M=0$. Fix any $\mu_1$ such that $M(\mu_1)$ is generated by its global sections. Then we have a surjection $(\D^{\lambda+\mu_1})^a\to M(\mu_1)$ which in fact factors through a surjection $f_1:(\D^{\lambda+\mu_1}_k)^a\to M(\mu_1)$. Let $K=\ker{f_1}$. Note that $K\in\coh(\Da^{\lambda+\mu_1})$ by Proposition \ref{defn_of_Dq} and that $\pi K=0$. So by the above argument applied to $K$, we can find $\mu_2>>0$ and a surjection $f_2:(\D^{\lambda+\mu_1+\mu_2}_k)^b\to K(\mu_2)$. Carrying on we obtain $\mu_1, \ldots, \mu_N\in P^+$ and a resolution in $\coh(\Da^{\lambda+\mu})$
$$
F_N\overset{f_N}{\rightarrow}\cdots\overset{f_2}{\rightarrow}F_1\overset{f_1}{\rightarrow} M(\mu)\to 0
$$
where $\mu=\sum_{j=1}^N \mu_i$ and for $1\leq i\leq N$, $F_i$ is a direct sum of finitely many copies of modules of the form $\D^{\lambda+\mu_1+\ldots+\mu_i}_k(\mu_{i+1}+\ldots+\mu_{N})$. Note that all the $F_i$ are $\Gamma$-acyclic by Proposition \ref{twists}. Write $K_0=M(\mu)$ and $K_i=\ker{f_i}$ for $1\leq i\leq N$. Then for each $1\leq i\leq N$ we have a short exact sequence
$$
0\to K_i\to F_i\to K_{i-1}\to 0
$$
Since $F_i$ is acyclic the long exact sequence implies that $R^j\Gamma(K_{i-1})\cong R^{j+1}\Gamma(K_i)$ for all $j\geq 1$. Thus we obtain
$$
R^j\Gamma(M(\mu))\cong R^{j+1}\Gamma(K_1)\cong R^{j+2}\Gamma(K_2)\cong\ldots\cong R^{j+N}\Gamma(K_N)=0
$$
for any $j\geq 1$ as required.
\end{proof}

\begin{remark} We expect the above result to hold for all modules, not just for those killed by $\pi$.
\end{remark}

\section{Banach $\OqBhat$-comodules}

In this Section, we define various categories of comodules over certain $\pi$-adically complete or Banach coalgebras. In doing so, we will often use techniques to do with topologies on tensor products, and so we begin by establishing the necessary facts on this topic.

\subsection{Completed tensor products and Banach Hopf algebras}\label{funal}

Recall from \cite[Section 17B]{NFA} that given two seminorms $p$ and $p'$ on the vector spaces $V$ and $W$ respectively, the \emph{tensor product seminorm} $p\otimes p'$ on $V\otimes_L W$ is defined in the following way: for $x\in V\otimes_L W$, we have
$$
p\otimes p'(x):=\inf\Big\{\max_{1\leq i\leq r}p(v_i)\cdot p'(w_i) : x=\sum_{i=1}^r v_i\otimes w_i, v_i\in V, w_i\in W\Big\}.
$$
If $V$ and $W$ are locally convex spaces, then we will always only consider the projective tensor topology on $V\otimes_LW$, i.e the topology obtained via these tensor product seminorms. One can then construct the Hausdorff completion $\ten{V}{W}$ of this space, which we call the completed tensor product of $V$ and $W$. Note that this construction is functorial, so that two continuous linear maps $f: V\to W$ and $g:X\to Y$ induce a continuous linear map $f\widehat{\otimes}g:\ten{V}{X}\to\ten{W}{Y}$. In general, if $V$ and $W$ are Hausdorff, so is $V\otimes_L W$. When $V$ and $W$ are Banach spaces, so is $\ten{V}{W}$, and $\widehat{\otimes}_L$ is a monoidal structure on the category of $L$-Banach spaces.

Given an $L$-vector space $V$ and an $R$-lattice $V^\circ\subset V$, we may define a norm on $V$ called the \emph{gauge norm}, given by
$$
\norm{v}_{\text{gauge}}=\inf_{\substack{a\in L\\ v\in a V^\circ}}\abs{a}.
$$
This infimum simply equals $\abs{\pi^n}$ where $n\in\Z$ is the largest integer such that $v\in \pi^nV^\circ$, hence the topology induced by the gauge norm is the topology induced by the $\pi$-adic filtration on $V^\circ$. Recall then that if $V$ is a normed $L$-vector space, then its norm is equivalent to the gauge norm associated to the unit ball $V^\circ$, see \cite[Lemma 2.2]{NFA}. Hence, without loss of generality, we will always assume that our normed vector spaces are equipped with the $\pi$-adic norm induced from their unit balls. Moreover, recall that given two normed $L$-vector spaces $V$ and $W$ with unit balls $V^\circ$ and $W^\circ$, the unit ball of $V\otimes_L W$ equipped with the tensor product norm as above is $V^\circ\otimes_R W^\circ$, see \cite[Lemma 2.2]{Andreas1}. This is a fact we will often use without further mention.

Recall that a bounded linear map $f:X\to Y$ between two $L$-locally convex spaces is called \emph{strict} if it induces a topological isomorphism
$$
\hat{f}:X/\ker{f}\to\im{f}.
$$
The following result says that strict maps behave well under tensor products:

\begin{lem} Suppose that $V$ is a vector subspace of a locally convex space $W$ equipped with the subspace topology and let $U$ be any other locally convex space. Then
\begin{enumerate}
\item the canonical maps $V\otimes_L U\to W\otimes_L U$ and $U\otimes_L V\to U\otimes_L W$ are strict embeddings where we give the left hand side the tensor product topology;
\item the canonical maps $\ten{V}{U}\to\ten{W}{U}$ and $\ten{U}{V}\to\ten{U}{W}$ are strict embeddings;
\item the functor of taking tensor product with $U$, in the category of locally convex spaces, preserves strict surjections.
\end{enumerate}
\end{lem}

\begin{proof}
(i) The map $V\otimes_L U\to W\otimes_L U$ is clearly injective and by \cite[Proposition 17.4.iii]{NFA} we see that it is isometric, hence an isomorphism onto its image.

(ii) This follows from (i) and \cite[1.1.9 Cor 6]{BGR}.

(iii) This follows immediately from the proof of \cite[Appendix A, Lemma A.34]{KreBB}.
\end{proof}

\begin{remark} From now on, given any $U,V,W$ as in the Lemma, we shall not distinguish between $\ten{V}{U}$ and the subspace of $\ten{W}{U}$ isomorphic to it. 
\end{remark}

Next we turn to Banach coalgebras and Hopf algebras.

\begin{definition} An \emph{$L$-Banach coalgebra} is a coalgebra object in the monoidal category of $L$-Banach spaces. In other words it is a Banach space $C$ equipped with continuous linear maps $\Delta: C\to \ten{C}{C}$ and $\varepsilon: C\to L$ which satisfy the usual axioms:
$$
(\Delta\widehat{\otimes}\id_C)\circ\Delta=(\id_C\widehat{\otimes}\Delta)\circ\Delta, \quad (\id_C\widehat{\otimes}\varepsilon)\circ\Delta=(\varepsilon\widehat{\otimes}\id_C)\circ\Delta=\id_C.
$$
A morphism of coalgebras $f:C\to D$ is a continuous linear map such that $\varepsilon_D\circ f=\varepsilon_C$ and $(f\widehat{\otimes} f)\circ \Delta_C=\Delta_D\circ f$.

Given a Banach coalgebra $C$ as above, a \emph{Banach $C$-comodule} is a Banach space $\M$ equipped with a continuous linear map $\rho_{\M}:M\to \ten{\M}{C}$, which satisfies:
$$
(\id_{\M}\widehat{\otimes}\Delta)\circ \rho_{\M}=(\rho_{\M}\widehat{\otimes}\id_C)\circ\rho_{\M}, \quad (\id_{\M}\widehat{\otimes}\varepsilon)\circ\rho_{\M}=\id_{\M}.
$$
A morphism of comodules $f:\M\to \Nn$ is then a continous linear map such that $\rho_{\Nn}\circ f=\rho_{\M}\circ(f\widehat{\otimes}\id_C)$. We denote by $\Comod{C}$ the category of Banach $C$-comodules.

An \emph{$L$-Banach Hopf algebra} is an $L$-Banach algebra $H$ which is also a coalgebra such that $\Delta$ and $\varepsilon$ are algebra homomorphisms, and furthemore $H$ is equipped with a continuous linear map $S:H\to H$, which satisfies
$$
m\circ (S\widehat{\otimes}\id_H)\circ\Delta=\iota\circ\varepsilon=m\circ (\id_H\widehat{\otimes}S)\circ\Delta
$$
where $m:\ten{H}{H}\to H$ and $\iota:L\to H$ denote the multiplication map and the unit in $H$ respectively. A morphism of Hopf algebras $f:H\to S$ is then a continuous algebra homomorphism which is also a morphism of coalgebras, such that $S_D\circ f=f\circ S_H$.
\end{definition}

We showed in \cite[Proposition 3.3]{Nico2} that for any two $R$-modules $M$ and $N$, there is a canonical isomorphism of Banach spaces $\ten{\widehat{M}_L}{\widehat{N}_L}\cong (\widehat{M\otimes_R N})_L$. In particular this implies that if $H$ is an $R$-Hopf algebra, then $\widehat{H}_L$ is a Banach Hopf algebra. Indeed, the maps $\varepsilon$ and $S$ extend to the completion, and the comultiplication gives rise to a map $\widehat{\Delta}:\widehat{H}_L\to(\widehat{H\otimes_R H})_L\cong\ten{\widehat{H}_L}{\widehat{H}_L}$. These satisfy the Hopf algebra axioms since they do on the dense subset $H_L$. So in particular we see that $\Oqhat:=\widehat{\A}\otimes_R L$ and $\OqBhat:=\widehat{\B}\otimes_R L$ are Banach Hopf algebras.

\subsection{$\widehat{\B}$-comodules}\label{comod2} We now define a suitable version of comodules over $\widehat{\B}$.

\begin{notation} Given two $R$-modules $M$ and $N$, we write $\tenR{M}{N}$ to denote the $\pi$-adic completion $\widehat{M\otimes_R N}$ of $M\otimes_R N$. This construction satisfies the usual associativity and additivity properties of tensor products, and is functorial.
\end{notation}

\begin{definition} A $\widehat{\B}$-comodule is a $\pi$-adically complete $R$-module $\M$ equipped with a map $\rho: \M\to \tenR{\M}{\B}$ such that
$$
(\rho\widehat{\otimes}1)\circ \rho=(1\widehat{\otimes}\Delta)\circ \rho, \quad\text{and}\quad (1\widehat{\otimes}\varepsilon)\circ \rho=1_{\M}.
$$
A morphism of $\widehat{\B}$-comodules is an $R$-module map $f:\M\to \Nn$ such that $(f\widehat{\otimes}1)\circ \rho_{\M}=\rho_{\Nn}\circ f$. We denote the set of comodule morphisms $\M\to\Nn$ by $\Hom_{\widehat{\B}}(\M,\Nn)$.
\end{definition}

\begin{lem}
Suppose that $\M$ is a $\widehat{\B}$-comodule. Then $\M/\pi^n\M$ is a $\B$-comodule for every $n\geq 1$. Hence $\M$ is a $\widehat{U^{\text{res}}(\bo)}$-module and, moreover, if $\rho_n$ denotes the $\B$-comodule map on $\M/\pi^n\M$ and $\rho$ denotes the $\widehat{\B}$-comodule map on $\M$, then $\rho=\varprojlim \rho_n$.
\end{lem}

\begin{proof}
There are isomorphisms
$$
(\tenR{\M}{\B})/\pi^n(\tenR{\M}{\B})\cong (\M\otimes_R \B)/\pi^n(\M\otimes_R\B)\cong (\M/\pi^n\M)\otimes_R \B,
$$
and hence $\rho: \M\to \tenR{\M}{\B}$ induces a map
$$
\rho_n: \M/\pi^n\M\to (\M/\pi^n\M)\otimes_R \B
$$
for every $n\geq 1$. The comodule axioms are satisfied since they are obtained by reducing the equalities
$$
(\rho\widehat{\otimes}1)\circ \rho=(1\widehat{\otimes}\Delta)\circ \rho, \quad\text{and}\quad (1\widehat{\otimes}\varepsilon)\circ \rho=1_{\M}
$$
modulo $\pi^n$. Hence $\M/\pi^n\M$ is a $U^{\text{res}}(\bo)$-module and even a $U^{\text{res}}(\bo)/\pi^nU^{\text{res}}(\bo)$-module, and the structures are compatible with the maps $\M/\pi^{n+1}\M\to\M/\pi^n\M$. Taking inverse limits we see that $\M$ is a $\widehat{U^{\text{res}}(\bo)}$-module. The last part is immediate since $\varprojlim \rho_n=\widehat{\rho}=\rho$ as $\M$ is $\pi$-adically complete.
\end{proof}

\begin{cor}
For any two $\widehat{\B}$-comodules $\M$ and $\Nn$, there is a canonical isomorphism $\text{\emph{Hom}}_{\widehat{\B}}(\M,\Nn)\cong \varprojlim\text{\emph{Hom}}_{\B}(\M/\pi^n\M,\Nn/\pi^n\Nn)$. Moreover every $\widehat{\B}$-comodule homomorphism is $\widehat{U^{\text{res}}(\bo)}$-linear.
\end{cor}

\begin{proof}
Given a $\widehat{\B}$-comodule homomorphism $f:\M\to \Nn$, the induced map $f_n:\M/\pi^n\M\to\Nn/\pi^n\Nn$ is a $\B$-comodule map for every $n\geq 1$: since $f$ is a comodule homomorphism we have that $\rho_{\Nn}\circ f=\rho_\M\circ(f\widehat{\otimes} 1)$, which gives that $f_n$ is a comodule homomorphism by reducing modulo $\pi^n$. Moreover the maps $f_n$ uniquely determine $f$ since $f=\varprojlim f_n$. Hence this implies that $f$ is $\widehat{U^{\text{res}}(\bo)}$-linear since the maps $f_n$ are all $U^{\text{res}}(\bo)$-linear. All this defines an injective map
$$
\Hom_{\widehat{\B}}(\M,\Nn)\to \varprojlim\Hom_{\B}(\M/\pi^n\M,\Nn/\pi^n\Nn)
$$
and we need to check that it is surjective. But given an inverse system of maps $f_n:\M/\pi^n\M\to\Nn/\pi^n\Nn$, passing to the inverse limit gives rise to a map $f:\M\to\Nn$ which is a comodule homomorphism since the axioms are satisfied modulo $\pi^n$ for every $n\geq 1$.
\end{proof}

\subsection{Topologically integrable $\widehat{U^{\text{res}}(\bo)}$-module}\label{integrable} We now start preparing for an equivalent notion to the notion of $\widehat{\B}$-comodules. Note that by Lemma \ref{comod2}, if $\M$ is a $\widehat{\B}$-comodule, then $\M/\pi^n\M$ is an integrable $U^{\text{res}}(\bo)$-module. We want an analogous notion of integrable modules at this $\pi$-adically complete level.

\begin{definition}
Let $\M$ be a $\pi$-adically complete $\widehat{U^{\text{res}}(\bo)}$-module. Given $\lambda\in P$, we define the $\lambda$-weight space $\M_\lambda$ to be the corresponding weight space of $\M$ viewing it as a $U^{\text{res}}(\bo)$-module. We say that $\M$ is \emph{topologically integrable} as a $\widehat{U^{\text{res}}(\bo)}$-module if:
\begin{enumerate}
\item $\M$ is  \emph{topologically $(U^{\text{res}})^0$-semisimple}, i.e for every $m\in M$ there exists a family $(m_\lambda)_{\lambda\in P}$ such that $m_\lambda\in\M_\lambda$ and $\sum_{\lambda\in P}m_\lambda$ converges to $m$ in $M$; and
\item for every $i$ the action of $E_{\alpha_i}$ on $M$ is \emph{locally topologically nilpotent}, i.e for every $m\in \M$ the sequence $E_{\alpha_i}^{(r)}\cdot m\rightarrow 0$ as $r\to \infty$.
\end{enumerate}
\end{definition}

\begin{prop} Let $M$ be a $U^{\text{res}}(\bo)$-module and let $\M$ be a $\widehat{U^{\text{res}}(\bo)}$-module. Then:\begin{enumerate}
\item if $\M$ is topologically integrable, then it has a canonical $\widehat{\B}$-comodule structure; and
\item if $M$ is integrable, then $\widehat{M}$ is a topologically integrable $\widehat{U^{\text{res}}(\bo)}$-module.
\end{enumerate}
\end{prop}

\begin{proof}
For (i), note that it follows immediately from the definition of topologically integrable $\widehat{U^{\text{res}}(\bo)}$-module that $\M/\pi^n\M$ is integrable as a $U^{\text{res}}(\bo)$-module for every $n\geq 1$. So there are comodule maps
$$
\rho_n:\M/\pi^n\M\to \M/\pi^n\M\otimes_R\B\cong (\M\otimes_R\B)/\pi^n(\M\otimes_R\B)
$$
for every $n\geq 1$, which are compatible with the maps $\M/\pi^{a+1}\M\to\M/\pi^a\M$. Taking inverse limits gives a map
$$
\rho: \M\to \tenR{\M}{\B}
$$
which gives a comodule structure to $\M$: the comodule axioms hold modulo $\pi^n$ for every $n\geq 1$ so hold for $\rho$. The module structure arising from $\rho$ agrees by definition with the initial module structure on $\M$.

For (ii), let $m\in \widehat{M}$. Then there exists $m_0, m_1, \ldots$ in $M$ such that
$$
m=\sum_{i=0}^{\infty} \pi^im_i.
$$
Now, as $M$ is integrable, we can find ascending chain of finite subsets $S_j\subseteq P$ such that $m_j=\sum_{\lambda\in S_j}m_{j,\lambda}$ for some $m_{j,\lambda}\in M_\lambda$. Let $S=\bigcup_{j\geq 0} S_j$. For each $\lambda\in S$, let
$$
n(\lambda)=\inf \{j : \lambda\in S_j\}.
$$
Then set
$$
m_\lambda=\sum_{j\geq n(\lambda)}\pi^j m_{j,\lambda}\in \pi^{n(\lambda)}\widehat{M}_\lambda.
$$
Since each set $S_j$ is finite, each set $\{\lambda : n(\lambda)<j\}$ is also finite and so $\sum_{\lambda\in S}m_\lambda$ converges to $m$.

Finally, write $m=\sum_{j\geq 0} \pi^jm_j$ again, and pick $N\in\N$. Since $M$ is integrable, for every $0\leq j< N$, $E_{\alpha_i}^{(r)}m_j=0$ for $r>>0$. So there exists $R>0$ such that for all $r>R$ and for any $0\leq j<N$, $E_{\alpha_i}^{(r)}m_j=0$. Then we have
$$
E_{\alpha_i}^{(r)}m=\sum_{j\geq N} \pi^j E_{\alpha_i}^{(r)}m_j\in \pi^N\widehat{M}
$$
for $r>R$. So $E_{\alpha_i}^{(r)}m\to 0$ as $r\to\infty$ as required.
\end{proof}

We aim to show a converse to Proposition \ref{integrable}(i). Similarly to the uncompleted situation, it boils down to showing that closed submodules of topologically integrable modules are topologically integrable. We are only able to do this for torsion-free modules, but this is sufficient for our needs.

\begin{definition} A Banach $\widehat{U^{\text{res}}(\bo)}_L$-module $\M$ is called \emph{topologically integrable} if its unit ball $\M^\circ$ is a topologically integrable $\widehat{U^{\text{res}}(\bo)}$-module.
\end{definition}

Note that topologically integrable $\widehat{U^{\text{res}}(\bo)}_L$-modules are automatically topologically semisimple as Banach $\widehat{(U^{\text{res}})^0_L}$-modules, in the sense of \cite[Section 5.1]{Nico2}. Thus we have:

\begin{thm}[{\cite[Theorem 5.1]{Nico2}}]
Suppose that $\M$ is a topologically integrable $\widehat{U^{\text{res}}(\bo)}_L$-module. Then for each $m\in \M $, there exists a unique family $(m_\lambda)_{\lambda\in P}$ with $m_\lambda\in\M_\lambda$ such that $\sum_{\lambda\in P}m_\lambda$ converges to $m$. Moreover, if $m\in \Nn$ where $\Nn$ is a closed $U_q^0$-invariant subspace, then each $m_\lambda\in \Nn$.
\end{thm}

\subsection{An equivalence of categories}\label{equiv}

We now use above results to obtain a description of the category of Banach $\OqBhat$-comodules.

\begin{prop} Let $\M$ be a $\pi$-torsion free $\widehat{\B}$-comodule. Then $\M$ is a topologically integrable $\widehat{U^{\text{res}}(\bo)}$-module.
\end{prop}

\begin{proof}
We have the comodule map $\M\to \tenR{\M}{\B}$ which is a split injection. As it is split, we must have
$$
\pi^n(\tenR{\M}{\B})\cap \rho(\M)=\pi^n\rho(\M)=\rho(\pi^n \M)
$$
so that $\rho$ is in fact an isometry with respect to the $\pi$-adic norms. Moreover, $\rho$ is a comodule homomorphism if we give $\tenR{\M}{\B}$ the comodule map $1\widehat{\otimes}\Delta$, c.f. Remark \ref{comodcat}.  Hence this gives rise to a $\widehat{U^{\text{res}}(\bo)}_L$-linear isometry $\M_L\to \ten{\M_L}{\OqBhat}$ by Corollary \ref{comod2}.

Note that $\tenR{\M}{\B}$ is the $\pi$-adic completion of $\M\otimes_R \B$, which is a $\B$-comodule via $1\otimes\Delta$. Since $\B$-comodules are integrable $U^{\text{res}}(\bo)$-modules by Theorem \ref{comodcat}, it follows from Proposition \ref{integrable}(ii) that $\tenR{\M}{\B}$ is topologically integrable, hence so is $\ten{\M_L}{\OqBhat}$. Now we identify $\M$ with its image in $\tenR{\M}{\B}=(\ten{\M_L}{\OqBhat})^\circ$. Since the map was an isometry we also have $\M=\M_L\cap \tenR{\M}{\B}$. Pick $m\in \M$. Then inside $\ten{\M_L}{\OqBhat}$ we automatically have $m=\sum_{\lambda\in P}m_\lambda$ and $E_{\alpha_i}^{(r)}m\to 0$ as $r\to \infty$. So we just need to check that each $m_\lambda\in\M$. However by Theorem \ref{integrable}, the $m_\lambda$ are uniquely determined by $m$ and must belong to $\M_L$ since it is complete hence closed in $\ten{\M_L}{\OqBhat}$. On the other hand, since $\tenR{\M}{\B}$ is topologically integrable and $m\in\M\subset\tenR{\M}{\B}$ we must have $m_\lambda\in \tenR{\M}{\B}$ for all $\lambda$. Therefore each $m_\lambda\in (\tenR{\M}{\B})\cap\M_L=\M$ as required.
\end{proof}

Note that by Proposition \ref{integrable}(i) there is a canonical functor between the category of topologically integrable $\widehat{U^{\text{res}}(\bo)}$-modules and $\widehat{\B}$-comodules. Indeed, given a module map $f:\M\to\Nn$ its restriction modulo $\pi^n$ is a module map between two integrable $U^{\text{res}}(\bo)$-modules by definition, hence is a comodule homomorphism. Passing to the inverse limit, $f$ is a $\widehat{\B}$-comodule homomorphism.

\begin{cor} The canonical functor between the category of topologically integrable $\widehat{U^{\text{res}}(\bo)}$-modules and the category of $\widehat{\B}$-comodules restricts to an equivalence of categories between the full subcategories of $\pi$-torsion free objects.
\end{cor}

\begin{proof}
By Proposition \ref{equiv}, the restriction of the functor to the torsion-free modules is essentially surjective. It is evidently faithful. Moreover, it is full by Corollary \ref{comod2}.
\end{proof}

If $\M$ is a topologically integrable $\widehat{U^{\text{res}}(\bo)}_L$-module then we may apply the above functor to its unit ball and extend scalars to construct a functor to $\Comod{\OqBhat}$. This gives our promised Theorem \ref{thmA}:

\begin{proof}[Proof of Theorem \ref{thmA}]
By the proof of the above Corollary, this functor is full and faithful so that we just need to show that it is essentially surjective. Now suppose that $\Nn$ is a Banach $\OqBhat$-comodule. Then there is a split injection $\rho:\Nn\to\ten{\Nn}{\OqBhat}$ which is therefore strict by the Lemma below. Moreover $\rho$ is a comodule homomorphism where we give the right hand side the comodule map $1\widehat{\otimes}\widehat{\Delta}$. Hence $\Nn$ is topologically isomorphic to a subcomodule $\M$ of $\ten{\Nn}{\OqBhat}$, where $\M$ is equipped with the subspace topology. We note that since $1\widehat{\otimes}\widehat{\Delta}$ has norm $\leq 1$, so does its restriction to $\M$, and so it preserves unit balls. Thus we see that $\M^\circ$ is a $\widehat{\B}$-comodule, and therefore is a topologically integrable $\widehat{U^{\text{res}}(\bo)}$-module by Proposition \ref{equiv}. So we have that $\M$ is in the image of our functor.
\end{proof}

\begin{lem} If $X$ and $Y$ are two $L$-Banach spaces and $f:X\to Y$ is a split continuous linear map, then $f$ is strict.
\end{lem}

\begin{proof}
Suppose the splitting is given by $g:Y\to X$. Then we have
$$
\norm{x}_X=\norm{g(f(x))}_X\leq \norm{g}\norm{f(x)}_Y
$$
for all $x\in X$. This implies that $f$ is strict by \cite[Lemma 1.1.9/2]{BGR}.
\end{proof}

\section{Analytic quantum flag varieties}

We now finally turn to analogues of our quantum flag varieties which consist of objects which are $\pi$-adically complete $R$-modules or $L$-Banach spaces. In doing so, we will review the concept of a quasi-abelian category, which will give us the homological tools that we will require later.

\subsection{The category $\widehat{\Ca}$}\label{Ca}

Given two $\widehat{\B}$-comodules $\M, \Nn$ we can define a tensor comodule structure on $\tenR{\M}{\Nn}$: there is an $R$-module map
$$
\M\otimes_R \B\otimes_R \Nn\otimes_R \B \to \M\otimes_R \Nn\otimes_R \B
$$
given by swapping the two middle terms first and then multiplying in $\B$. Pre-composing the $\pi$-adic completion of this map with $\rho_{\M}\widehat{\otimes}\rho_{\Nn}$ gives a map
$$
\tenR{\M}{\Nn}\to\tenR{\tenR{\M}{\Nn}}{\B}
$$
which is easily checked to satisfy the comodule axioms. Indeed since the counit and comultiplication are algebra homomorphisms it's enough to check that
$$
(\rho_{\M}\widehat{\otimes}1\widehat{\otimes}\rho_{\Nn}\widehat{\otimes}1)\circ(\rho_{\M}\widehat{\otimes}\rho_{\Nn})=(1\widehat{\otimes}\Delta\widehat{\otimes}1\widehat{\otimes}\Delta)\circ(\rho_{\M}\widehat{\otimes}\rho_{\Nn}) \text{ and }(1\widehat{\otimes}\varepsilon\widehat{\otimes}\varepsilon)\circ(\rho_{\M}\widehat{\otimes}\rho_{\Nn})=1
$$
which both follow since $\M$ and $\Nn$ are comodules.

\begin{definition} The category $\widehat{\Ca}$ denotes the category whose objects are $\pi$-adically complete $\widehat{\A}$-modules and $\widehat{\B}$-comodules $\M$ such that the action map $\tenR{\widehat{\A}}{\M}\to \M$ is a comodule homomorphism with $\tenR{\widehat{\A}}{\M}$ being given the tensor comodule described above. Morphisms in $\widehat{\Ca}$ are $R$-linear maps preserving both structures. We say that $\M\in\widehat{\Ca}$ is \emph{coherent} if it is finitely generated over $\widehat{\A}$ and denote the full subcategory of coherent modules by $\coh(\widehat{\Ca})$.

Given $\M\in\widehat{\Ca}$, its \emph{global sections} are defined to be $\Gamma(\M):=\Hom_{\widehat{\Ca}}(\widehat{\A},\M)$. The map $\Gamma(\M)\to \M$ defined by $f\mapsto f(1)$ gives an isomorphism between $\Gamma(\M)$ and the $R$-module $\{m\in \M : \rho_{\M}(m)=m\otimes 1\}$.
\end{definition}

\begin{lem} Suppose $\M\in\widehat{\Ca}$. Then $\M/\pi^n \M\in\Ca$ for all $n\geq 0$ and is coherent if $\M$ was coherent.
\end{lem}

\begin{proof}
Clearly $\M/\pi^n \M$ is an $\widehat{\A}/\pi^n\widehat{\A}\cong \A/\pi^n\A$-module and so an $\A$-module, and it is finitely generated if $\M$ was finitely generated. Moreover we know that $\M/\pi^n \M$ is a $\B$-comodule by Lemma \ref{comod2}. So we just need to show that the $\A$-action map is a comodule homomorphism, i.e that for all $n\geq 1$, the diagram
\[
\begin{tikzcd}
\A\otimes_R\M/\pi^n\M \arrow{r} \arrow{d} & \M/\pi^n\M \arrow{d}{\rho_n}\\
\A\otimes_R\M/\pi^n\M\otimes_R \B \arrow{r} & \M/\pi^n\M\otimes_R\B
\end{tikzcd}
\]
commutes. But this follows by tensoring with $R/\pi^n R$ the diagram
\[
\begin{tikzcd}
\tenR{\A}{\M} \arrow{r} \arrow{d} & \M \arrow{d}{\rho}\\
\tenR{\tenR{\A}{\M}}{\B} \arrow{r} & \tenR{\M}{\B}
\end{tikzcd}
\]
which commutes since $\M\in\widehat{\Ca}$.
\end{proof}

\subsection{The functor $M\mapsto \widehat{M}$}\label{completion1} We will now see how to construct objects in $\widehat{\Ca}$ from those in $\Ca$.

\begin{lem} The assignment $M\mapsto\widehat{M}$ defines a functor $\Ca\to \widehat{\Ca}$.
\end{lem}

\begin{proof}
Suppose $M\in\Ca$ and let $\rho:M\to M\otimes_R \B$ denote the comodule map of $M$. Then clearly $\pi^n M$ is an integrable $U^{\text{res}}(\bo)$-submodule, hence so is the quotient $M/\pi^nM$, giving comodule maps $\rho_n: M/\pi^n M\to M/\pi^n M\otimes_R \B$ for every $n\geq 1$. So we can define
$$
\widehat{\rho}=\varprojlim \rho_n:\widehat{M}\to\varprojlim(M/\pi^nM\otimes_R \B)\cong\tenR{M}{\B},
$$
which is just the $\pi$-adic completion of $\rho$. This gives the structure of a $\widehat{\B}$-comodule since all the maps $\rho_n$ are comodule maps. We also have that $\widehat{M}$ is an $\widehat{\A}$-module and the action map is a comodule homomorphism since the $\A$-action map on $M/\pi^n M$ is a comodule homomorphism for every $n\geq 1$.

For the last part, given a morphism $f:M\to N$ in $\Ca$, we have that the induced map $\widehat{f}:\widehat{M}\to\widehat{N}$ is an $\widehat{\A}$-module map. Moreover, by functoriality of taking $\pi$-adic completion in $R$-modules, the diagram
\[
\begin{tikzcd}
\widehat{M} \arrow{r}{\widehat{f}} \arrow{d}{\widehat{\rho_M}} & \widehat{N} \arrow{d}{\widehat{\rho_N}}\\
\tenR{M}{\B} \arrow{r}{f\widehat{\otimes}1} & \tenR{N}{\B}
\end{tikzcd}
\]
commutes as required.
\end{proof}

\begin{prop} If $\M, \Nn\in \widehat{\Ca}$ then there is a canonical isomorphism
$$
\text{\emph{Hom}}_{\widehat{\Ca}}(\M,\Nn)\cong \varprojlim\text{\emph{Hom}}_{\Ca}(\M/\pi^n\M,\Nn/\pi^n\Nn).
$$
In particular we have $\Gamma(\Nn)\cong\varprojlim \Gamma(\Nn/\pi^n \Nn)$ and so $\Gamma(\widehat{M})\cong\varprojlim \Gamma(M/\pi^n M)$ for $M\in\Ca$.
\end{prop}

\begin{proof}
Any $f:\M\to \Nn$ in $\widehat{\Ca}$ induces maps $f_n:\M/\pi^n\M\to\Nn/\pi^n\Nn$ in $\Ca$ for all $n$, and these determine $f$ uniquely by passing to the inverse limit. Hence there is a canonical injection
$$
\Hom_{\widehat{\Ca}}(\M,\Nn)\to \varprojlim\Hom_{\Ca}(\M/\pi^n\M,\Nn/\pi^n\Nn)
$$
which is surjective by the same argument as in Corollary \ref{comod2}. Hence by putting $\M=\widehat{\A}$ and using the fact that
$$
\Hom_{\Ca}(\A/\pi^n\A,\Nn/\pi^n\Nn)\cong \Hom_{\Ca}(\A,\Nn/\pi^n\Nn)=\Gamma(\Nn/\pi^n\Nn),
$$
we get the result on global sections.
\end{proof}

\begin{remark} This result is something one would expect given what the sections of an inverse limit of sheaves on a topological space are, see \cite[Proposition II.9.2]{Hartshorne}.
\end{remark}

\subsection{Weyl group localisations}\label{ore2} Recall the Ore localisations $\Aw$ of $\A$. We had a category $\Ca^w$ of $B$-equivariant $\Aw$-modules. We may analogously define categories $\widehat{\Ca^w}$ as follows. The objects are $\pi$-adically complete $\widehat{\Aw}$-modules and $\widehat{\B}$-comodules $\M$ such that the action map $\tenR{\widehat{\Aw}}{\M}\to \M$ is a comodule homomorphism with $\tenR{\widehat{\Aw}}{\M}$ being given the tensor comodule structure. Morphisms in $\widehat{\Ca^w}$ are $R$-linear maps preserving both structures. The results from sections \ref{Ca}-\ref{completion1} apply to $\widehat{\Ca^w}$ as well with identical proofs.

Now, the localisation map $\varphi:\A\to\Aw$ gives rise to a map $\widehat{\varphi}:\widehat{\A}\to\widehat{\Aw}$.

\begin{lem} Let $\M$ be an $\widehat{\A}$-module and let $n\geq 1$. We have: \begin{enumerate}
\item $\widehat{\Aw}$ is a Noetherian $R$-algebra, and is flat as an $\widehat{\A}$-module; and
\item there is an isomorphism
$$
\left(\widehat{\Aw}\otimes_{\widehat{\A}}\M\right)/\pi^n\left(\widehat{\Aw}\otimes_{\widehat{\A}}\M\right)\cong \Aw\otimes_{\A}(\M/\pi^n\M)
$$
of $\Aw$-modules.
\end{enumerate}
\end{lem}

\begin{proof} (i) That $\widehat{\Aw}$ is Noetherian will follow from Proposition \ref{Noetherian}(i) if we show that $\Aw/\pi\Aw$ is Noetherian. But it is the localisation of the commutative Noetherian ring $\A/\pi\A$ at the image of $S_w$ in the quotient, hence it is Noetherian. Then by \cite[3.2.3(vii)]{Berthelot} the flatness will follow if we show that the maps
$$
\varphi_a:\A/\pi^a\A\to \Aw/\pi^a\Aw
$$
are all flat for $a\geq 1$. But we have more generally that $\Aw/\pi^a\Aw$ is the localisation of $\A/\pi^a\A$ at the image of $S_w$ in the quotient, and by \cite[Proposition 2.1.16]{noncomalg} localisation is flat. So the maps $\varphi_a$ are flat.

(ii) We have
\begin{align*}
\left(\widehat{\Aw}\otimes_{\widehat{\A}}\M\right)/\pi^n\left(\widehat{\Aw}\otimes_{\widehat{\A}}\M\right)&\cong \left(\widehat{\Aw}/\pi^n \widehat{\Aw}\right)\otimes_{\widehat{\Aw}} \left(\widehat{\Aw}\otimes_{\widehat{\A}}\M\right)\\
&\cong \left(\widehat{\Aw}/\pi^n \widehat{\Aw}\right)\otimes_{\widehat{\A}}\M\\
&\cong \left(\Aw/\pi^n \Aw\right)\otimes_{\widehat{\A}}\M\\
&\cong \left(\Aw\otimes_{\A} \A/\pi^n\A\right)\otimes_{\widehat{\A}}\M\\
&\cong \Aw\otimes_{\A} \left(\widehat{\A}/\pi^n\widehat{\A}\otimes_{\widehat{\A}}\M\right)\\
&\cong \Aw\otimes_{\A}(\M/\pi^n\M)
\end{align*}
as required.
\end{proof}

\begin{remark} We believe $\widehat{\Aw}$ to be a microlocalisation of $\widehat{\A}$, but we haven't worked out the details.
\end{remark}

\begin{definition} Given an $\widehat{\A}$-module $\M$ and $w\in W$, we define the \emph{$w$-localisation} of $\M$ to be the $\pi$-adic completion $\widehat{\Aw}\widehat{\otimes}_{\widehat{\A}}\M$ of $\widehat{\Aw}\otimes_{\widehat{\A}}\M$. We sometimes denote it as $S_w^{-1}\M$ by abuse of notation.
\end{definition}

\begin{prop} The functor $\widehat{f_w}^*:\M\mapsto \widehat{\Aw}\widehat{\otimes}_{\widehat{\A}}\M$ sends $\widehat{\Ca}$ to $\widehat{\Ca^w}$.
\end{prop}

\begin{proof} By Lemma \ref{ore2}(ii), for any $\M\in\widehat{\Ca}$ and any $n\geq 1$ we have that
$$
\left(\widehat{\Aw}\otimes_{\widehat{\A}}\M\right)/\pi^n\left(\widehat{\Aw}\otimes_{\widehat{\A}}\M\right)\cong S_w^{-1}(\M/\pi^n\M).
$$
Now for each $n\geq 1$, $\M/\pi^n\M\in\Ca$ by Lemma \ref{Ca}, and so its localisation is in $\Ca^w$. The comodule maps $\rho_n$ on $\M/\pi^n\M$ are compatible with the reduction maps and hence so are their localisations $S_w^{-1}\rho_n$. Thus we may define a comodule map
$$
S_w^{-1}\rho=\varprojlim S_w^{-1}\rho_n: S_w^{-1}\M\to S_w^{-1}\M\widehat{\otimes}_R \B
$$
and the $\widehat{\Aw}$-action is a comodule homomorphism since it is modulo $\pi^n$ for all $n$.

Finally $\widehat{f_w}^*$ really is a functor since it sends a given morphism $\varphi:\M\to\Nn$ to $\varprojlim S_w^{-1}\varphi_n$ where $\varphi_n:\M/\pi^n\M\to\Nn/\pi^n\Nn$.
\end{proof}

Recall the notation from Section \ref{Cech}. For each $i\in J$ we have a forgetful functor $(\widehat{f_{w_i}})_*:\widehat{\Ca^{w_i}}\to\widehat{\Ca}$ which is a right adjoint to $\widehat{f_{w_i}}^*$: this is easily seen from Proposition \ref{completion1} since their restrictions modulo $\pi^n$ are adjoints for all $n$. Let $\sigma_i:=(\widehat{f_{w_i}})_*\circ \widehat{f_{w_i}}^*$ and, for any $\mathbf{i}=(i_1,\ldots, i_n)\in J^n$, set $\sigma_{\mathbf{i}}=\sigma_{i_1}\circ\cdots\circ \sigma_{i_n}$. We may then completely analogously as in \ref{Cech} write down a complex
$$
C^{\text{aug}}: \id_{\Ca}\to \bigoplus_{i\in J} \sigma_i\to \bigoplus_{\mathbf{i}\in J^2} \sigma_{\mathbf{i}}\to \bigoplus_{\mathbf{i}\in J^3} \sigma_{\mathbf{i}}\to\cdots
$$
which we may post-compose with the functor of taking global sections to obtain a complex $\check{C}^{\text{aug}}$, which we still call the \emph{augmented standard complex} of $\Gamma$. We often consider those complexes without the left hand most term, which we denote by $C$ and $\check{C}$. The complex $\check{C}$ is called the \emph{standard complex}.

\begin{remark} Note that by Lemma \ref{ore2}(ii) and Proposition \ref{completion1}, for any $\M\in\widehat{\Ca}$, we have isomorphisms
$$
C(\M)^{\text{aug}}\cong \varprojlim C^{\text{aug}}(\M/\pi^n\M)\quad\text{and}\quad \check{C}(\M)^{\text{aug}}\cong \varprojlim \check{C}^{\text{aug}}(\M/\pi^n\M),
$$
and the same is true for the non-augmented complexes.
\end{remark}

\begin{definition} Given $\M\in\widehat{\Ca}$ and $i\geq 0$, we define the \emph{\v{C}ech cohomology} $\check{H}^i(\M)$ to be the $i$-th cohomology group of $\check{C}(\M)$.
\end{definition}

\subsection{Quasi-abelian categories}\label{qacs}

In the next three subsections, we recall some of Schneiders' theory of quasi-abelian categories, following \cite{qacs}, and apply it to the category of Banach comodules over a Banach coalgebra.

We first introduce some notation. Let $\ban$ denote the category of Banach spaces over $L$. For an $L$-Banach algebra $A$, we denote by $\Mod{A}$ the category of Banach $A$-modules, and recall that for an $L$-Banach coalgebra $C$, we write $\Comod{C}$ for the category of Banach $C$-comodules.

\begin{definition}
Let $\C$ be an additive category with kernels and cokernels.
\begin{enumerate}
\item We say that a morphism $f:E\to F$ in $\C$ is \emph{strict} if the induced morphism
$$
\coim(f)\to \im(f)
$$
is an isomorphism, where $\im(f)$ is the kernel of the morphism $F\to \coker(f)$ and $\coim(f)$ is the cokernel of the morphism $\ker(f)\to E$.
\item We say that $\C$ is \emph{quasi-abelian} if it satisfies the following:
\begin{itemize}
\item In a cartesian square
\[
\begin{tikzcd}
E' \arrow{r}{f'} \arrow{d} & F' \arrow{d}\\
E \arrow{r}{f} & F
\end{tikzcd}
\]
if $f$ is a strict epimorphism, then so is $f'$.
\item In a co-cartesian square
\[
\begin{tikzcd}
E \arrow{r}{f} \arrow{d} & F \arrow{d}\\
E' \arrow{r}{f'} & F'
\end{tikzcd}
\]
if $f$ is a strict monomorphism, then so is $f'$.
\end{itemize}
\end{enumerate}
\end{definition}

\begin{example} Of course, abelian categories trivially satisfy the above definition since all morphisms are then strict. Moreover, $\ban$ and $\Mod{A}$ are quasi-abelian by \cite[Appendix A, Lemma A.30]{KreBB} and \cite[Prop 1.5.1]{qacs} respectively. Furthermore, the forgetful functor $\Mod{A}\to \ban$ preserves limits and colimits for any $L$-Banach algebra $A$. Hence a morphism is strict in $\Mod{A}$ if and only if it is strict in $\ban$.
\end{example}

\begin{remark}
Note that in $\ban$ (and in $\Mod{A}$ and $\Comod{C}$), while the kernel of a morphism $f:X\to Y$ is just the usual algebraic kernel (which is automatically closed), the cokernel of the same map is the canonical projection $Y\to Y/\overline{\im(f)}$. Thus the categorical image of $f$ in $\ban$ (i.e the kernel of the cokernel) is in fact the closure of the set theoretical image $\im(f)$.
\end{remark}

\begin{lem} Let $C$ be an $L$-Banach coalgebra. Then $\Comod{C}$ is quasi-abelian.
\end{lem}

\begin{proof} This category is clearly additive and, moreover, it follows from \cite[Lemma II.1.1]{Diarra} that kernel and images of homomorphisms of right Banach comodules are subcomodules (where by image we mean the closure of the set-theoretic image). Hence $\Comod{C}$ has kernels and cokernels. Note that by the proof of \cite[Appendix A, Lemma A.30]{KreBB}, the forgetful functor $\Comod{C}\to\ban$ sends (co)cartesian squares to (co)cartesian squares. Thus, since a morphism in $\Comod{C}$ is strict if and only if it is strict in $\ban$, it follows that $\Comod{C}$ is quasi-abelian.
\end{proof}

\subsection{Injective objects}\label{injectives} We now explain how to do homological algebra in quasi-abelian categories.

\begin{definition}
Let $\C$ and $\E$ be quasi-abelian categories.
\begin{enumerate}
\item A null sequence $E\overset{e}{\rightarrow} E'\overset{e'}{\rightarrow} E''$ in $\C$ is called \emph{exact} if the canonical map $\im(e)\to\ker(e')$ is an isomorphism. If furthermore $e$ (resp. $e'$) is strict then we say that this complex is \emph{strictly exact} (resp. \emph{strictly coexact}). A complex $E_1\to\cdots\to E_n$ is called exact, resp. strictly (co)exact if each subsequence $E_{i-1}\to E_i\to E_{i+1}$ is exact, resp. strictly (co)exact.
\item An additive functor $F:\C\to\E$ is called \emph{left exact} if it sends every strictly exact sequence
$$
0\to E\to E'\to E''\to 0
$$
in $\C$ to a strictly exact sequence
$$
0\to F(E)\to F(E')\to F(E'')
$$
in $\E$. In other words $F$ is left exact if it preserves kernels of strict morphisms.

We say that $F$ is \emph{strongly left exact} if it sends every strictly exact sequence
$$
0\to E\to E'\to E''
$$
in $\C$ to a strictly exact sequence
$$
0\to F(E)\to F(E')\to F(E'')
$$
in $\E$. In other words $F$ is strongly left exact if it preserves kernels of arbitrary morphisms.
\item Similarly, we say that $F$ is \emph{right exact} if it sends every strictly exact sequence
$$
0\to E\to E'\to E''\to 0
$$
in $\C$ to a strictly coexact sequence
$$
F(E)\to F(E')\to F(E'')\to 0
$$
in $\E$. There is an analogous notion of strongly right exact functors.
\item We say that $F$ is \emph{exact} if it is both left exact and right exact, i.e it sends every strictly exact sequence
$$
0\to E\to E'\to E''\to 0
$$
in $\C$ to a strictly exact sequence
$$
0\to F(E)\to F(E')\to F(E'')\to 0
$$
in $\E$.
\item An object $I$ in $\C$ is called \emph{injective} if the functor $E\mapsto \Hom(E,I)$ is exact, i.e for any strict monomorphism $E\to F$, the induced map $\Hom(F,I)\to\Hom(E,I)$ is surjective. Dually an object $P$ is called \emph{projective} if the functor $E\mapsto \Hom(P,E)$ is exact, i.e for any strict epimorphism $E\to F$, the induced map $\Hom(P,E)\to\Hom(P,F)$ is surjective.
\item We say that $\C$ has \emph{enough injectives} if for any object $E$ in $\C$, there is a
strict monomorphism $E\to I$ where $I$ is injective. Dually we say that $\C$ has \emph{enough projectives} if for every $E$ there is a strict epimorphism $P\to E$ from a projective object $P$.
\end{enumerate}
\end{definition}

\begin{example} In $\ban$, a map $f:X\to Y$ induces a continuous injection 
$$
X/\ker(f)\to \im(f)\hookrightarrow \overline{\im(f)},
$$
and $f$ is strict if and only if this injection is in fact an isomorphism of Banach spaces. This is equivalent to $\im(f)$ being closed by the Banach Open Mapping Theorem \cite[Corollary 8.7]{NFA}.

Now, a sequence
$$
0\to X\overset{f}{\to} Y\overset{g}{\to} Z\to 0 
$$
in $\ban$ is exact if and only if the following conditions are satisfied:
\begin{itemize}
\item $f$ is injective;
\item $\ker(g)=\overline{\im(f)}$; and
\item $g$ has dense image, i.e $\overline{\im(g)}=Z$.
\end{itemize}
The sequence is strict exact when $f$ and $g$ are furthermore strict. By the above criterion for strictness, we see that these three conditions then become that $f$ is injective, $g$ is surjective and $\ker(g)=\im(f)$, i.e the sequence is exact as a sequence of $L$-vector space. Hence we see that a sequence
$$
0\to X\to Y\to Z\to 0
$$
in $\ban$ is strict exact if and only if it is exact as a sequence of $L$-vector space. We will simply say in such a setting that the sequence is algebraically exact. The same will hold in $\Mod{A}$ and $\Comod{C}$ for an $L$-Banach algebra $A$ and an $L$-Banach coalgebra $C$ respectively.
\end{example}

Suppose that $C$ is an $L$-Banach coalgebra. Then there is an adjunction $(\phi^*,\phi_*)$ between $\Comod{C}$ and $\ban$ just like in \ref{recap1}. Namely $\phi^*$ is the forgetful functor while $\phi_*:\M\to \ten{\M}{C}$ with coaction $\id_{\M}\widehat{\otimes}\Delta$. The bijection giving rise to this sends a map $f:\M\to \ten{\Nn}{C}$ in $\Comod{C}$ to $(\id_{\Nn}\widehat{\otimes} \varepsilon)\circ f:\M\to \Nn$, whose inverse sends a map $g:\M\to \Nn$ in $\ban$ to $(g\widehat{\otimes}\id_C)\circ\rho_{\M}:\M\to \ten{\Nn}{C}$. Having established this we have: 

\begin{prop} Let $A$ be an $L$-Banach algebra and $C$ be an $L$-Banach coalgebra. Then \emph{$\Mod{A}$} and \emph{$\Comod{C}$} have enough injectives.
\end{prop}

\begin{proof}
The arguments in \cite[Lemma 4.25]{KreBB} for $\Mod{A}$ generalise straightforwardly to the noncommutative case. Note that we use the fact that $\ban$ has enough injectives by \cite[Appendix A, Lemma A.42]{KreBB}. For $\Comod{C}$, note that since the forgetful functor $\phi^*$ preserves strict monomorphisms, and since $\phi_*$is its right adjoint, it follows from \cite[Lemma 4.26]{KreBB} that $\phi_*$ preserves injective objects. Now let $\M$ be an object of $\Comod{C}$. Since $\ban$ has enough injectives there exists a strict monomorphism $f:\M\hookrightarrow I$ of Banach spaces where $I$ is injective. This induces a map $\hat{f}:=f\widehat{\otimes}\id_C:\ten{\M}{C}\to\ten{I}{C}$ where we view $\ten{M}{C}$ as $\phi_*(\phi^*(M))$. By the above $\phi_*(I)=\ten{I}{C}$ is injective in $\Comod{C}$, and we have the adjunction map $\rho:\M\to \ten{\M}{C}$ which is just the comodule map. Therefore we have a morphism
$$
\iota:=\hat{f}\circ\rho:\M\to\ten{I}{C}
$$
in $\Comod{C}$ from $\M$ to an injective object. We claim that it is a strict monomorphism in $\ban$, which implies the result. Now note that since $\rho$ has a left inverse given by $1\widehat{\otimes}\bar{\varepsilon}$, it follows by Lemma \ref{equiv} that $\rho$ is a strict monomorphism. Moreover $\hat{f}$ is a strict monomorphism by Lemma \ref{funal}(ii). It now follows from the Lemma below that $\iota$ is a strict monomorphism as well.
\end{proof}

\begin{lem} In $\ban$, a composite of strict injections is strict.
\end{lem}

\begin{proof}
If $f:X\to Y$ and $g:Y\to Z$ are strict injections, then by definition this means they define topological isomorphisms onto their image. Equivalently, there are real numbers $C_1, C_2, D_1, D_2>0$ such that for all $x\in X$ and all $y\in Y$, we have
$$
C_1\norm{x}_X\leq \norm{f(x)}_Y\leq C_2\norm{x}_X \quad\text{and}\quad D_1\norm{y}_Y\leq \norm{g(y)}_Z\leq D_2\norm{y}_Y.
$$
But by putting $y=f(x)$, this gives $C_1D_1\norm{x}_X\leq \norm{g(f(x))}_Z\leq C_2D_2\norm{x}_X$, so that $g\circ f$ also defines a homeomorphism onto its image, and is therefore strict.
\end{proof}

\begin{remark} The fact that a composite of strict monomorphisms is strict holds more generally in any quasi-abelian category, see \cite[Proposition 1.1.7]{qacs}, but we thought it would be clearer to give a direct proof.
\end{remark}

\subsection{Derived categories}\label{derived} We now turn to derived categories and derived functors.

\begin{definition}
Let $\C$ be a quasi-abelian category.
\begin{enumerate}
\item Let $K(\C)$ be the homotopy category of $\C$, i.e the category of complexes modulo homotopies. Then the \emph{derived category} of $\C$ is defined to be
$$
D(\C)=K(\C)/N(\C)
$$
where $N(\C)$ is the full subactegory of strictly exact sequences.
\item Let $F:\C\to \E$ be an additive functor between $\C$ and another quasi-abelian category $\E$. A full additive subcategory $\I$ of $\C$ is called \emph{$F$-injective} if:
\begin{enumerate}
\item for any object $V$ of $\C$ there is a strict monomorphism $V\to I$ where $I$ is an object of $\I$
\item for any strictly exact sequence
$$
0\to V'\to V\to V''\to 0
$$
in $\C$, if $V$ and $V''$ are objects of $\I$, then $V'$ is as well
\item for any strictly exact sequence
$$
0\to V'\to V\to V''\to 0
$$
in $\C$ where $V$, $V'$ and $V''$ are objects of $\I$, the sequence
$$
0\to F(V')\to F(V)\to F(V'')\to 0
$$
is strictly exact in $\E$
\end{enumerate}
\item Given $F:\C\to\E$ as above, we say that $F$ is \emph{right derivable} if it has a right derived functor $RF:D^+(\C)\to D^+(\E)$ which satisfies the usual universal property.
\item Assume $F$ is right derivable. We say that an object $I$ of $\C$ is \emph{$F$-acyclic} if $RF(I)\cong F(I)$.
\end{enumerate}
\end{definition}

We now introduce the left heart of the derived category. By \cite[Definition 1.2.17]{qacs}, there is a left $t$-structure on the derived category $D(\C)$ of a quasi-abelian category $\C$, whose heart we denote by $\LH(\C)$, called the \emph{left heart} of the quasi-abelian category $\C$. The left heart is an abelian category and there is a natural functor $I:\C\to \LH(\C)$.

\begin{lem}\begin{enumerate} \item \emph{(\cite[Corollary 1.2.27]{qacs})} The functor $I$ is a fully faithful embedding. Moreover a sequence in $\C$ is strictly exact if and only if it's exact in $\LH(\C)$.
\item \emph{(\cite[Proposition 1.2.31]{qacs})} The functor $I$ induces an equivalence of derived categories $D(I): D(\C)\to D(\LH(\C))$.
\item \emph{(\cite[Corollary 1.2.19]{qacs})} For $n\in\Z$ let $LH^n:D(\C)\to\LH(\C)$ denote the $n$-th cohomology functor. Then, for $E\in D(\C)$, we have that $LH^n(E)=0$ if and only if $E$ is strictly exact in degree $n$.
\end{enumerate}
\end{lem}

We gather all the important results that we need in the following:

\begin{prop}
Let $F:\C\to \E$ be an additive functor between quasi-abelian categories $\C$ and $\E$.
\begin{enumerate}
\item \emph{(\cite[Prop 1.3.5]{qacs})} Assume that $\C$ has an $F$-injective subcategory. Then $F$ has a right derived functor $RF:D^+(\C)\to D^+(\E)$ (in this situation we say that $F$ is \emph{explicitly right derivable}).
\item \emph{(\cite[Remark 1.3.21]{qacs})} Assume that $\C$ has enough injectives. Then the full subcategory of injective objects is an $F$-injective subcategory.
\item \emph{(\cite[Remark 1.3.7]{qacs})} Assume that $F$ has a right derived functor $RF:D^+(\C)\to D^+(\E)$, and suppose that for any object $V$ of $\C$, there is a monomorphism $V\to I$ where $I$ is $F$-acyclic. Then the $F$-acyclic objects form an $F$-injective subcategory.
\item \emph{(\cite[Proposition 1.3.8 \& Proposition 1.3.14]{qacs})} Assume that $F$ is explicitly right derivable and let $\I$ be an $F$-injective subcategory of $\C$. Suppose that $F$ is strongly left exact and that, for any monomorphism $I_0\to I_1$ in $\C$ between objects of $\I$, $FI_0\to FI_1$ is a monomorphism. Then $F$ extends to an explicitly right derivable left exact functor $G:\LH(\C)\to\LH(\E)$ such that $RG\cong RF$.
\end{enumerate}
\end{prop}

\subsection{The analytic quantum flag variety}\label{compflag}

Recall the Banach Hopf algebras $\Oqhat$ and $\OqBhat$, and that $\Oqhat$ is Noetherian by Corollary \ref{Noetherian}. Note that, completely analogously to section \ref{Ca}, given two Banach $\OqBhat$-comodules $\M$ and $\Nn$, we may define a tensor comodule structure on $\ten{\M}{\Nn}$.

\begin{definition}
We let $\flaghat$ denote the category whose objects are triples $(\M,\alpha,\beta)$ where $\M$ is an $L$-Banach space, $\alpha:\ten{\Oqhat}{\M}\to \M$ is a left $\Oqhat$-module action and $\beta:\M\to\ten{\M}{\OqBhat}$ is a right $\OqBhat$-comodule action, such that $\alpha$ is a comodule homomorphism where $\ten{\Oqhat}{\M}$ is given the tensor comodule structure. The morphisms are just the continuous linear maps which are both module and comodule homomorphisms.

A triple $(\M,\alpha,\beta)$ where $\M$ is a finitely generated $\Oqhat$-module is called \emph{coherent}, and we let $\coh(\flaghat)$ denote the full subcategory of coherent modules.
\end{definition}

\begin{remark}
Of course, $\Oqhat\in\flaghat$ and is coherent by definition. Note that by \cite[Prop 2.1]{SchTeit03}, any finitely generated module over a Noetherian Banach algebra $A$ automatically come equipped with a canonical topology, which is the unique topology making them into Banach modules. Namely, given a finite generating set for a module $\M$ we obtain a surjective map $A^a\to \M$, and we give $\M$ the quotient topology or in other words set the unit ball of $\M$ to be the image of $(A^\circ)^a$ under that map. This topology does not depend on the choice of generating set. We also have that all module maps between finitely generated modules are automatically continuous and strict. Hence we will always assume that our coherent modules are equipped with this topology.
\end{remark}

By the above remark, the full subcategory of finitely generated modules in $\Mod{\Oqhat}$ is abelian. Similarly $\coh(\flaghat)$ is also abelian. However there is no guarantee that it has enough injectives. Instead we will work in $\flaghat$.

We begin by generalising one of the adjunction from Section \ref{recap1}. There is an adjoint pair of functors $(\theta^*,\theta_*)$ between $\Mod{\Oqhat}$ and $\flaghat$. The functor $\theta_*:\Mod{\Oqhat}\to\flaghat$ is given by $N\mapsto \ten{N}{\OqB}$ where $\Oqhat$ acts on $\theta_*(N)$ via the tensor action and the $\OqBhat$-coaction comes from the second factor, while $\theta^*:\flaghat\to\Mod{\Oqhat}$ is just the forgetful functor. The bijection making this an adjunction is as follows: let $M\in \flaghat$ and $N\in\Mod{\Oqhat}$, and let $\rho: M\to\ten{M}{\OqBhat}$ and $\bar{\varepsilon}:\OqBhat\to L$ be the comodule map and the counit of $\OqBhat$ respectively; given a module homomorphism $f:M\to N$, construct a morphism $g:M\to \ten{N}{\OqBhat}$ in $\flaghat$ by taking the composite $(\ten{f}{\id})\circ \rho$. Conversely, given a morphism $g:M\to \ten{N}{\OqBhat}$ in $\flaghat$, we construct a module homomorphism $f:M\to N$ by taking the composite $(\ten{\id}{\bar{\varepsilon}})\circ g$. Having established this, we can now prove:

\begin{lem}
The category $\flaghat$ is quasi-abelian and has enough injectives.
\end{lem}

\begin{proof}
The proof is entirely analogous to the proofs of Lemma \ref{qacs} and Proposition \ref{injectives}, using the adjunction $(\theta^*,\theta_*)$.
\end{proof}

We now define a global section functor. First note that for any $M, N\in\flaghat$, the space $\Hom_{\flaghat}(M,N)$ is a closed subspace of the Banach space $B(M,N)$ of all bounded linear maps from $M$ to $N$. In particular it is itself a Banach space.

\begin{definition}
The global section functor $\Gamma:\flaghat\to \ban$ is defined to be the functor
$$
\Gamma: \M\mapsto \Hom_{\flaghat}(\Oqhat,\M).
$$
Alternatively, let $\M^{B_q}:=\{m\in \M: \rho(m)=m\otimes 1\}$, a closed subspace of $\M$. We may think of global sections as taking $B_q$-invariants since there is a canonical isomorphism
\begin{align*}
\Gamma(\M)&\overset{\sim}{\longrightarrow} \M^{B_q}\\
f&\longmapsto f(1)
\end{align*}
by the Banach Open Mapping Theorem \cite[Corollary 8.7]{NFA}.
\end{definition}

Since $\Gamma$ is just a hom functor, it is left exact. In fact it clearly satisfies the conditions of Proposition \ref{derived}(iv). Hence, as $\flaghat$ has enough injectives, this functor is right derivable and extends to a left exact functor $\LH(\flaghat)\to\LH(\ban)$ with the same cohomology, and we denote by $R^i\Gamma$ the corresponding cohomology groups in $\LH(\ban)$.

\subsection{Coherent lattices}\label{basechange2} The functor $\M\mapsto \M_L$ of course sends $\widehat{\Ca}$ to $\flaghat$. We show that this functor is in fact surjective.

\begin{prop} Let $\M\in \widehat{\Ca}$ and $\Nn\in \flaghat$. Then:
\begin{enumerate}
\item $\Gamma(\M_L)\cong \Gamma(\M)\otimes_R L$; and
\item $\Nn$ contains an $R$-lattice $F_0\Nn$ which is an element of $\widehat{\Ca}$. Moreover, $F_0\Nn$ can be chosen to be coherent if $\Nn\in \text{\emph{coh}}(\flaghat)$.
\end{enumerate}
\end{prop}

\begin{proof}
(i) Same as the first part of the proof of Proposition \ref{base_change}.

(ii) Denote the unit ball of $\Nn$ by $\Nn^\circ$. Notice that the adjunction morphism $\Nn\to \theta_*\theta^*(\Nn)=\ten{\Nn}{\OqBhat}$ is just the comodule map $\rho$, which has a left inverse given by $1\widehat{\otimes}\bar{\varepsilon}$, so which is a strict injection by Lemma \ref{equiv}. Thus we see that $\Nn$ is isomorphic to a subobject of $\ten{\Nn}{\OqBhat}$ in $\flaghat$, where $\ten{\Nn}{\OqBhat}$ has the tensor $\Oqhat$ module structure and comodule structure given by $1\widehat{\otimes}\widehat{\Delta}$. From now on we identify $\Nn$ with its image in $\ten{\Nn}{\OqBhat}$ via $\rho$, and in particular equipped with the subspace norm.

Let $F_0\Nn=\Nn\cap \Nn^\circ\widehat{\otimes}_R\widehat{\B}$ denote the unit ball of $\Nn$ with respect to that new norm. Note that the comodule map on $\ten{\Nn}{\OqBhat}$ has norm at most 1, hence so is its restriction to $\Nn$, and thus we see that $F_0\Nn$ is a $\widehat{\B}$-comodule. Moreover, since the $\Oqhat$-action on $\ten{\Nn}{\OqBhat}$ has norm at most 1, i.e $\widehat{\A}$ preserves the unit ball $\Nn^\circ\widehat{\otimes}_R\widehat{\B}$, it follows that $F_0\Nn$ is an $\widehat{\A}$-module. Hence $F_0\Nn$ belongs in $\widehat{\Ca}$. 

For the last part it is just left to check that it is finitely generated if $\Nn\in\coh(\flaghat)$. By definition, the topology on $\Nn$ is the quotient topology from a surjection $(\Oqhat)^a\twoheadrightarrow \Nn$ for some choice of generators in $\Nn$, and the unit ball $\Nn^\circ$ is given by the image of $(\widehat{\A})^a$. But now, by the equivalence of norms, there exists $b\geq 0$ such that
$$
F_0\Nn\subseteq \pi^{-b} \Nn^\circ,
$$
and hence $F_0\Nn$ is finitely generated since $\Nn^\circ$ is and $\widehat{\A}$ is Noetherian.
\end{proof}

\subsection{Exactness of Weyl group localisation functors}\label{exactloc} We now start working towards a \v{C}ech complex for computing the cohomology of elements of $\flaghat$. For each $w\in W$, let $\Oqwhat:=\widehat{\Aw}\otimes_R L$ and let $\flaghat_w$ denote the category of $B$-equivariant $\Oqwhat$-modules. Specifically, the objects are triples $(\M,\alpha,\beta)$ where $\M$ is an $L$-Banach space, $\alpha:\ten{\Oqwhat}{\M}\to \M$ is a left $\Oqwhat$-module action and $\beta:\M\to\ten{\M}{\OqBhat}$ is a right $\OqBhat$-comodule action, such that $\alpha$ is a comodule homomorphism where $\ten{\Oqwhat}{\M}$ is given the tensor comodule structure. The morphisms are just the continuous linear maps which are both module and comodule homomorphisms.

We first recall some general facts before applying them to this situation. Let $A$ be a Banach algebra, $\M$ a Banach right $A$-module, $\Nn$ a Banach left $A$-module. By \cite[2.1.7]{BGR} we may then define the completed tensor product $\M\widehat{\otimes}_A\Nn$ to be the completion of the tensor product $\M\otimes_A\Nn$ with respect to the semi-norm
$$
\norm{x}:=\inf\Big\{\max_{1\leq i\leq r}\norm{m_i}_{\M}\cdot\norm{n_i}_{\Nn} : x=\sum_{i=1}^r m_i\otimes n_i, m_i\in \M, n_i\in \Nn\Big\}.
$$
Now suppose that $B$ is another Banach algebra, and assume there is a continuous algebra homomorphism $f:A\to B$. Then this induces a functor $f^*:\Mod{A}\to\Mod{B}$ given by $\M\mapsto B\widehat{\otimes}_A\M$. This functor has a right adjoint given by restriction. We now investigate when the functor $f^*$ is strict exact. What we need is the following result of Bode:

\begin{prop}\emph{(\cite[Proposition 1.3]{Andreas1})} Suppose $A$ and $B$ are Banach algebras, with unit balls $A^\circ$ and $B^\circ$ respectively. Furthermore, suppose that $A^\circ$ and $B^\circ$ are Noetherian $R$-algebras and assume that there is a continuous algebra homomorphism $f:A\to B$ such that $f(A^\circ)\subseteq B^\circ$. If $f$ makes $B^\circ$ into a flat $A^\circ$-module, then the functor $f^*$ is strict exact.
\end{prop}

Now, we will need the following general result:

\begin{lem}
Suppose that $A$ and $B$ are as in the Proposition and let $\M$ be a $\pi$-adically complete $A^\circ$-module which is $\pi$-torsion free. Then we have an isomorphism
$$
\widehat{B^\circ\otimes_{A^\circ} \M}\otimes_R L\cong B\widehat{\otimes}_A \M_L
$$
and the $B^\circ$-module $B^\circ\otimes_{A^\circ} \M$ is $\pi$-torsion free.
\end{lem}

\begin{proof}
By the discussion in section \ref{funal}, without loss of generality we may assume that the topologies on $A$ and $B$ are given by the $\pi$-adic topologies on $A^\circ$ and $B^\circ$ respectively. In this setting, by flatness, we have an injection
$$
B^\circ\otimes_{A^\circ} \M\hookrightarrow B^\circ\otimes_{A^\circ} \M_L\cong B\otimes_A \M_L
$$
which is a strict injection onto the unit ball of $B\otimes_A \M_L$ by \cite[Lemma 2.2]{Andreas1} (where we give the left hand side the $\pi$-adic topology). Thus we immediately get that $B^\circ\otimes_{A^\circ} \M$ is $\pi$-torsion free, and the required isomorphism is obtained by taking completions.
\end{proof}

Now there is a continuous algebra homomorphism $\Oqhat\to\Oqwhat$ induced from the localisation and this induces a functor $\widehat{f_w}^*:\Mod{\Oqhat}\to\Mod{\Oqwhat}$ with right adjoint $(\widehat{f_w})_*$ given by restriction as above.

\begin{cor} Let $w\in W$. Then:
\begin{enumerate}
\item for any $\pi$-adically complete $\widehat{\A}$-module $\Nn$,
$$
(\widehat{\Aw}\widehat{\otimes}_{\widehat{\A}}\Nn)_L\cong \Oqwhat\widehat{\otimes}_{\Oqhat}\Nn_L.
$$
In particular, if $\Nn\in\widehat{\Ca}$, then $(\widehat{f_w})^*(\Nn_L)\cong (\widehat{f_w})^*(\Nn)\otimes_R L$.
\item the functor $(\widehat{f_w})^*$ defined above is strict exact and sends $\flaghat$ to $\flaghat_w$.
\end{enumerate}
\end{cor}

\begin{proof}
Part (i) follows immediately by putting $A=\Oqhat$ and $B=\Oqwhat$ in the above Lemma, and using Lemma \ref{ore2}(i)

For (ii), it follows immeditaly from the Proposition that $(\widehat{f_w})^*$ is strict exact by Lemma \ref{ore2}(i). Now by Proposition \ref{basechange2}, given any $\M\in\flaghat$ we may assume that $\M=\Nn\otimes_R L$ for some $\Nn\in\widehat{\Ca}$. The result therefore follows from (i) and from Proposition \ref{ore2}.
\end{proof}

\subsection{Global sections on $\flaghat_w$}\label{acyclic1} We now investigate exactness properties of global sections on the various localisations $\flaghat_w$. Recall from section \ref{ore} that there is a comodule map $\Delta_w:\Aw\to \Aw\otimes_R \B$ which makes $\Aw\otimes_R \B$ into an $\Aw$-module. After taking $\pi$-adic completion and extending scalars, there is an analgous comodule map $\widehat{\Delta_w}:\Oqwhat\to\ten{\Oqwhat}{\OqBhat}$. Note then that given any Banach $\Oqwhat$-module $\M$, we may give an $\Oqwhat$-module structure to $\ten{\M}{\OqBhat}$ via the map $\widehat{\Delta_w}$, and an $\OqBhat$-comodule structure by coacting on the right factor. Thus $\M\mapsto \ten{\M}{\OqBhat}$ defines a functor from $\Mod{\Oqwhat}\to\flaghat_w$, and this has a left adjoint given by the forgetful functor (this is completely analogous to the adjunction in section \ref{Rflag}). With the same proofs as in Lemma \ref{compflag} and Proposition \ref{basechange2}, we obtain:

\begin{lem} Let $w\in W$. Then:
\begin{enumerate}
\item the category $\flaghat_w$ is quasi-abelian and has enough injectives; and
\item for any $\M\in\flaghat_w$, there exists $F_0\M\in\widehat{\Ca^w}$ such that $\M\cong (F_0\M)_L$.
\end{enumerate}
\end{lem}

Now we define the global section functor on $\flaghat_w$ to be the composite $\Gamma\circ(\widehat{f_w})_*$. Analogously to Lemma \ref{Cech}, we have:

\begin{prop} For any $w\in W$, the global section functor $\Gamma\circ(\widehat{f_w})_*$ is strict exact and objects of $\flaghat_w$ have acyclic image under $(\widehat{f_w})_*$.
\end{prop}

\begin{proof} Suppose we have a strict short exact sequence
$$
0\to\K\to\M\to\Nn\to 0
$$
in $\flaghat_w$. By part (ii) of the Lemma, we may assume that the unit ball $\M^\circ$ of $\M$ is in $\widehat{\Ca^w}$. By strictness, we may assume that $\K$ and $\Nn=\M/\K$ are equipped with the subspace and quotient topologies respectively. Note that from these assumptions, the comodule maps on $\M$, $\K$ and $\Nn$ all have norm at most 1, and so the unit balls are all in $\widehat{\Ca^w}$. Furthermore, we have a short exact sequence
$$
0\to \K^\circ\to\M^\circ\to \Nn^\circ\to 0
$$
which induces short exact sequences
$$
0\to \K^\circ/\pi^n\K^\circ\to\M^\circ/\pi^n\M^\circ\to \Nn^\circ/\pi^n\Nn^\circ\to 0
$$
for every $n\geq 1$ by strictness. Now $\K^\circ/\pi^n\K^\circ\in\Ca^w$ and so it is acyclic by Lemma \ref{Cech}. Hence we get a tower of short exact sequences
$$
0\to \Gamma(\K^\circ/\pi^n\K^\circ)\to\Gamma(\M^\circ/\pi^n\M^\circ)\to \Gamma(\Nn^\circ/\pi^n\Nn^\circ)\to 0.
$$
Moreover, since $\K^\circ$ has no $\pi$-torsion, we may apply Lemma \ref{MLcond} below to obtain that $(\Gamma(\K^\circ/\pi^n\K^\circ))_{n\geq 1}$ satisfies the Mittag-Leffler condition. Thus by applying Proposition \ref{completion1}, we can pass to the inverse limit to get a short exact sequence
$$
0\to\Gamma(\K^\circ)\to\Gamma(\M^\circ)\to\Gamma(\Nn^\circ)\to 0.
$$
Extending scalars to $L$ and applying Proposition \ref{basechange2}(i), we get that
$$
0\to\Gamma(\K)\to\Gamma(\M)\to\Gamma(\Nn)\to 0.
$$
is algebraically exact, hence strict exact.

The last part follows exactly as in Lemma \ref{Cech}, using the fact that $(\widehat{f_w})^*$ is strict exact by Corollary \ref{exactloc}(ii) and so its right adjoint preserves injectives by \cite[Lemma 4.26]{KreBB}.
\end{proof}

\subsection{Lemma}\label{MLcond} \emph{Suppose that $\M\in\widehat{\Ca}$ is $\pi$-torsion free and that $\M/\pi\M$ is $\Gamma$-acyclic. Then for every $n\geq 1$ the natural map
$$
\Gamma(\M/\pi^{n+1}\M)\to \Gamma(\M/\pi^n\M)
$$
is surjective. In particular, $(\Gamma(\M/\pi^n\M))_{n\geq 1}$ satisfies the Mittag-Leffler condition.}

\begin{proof} Since $\M$ has no $\pi$-torsion, we have a short exact sequence
$$
0\to \M/\pi\M\overset{\pi^n}{\to} \M/\pi^{n+1}\M\to\M/\pi^n\M\to 0
$$
in $\Ca$ for each $n\geq 1$. Thus by acyclicity of $\M/\pi\M$ we get that the maps
$$
\Gamma(\M/\pi^{n+1}\M)\to\Gamma(\M/\pi^n\M)
$$
are surjective for every $n\geq 1$, as required.
\end{proof}

\subsection{\v{C}ech cohomology}\label{Cech2} We are now in a position to construct a \v{C}ech complex and prove that it computes the cohomology. Just as in \ref{Cech} and \ref{ore2}, we may define an augmented complex $C^{\text{aug}}$ and an augmented standard complex $\check{C}^{\text{aug}}$ as well as their non-augmented counter-parts $C$ and $\check{C}$.

\begin{thm} For any $\M\in\flaghat$, the complex $C^{\text{aug}}(\M)$ is strict exact and is an acyclic resolution of $\M$ in $\flaghat$.
\end{thm}

\begin{proof}
The fact that the terms in $C(\M)$ are $\Gamma$-acyclic follows from Proposition \ref{acyclic1}. Moreover the last part follows from the first part by Proposition \ref{derived}(iii). So we are left to show that $C^{\text{aug}}(\M)$ is strict exact. By Proposition \ref{basechange2} we may pick $F_0\M\in\widehat{\Ca}$ such that $\M=(F_0\M)_L$. Also, for any $\Nn\in\widehat{\Ca}$ we have that $(\widehat{f_w})_*(\M)\cong (\widehat{f_w})_*(\Nn)\otimes_R L$ by Corollary \ref{exactloc}(i). Thus we see that
$$
C^{\text{aug}}(\M)=C^{\text{aug}}(F_0\M)\otimes_R L
$$
and it will suffice to show that $C^{\text{aug}}(F_0\M)$ is exact (algebraically).

Now we have that
$$
C^{\text{aug}}(F_0\M)\cong \varprojlim C^{\text{aug}}(F_0\M/\pi^nF_0\M)
$$
by Remark \ref{ore2}. However the complexes $C^{\text{aug}}(F_0\M/\pi^nF_0\M)$ are exact for all $n\geq 1$ by Lemma \ref{Ca} and Proposition \ref{Cech}. Moreover the maps
$$
C^{\text{aug}}(F_0\M/\pi^{n+1}F_0\M)\to C^{\text{aug}}(F_0\M/\pi^nF_0\M)
$$
are all surjective, so the inverse system of complexes satisfies the Mittag-Leffler condition. The induced maps between the cohomology groups of thoses complexes, which are all zero, trivially also satisfy the Mittag-Leffler condition. Hence the result follows from \cite[Proposition 0.13.2.3]{EGAIII.1}.
\end{proof}

From this we may deduce our promised Theorem \ref{thmB}:

\begin{proof}[Proof of Theorem \ref{thmB}]
This now follows immediately from the previous Theorem by Proposition \ref{derived}(iii).
\end{proof}

\begin{cor} The global section functor on $\flaghat$ has finite cohomological dimension. More specifically, $R^i\Gamma=0$ for $i>N+1$.
\end{cor}

\begin{proof} Take $\M\in\flaghat$. By the Theorem, the cohomology $R\Gamma(\M)$ is computed by the complex $\check{C}(\M)$. Now let $F_0\M\in\widehat{\Ca}$ be the lattice in $\M$ given by Proposition \ref{basechange2}, so that $\check{C}(\M)=\check{C}(F_0\M)\otimes_R L$, and we have
$$
\check{C}(F_0\M)\cong \varprojlim \check{C}(F_0\M/\pi^nF_0\M).
$$
Now $C(F_0\M)$ is a complex of the form
$$
C(F_0\M):\Nn_1\to \Nn_2\to\cdots
$$
such that each $\Nn_i$ is $\pi$-torsion free by Lemma \ref{exactloc} and, for every $n\geq 1$, each $\Nn_i/\pi^n\Nn_i$ is $\Gamma$-acyclic by Lemma \ref{Cech}. Hence the complexes
$$
\check{C}(F_0\M/\pi^nF_0\M): \Gamma (\Nn_1/\pi^n \Nn_1)\to\Gamma (\Nn_2/\pi^n \Nn_2)\to\cdots 
$$
satisfy the Mittag-Leffler condition by Lemma \ref{MLcond}. Moreover, for every $n\geq 1$, the $i$-th cohomology group of $\check{C}(F_0\M/\pi^nF_0\M)$ is zero for $i>N$ by Proposition \ref{Cech} and Proposition \ref{line}(iii). Thus the induced maps
$$
R^i\Gamma(F_0\M/\pi^{n+1}F_0\M)\to R^i\Gamma(F_0\M/\pi^nF_0\M)
$$
trivially satisfy the Mittag-Leffler condition for $i>N$, and we may therefore invoke \cite[Proposition 0.13.2.3]{EGAIII.1} to obtain that the complex $\check{C}(F_0\M)$ is exact in degrees bigger than $N+1$. Hence the complex $\check{C}(\M)$ is strict exact in degrees bigger than $N+1$ and the result now follows by Lemma \ref{derived}(iii).
\end{proof}

\section{The Beilinson-Bernstein Theorem for $\widehat{\Dflag}$}

In this Section, we define suitable notions of $D$-modules on our analytic quantum flag varieties and apply the results from the previous Sections to obtain a Beilinson-Bernstein localisation theorem.

Recall the notion of a module algebra over a Hopf algebra from section \ref{recap2}. If $H$ denotes a Hopf $R$-algebra, then its $\pi$-adic completion $\widehat{H}$ may be thought of as a Hopf algebra-like object, with maps $\widehat{\Delta}:\widehat{H}\to\tenR{H}{H}$, $\widehat{\varepsilon}:\widehat{H}\to R$ and $\widehat{S}:\widehat{H}\to \widehat{H}$ satisfying the usual axioms. We will then say that an $R$-algebra $A$ which is also an $\widehat{H}$-module is an $\widehat{H}$-module algebra if, when viewed as an $H$-module, it is an $H$-module algebra.

\subsection{The category $\widehat{\Da^\lambda}$}\label{Dhatmod} We now define those elements of $\widehat{\Ca}$ which play the role of $D$-modules. Recall that, by Lemma \ref{comod2}, if $\M$ is a $\widehat{\B}$-comodule then it is a $\widehat{U^{\text{res}}(\bo)}$-module. Also note that the inclusion $U^{\geq 0}\subseteq U^{\text{res}}(\bo)$ induces an $R$-algebra homomorphism $\widehat{U^{\geq 0}}\to\widehat{U^{\text{res}}(\bo)}$ so that $\widehat{U^{\text{res}}(\bo)}$-modules are naturally $\widehat{U^{\geq 0}}$-modules.

Since $\A$ is a $U$-module algebra and since $\pi^a U$ is a Hopf ideal in $U$ for all $a\geq 1$, $\A/\pi^a\A$ is a $U/\pi^a U$-module algebra for all $a\geq 1$ and thus $\widehat{\A}$ is a $\widehat{U}$-module algebra. Hence since $\widehat{\D}=\tenR{\widehat{\A}}{\widehat{U}}$ as an $R$-module, we see that we may think of $\widehat{\D}$ as the smash product algebra of $\widehat{\A}$ and $\widehat{U}$. Moreover it is a $\widehat{U}$-module algebra. Similarly to section \ref{Verma}, we then see that $\widehat{\D}$ is a $\widehat{U^{\text{res}}}$-module algebra since $\D$ is a $U^{\text{res}}$-module algebra. Finally observe that the inclusion $U^{\geq 0}\to \D$ induces a map $\widehat{U^{\geq 0}}\to\widehat{\D}$, so that every $\widehat{\D}$-module is naturally a $\widehat{U^{\geq 0}}$-module.

\begin{definition} Let $\lambda\in T_P^R$. We define the category $\widehat{\Da^\lambda}$ to have objects triples $(\M, \alpha, \beta)$ where $\M$ is a $\pi$-adically complete $R$-module, $\alpha: \tenR{\widehat{\D}}{\M}\to\M$ is a $\widehat{\D}$-module structure, and $\beta:\M\to\tenR{\M}{\B}$ is a $\widehat{\B}$-comodule structure. By the above this induces a $\widehat{U^{\text{res}}(\bo)}$-module structure on $\M$, which we also denote by $\beta$. We require these to satisfy the following:
\begin{enumerate}
\item the $\widehat{U^{\geq 0}}$-actions on $\M\otimes_R R_\lambda$ given by $\beta\otimes \lambda$ and $\alpha\vert_{U^{\geq 0}}\otimes 1$ are equal; and
\item the action map $\alpha:\tenR{\D}{\M}\to \M$ is $\widehat{U^{\text{res}}(\bo)}$-linear.
\end{enumerate}
The morphisms are the $R$-module maps which are both $\widehat{\B}$-comodule homomorphisms and $\widehat{\D}$-module homomorphisms. As usual we call $\M\in \widehat{\Da^\lambda}$ \emph{coherent} if it is finitely generated as $\widehat{\D}$-module, and denote by $\coh(\widehat{\Da^\lambda})$ the full subcategory of such objects.
\end{definition}

As always there is a forgetful functor $\widehat{\Da^\lambda}\to\widehat{\Ca}$ and we define the global sections of $\M\in \widehat{\Da^\lambda}$ to be its global sections as an object of $\widehat{\Ca}$.

As we saw earlier that the $\pi$-adic completion functor sends the category $\Ca$ to the category $\widehat{\Ca}$, it is straightforward to see that $\pi$-adic completion will also send the category $\Da^\lambda$ to $\widehat{\Da^\lambda}$. In particular we have a ``structure sheaf'' $\widehat{\D^\lambda}$. We now check that it represents global sections, which in particular gives that $\Gamma(\widehat{\D^\lambda})$ is a ring.

\begin{lem} Let $\M\in \widehat{\Da^\lambda}$. Then $\Gamma(\M)\cong\text{\emph{Hom}}_{\widehat{\Da^\lambda}}(\widehat{\D^\lambda},\M)$.
\end{lem}

\begin{proof}
By Proposition \ref{completion1} we have isomorphisms
\begin{align*}
\Gamma(\M)\cong \varprojlim \Gamma(\M/\pi^a \M)&\cong \varprojlim \Hom_{\Da^\lambda}(\D^\lambda, \M/\pi^a \M)\\
&\cong \varprojlim \Hom_{\Da^\lambda}(\D^\lambda/\pi^a\D^\lambda, \M/\pi^a \M)
\end{align*}
since global sections in $\Da^\lambda$ are represented by $\D^\lambda$. Now the same argument as in the proof of Proposition \ref{completion1} shows that the latter is isomorphic to $\Hom_{\widehat{\Da^\lambda}}(\widehat{\D^\lambda},\M)$.
\end{proof}

\subsection{Torsion-free coherent modules}\label{generators} Recall that the functor $M\mapsto M(\mu)$ sends objects in $\Da^\lambda$ to $\Da^{\lambda+\mu}$. Given $\M\in\widehat{\Ca}$ and $\mu\in T_P^R$, we define $\M(\mu):=\M\otimes_R R_{-\mu}$ which belongs to $\widehat{\Ca}$ with the tensor comodule structure (since $R_{-\mu}$ is $\pi$-adically complete, it is a $\widehat{\B}$-comodule).

\begin{notation} Given $\M\in \widehat{\Ca}$, we write $\gr_0 \M:=\M/\pi\M\in\Ca$ (by Lemma \ref{Ca}). Following \cite[Section 2.7]{Wadsley1} we call $\gr_0 \M$ the \emph{slice} of $\M$.
\end{notation}

\begin{thm} Assume $p$ is a good prime. Let $\M\in$\emph{coh}$(\widehat{\Da^\lambda})$ be $\pi$-torsion free. Then there is a surjection
$$
\left(\widehat{\D^{\lambda+\mu}}(-\mu)\right)^a\twoheadrightarrow\M
$$
for some $a\geq 1$ and some $\mu\in P$.
\end{thm}

\begin{proof} By Theorem \ref{twists}, $(\gr_0\M)(\mu)$ is $\Gamma$-acyclic and generated by its global sections for $\mu>>0$. Fix such a $\mu$. Then there is a map $(\D^{\lambda+\mu})^a\twoheadrightarrow (\gr_0\M)(\mu)$ in $\Da^{\lambda+\mu}$ for some $a\geq 1$. Now let $\Nn=\M(\mu)$. Since $\M$ has no $\pi$-torsion, so does $\Nn$ and we obtain a surjection
$$
\Gamma(\Nn/\pi^{n+1}\Nn)\twoheadrightarrow \Gamma(\Nn/\pi^n\Nn)\quad \forall n\geq 1
$$
by Lemma \ref{MLcond} since $(\gr_0\M)(\mu)$ is $\Gamma$-acyclic. Hence, starting from a generating set of $(\gr_0\M)(\mu)$ in its global sections, we can inductively  construct $w_1,\ldots, w_a\in\varprojlim \Gamma(\Nn/\pi^n\Nn)$. Since $\Gamma(\Nn)\cong \varprojlim \Gamma(\Nn/\pi^n\Nn)$ by Proposition \ref{completion1}, these elements correspond to global sections in $\Nn$ and they define a map $(\widehat{\D^{\lambda+\mu}})^a\to\Nn$ which must be surjective by the Lemma below, since it is surjective modulo $\pi$. Thus, twisting by $-\mu$, we obtain a map
$$
\left(\widehat{\D^{\lambda+\mu}}(-\mu)\right)^a\twoheadrightarrow\M
$$
as required.
\end{proof}

\begin{lem} If $D$ is a $\pi$-adically complete Noetherian $R$-algebra and $M$ is a finitely generated $D$-module, then $M$ is $\pi$-adically complete. Moreover if $\pi M=M$, then $M=0$.
\end{lem}

\begin{proof}
By \cite[3.2.3]{Berthelot}, $\widehat{M}\cong \widehat{D}\otimes_D M\cong D\otimes_D M\cong M$ so that $M$ is $\pi$-adically complete. If now $M=\pi M$ then $M=\pi^nM$ for all $n\geq 1$ and so $M\cong \widehat{M}=\varprojlim M/\pi^nM=0$.
\end{proof}

\subsection{Analytic quantum $D$-modules}\label{defofdhat} We now move on to define a suitable category of $D$-modules in $\flaghat$. We let $\Dqnhat:=\widehat{\D}\otimes_R L$ and $\Uqbnhat:=\widehat{U^{\geq 0}}\otimes_R L$. Suppose that $A$ is an $L$-Banach algebra, and $H$ is a torsion-free $R$-Hopf algebra. We will say that $A$ is a $\widehat{H_L}$-module algebra if, viewed as a $H_L$-module, it is a module algebra. In such a situation we may then define the smash product algebra $\widehat{H_L}\# A$ to be the completion of the smash product $H_L\# A$.

After extending scalars, we see from our discussion in section \ref{Dhatmod} that $\Oqhat$ is a $\Uqnhat$-module algebra, $\Dqnhat\cong \Oqhat\#\Uqnhat$ is a $\widehat{U^{\text{res}}_L}$-module algebra and there is a continuous algebra homomorphism $\Uqbnhat\to\Dqnhat$. Moreover we have a continuous algebra homomorphism $\Uqbnhat\to \widehat{U^{\text{res}}(\bo)_L}$, and so any Banach $\widehat{U^{\text{res}}(\bo)_L}$-module will also be a Banach $\Uqbnhat$-module. We've already observed that $R_\lambda$ is a $\widehat{\B}$-comodule. Extending scalars, that makes $L_\lambda$ into a $\OqBhat$-comodule. Finally, recall from Theorem \ref{thmA} that a $\OqBhat$-comodule is the same thing as a topologically integrable $\widehat{U^{\text{res}}(\bo)_L}$-module, and comodule homomorphisms are just $\widehat{U^{\text{res}}(\bo)_L}$-linear maps.

\begin{definition}
Let $\lambda\in T_P^R$. A \emph{$(B_q, \lambda)$-equivariant $\Dqnhat$-module} is a triple $(\M, \alpha, \beta)$ where $\M$ is an $L$-Banach space, $\alpha:\ten{\Dqnhat}{\M}\to \M$ is a left $\Dqnhat$-module action and $\beta: \M\to \ten{\M}{\OqBhat}$ is a right $\OqBhat$-comodule action. The map $\beta$ induces a left $\widehat{U^{\text{res}}(\bo)_L}$-action on $\M$ which we also denote by $\beta$. These actions must satisfy:
\begin{itemize}
\item[(i)] The $\Uqbnhat$-actions on $\M\otimes_L L_\lambda$ given by $\beta\otimes \lambda$ and $\alpha\vert_{\Uqbnhat}\otimes 1$ are equal.
\item[(ii)] The map $\alpha$ is $\widehat{U^{\text{res}}(\bo)_L}$-linear with respect to the $\beta$-action on $\M$ and the above action on $\Dqnhat$.
\end{itemize}
The above form a category which we denote by $\widehat{\Dflag}$. We call an object $\M\in\widehat{\Dflag}$ \emph{coherent} if it is finitely generated as a $\Dqnhat$-module. We denote by $\cohDn$ the full subcategory of $\widehat{\Da^\lambda}$ consisting of coherent modules. Since $\Dqnhat$ is Noetherian by Proposition \ref{defn_of_Dq}, this category is abelian.
\end{definition}

Note that condition (ii) in the definition above implies that the restriction of $\alpha$ to $\ten{\Oqhat}{\M}$ is $U^{\text{res}}(\bo)_L$-linear. Since both sides are topologically integrable, this map is therefore a comodule homomorphism. Thus we see that there is a forgetful functor $\widehat{\Dflag}\to \flaghat$.

Recall from Remark \ref{compflag} that any object $\M\in \cohDn$ is canonically equipped with the quotient topology coming from a surjection $(\Dqnhat)^a\to \M$ given by a finite set of generators.

\begin{thm} Let $\M\in \widehat{\Dflag}$. Then $\M$ contains an $R$-lattice $F_0\M$ which is an element of $\widehat{\Da^\lambda}$. Moreover, $F_0\M$ can be chosen to be an element of $\text{\emph{coh}}(\widehat{\Da^\lambda})$ if $\M\in \text{\emph{coh}}(\widehat{\Dflag})$.
\end{thm}

\begin{proof}
By Proposition \ref{basechange2}(ii) and its proof, there exists some $R$-lattice $\Nn\subset \M$ such that $\Nn\in\widehat{\Ca}$ and there is some $m\geq 0$ such that
\begin{equation}\label{useful_eqn}
Nn\subseteq \pi^{-m}\M^\circ.
\end{equation}
We now define $F_0\M$ to be the closure of $\widehat{\D}\cdot \Nn$. By construction it is a $\widehat{\D}$-module and contains $\Nn$. Moreover, since $M^\circ$ is a closed $\widehat{\D}$-submodule of $M$, we get from (\ref{useful_eqn}) that
\begin{equation}\label{useful_eqn2}
F_0\M\subseteq \pi^{-m}M^\circ.
\end{equation}
So we see that $F_0\M$ is a closed lattice inside $\M$, hence is $\pi$-adically complete. We just need to show that it is also a $\widehat{\B}$-comodule, as the compatibility condition will be inherited from the one in $\M$. Note that the $\widehat{\D}$-action on $\M$ is $\widehat{U^{\text{res}}(\bo)}$-linear by the axioms of $\widehat{\Dflag}$, and so we see that $F_0\M$ is a $\widehat{U^{\text{res}}(\bo)}$-submodule of $\M$.

Let $\K:=\D\cdot \Nn$, which is a $U^{\text{res}}(\bo)$-module. Since the action of any $E_{\alpha_i}$ on $\M$ is locally topologically nilpotent, we see that for any $a\geq 1$, the action of any $E_{\alpha_i}^{(r)}$ on $F_0\M/\pi^aF_0\M$ is zero for $r>>0$. Note that every element of $F_0\M/\pi^aF_0\M$ can be represented by an element of $\K$ since $\K$ is dense in $F_0\M$. The PBW basis in $U^{\pm}$ is a weight basis for the adjoint action of $U^{\text{res}}$, and every element of $\A$ is a finite sum of weight vectors since $\A$ is integrable. Thus, by the triangular decomposition for $U$, we see that every element of $\D$ is a finite sum of weight vectors for the action $U^{\text{res}}$. Moreover, by Proposition \ref{equiv}, $\Nn$ is topologically semisimple and so every element is a convergent sum of weight vectors, hence a finite sum of weight vectors modulo $\pi^a$. Thus every element of $\K$ has image in $F_0\M/\pi^aF_0\M$ which is a finite sum of weight vectors. Thus we see that $F_0\M/\pi^aF_0\M$ is integrable, i.e a $\B$-comodule. Passing to the inverse limit, we see that $F_0\M$ is a $\widehat{\B}$-comodule as required.

Now if $\M$ is coherent, then $\M^\circ$ is a quotient of a finite direct sum of $\widehat{\D}$'s. Since $\M^\circ$ is a finitely generated module, we see from (\ref{useful_eqn2}) that $F_0\M$ is as well as $\widehat{\D}$ is Noetherian by Proposition \ref{defn_of_Dq} (in fact $F_0\M=\widehat{\D}\cdot \Nn$ in this case since the right hand side is then finitely generated and so $\pi$-adically complete by \cite[3.2.3(v)]{Berthelot}).
\end{proof}

\subsection{The functor $M\mapsto \widehat{M_L}$}\label{MtoMhat}

The functor $-\otimes_R L$ maps $\coh(\widehat{\Da^\lambda})$ to $\cohDn$ and so by pre-composing with the functor of taking $\pi$-adic completion, we obtain a functor $M\mapsto \widehat{M_L}$ from $\coh(\Da^\lambda)$ to $\cohDn$. Hence in particular we have an object $\widehat{\D_q^\lambda}:=\widehat{\D_L^\lambda}$.

\begin{lem} The global section functor on $\widehat{\Dflag}$ is represented by $\widehat{\D_q^\lambda}$.
\end{lem}

\begin{proof}
By Theorem \ref{defofdhat}, given a module $\M\in \widehat{\Dflag}$, we have a lattice $F_0\M\subseteq \M$ such that $F_0\M$ is in $\Da^\lambda$. The result therefore follows by Proposition \ref{basechange2}(i) and Lemma \ref{Dhatmod}.
\end{proof}

The next few results are analogous to \cite[Theorem 6.6]{Wadsley1}.

\begin{thm} Suppose that $M\in\text{\emph{coh}}(\Da^\lambda)$ is such that $M_L$ is generated by its global sections as a $\D_q$-module. Then $\widehat{M_L}$ is generated by its global sections.
\end{thm}

\begin{proof}
Since $\D_q$ is Noetherian by Proposition \ref{recap2} and since $M_L$ is finitely generated as a $\D_q$-module, we can find $m_1, \ldots, m_a\in\Gamma(M_L)$ generating $M_L$ as a $\D_q$-module. By rescaling we may assume that $m_1, \ldots, m_a$ lie in $\Gamma(M)$ as $\Gamma(M_L)=\Gamma(M)\otimes_R L$ by Proposition \ref{base_change}. Let $\alpha:(\D^\lambda)^a\to M$ be the map in $\coh(\Da^\lambda)$ defined by these global sections, let $N=\im \alpha$, and let $C=\coker(\alpha)\in\coh(\Da^\lambda)$. By assumption, $C_L=0$ and thus there is some $m\geq 0$ such that $\pi^mC=0$. In other words, $\pi^m M\subseteq N$.

Now we have a short exact sequence
$$
0\to N\to M \to C\to 0
$$
in $\coh(\Da^\lambda)$, which induces a tower of short exact sequences
$$
0\to (N+\pi^nM)/\pi^nM\to M/\pi^nM\to C/\pi^n C\to 0
$$
for $n\geq 1$. Since the inverse system $((N+\pi^nM)/\pi^nM)_n$ trivially satisfies the Mittag-Leffler condition, passing to the inverse limit yields a short exact sequence
$$
0\to\varprojlim (N+\pi^nM)/\pi^nM\to \widehat{M}\to\widehat{C}\to 0.
$$
On the other hand, since $\pi^m M\subseteq N$, it follows that for every $n\geq m$, we have $\pi^n N\subseteq N\cap\pi^nM\subseteq \pi^{n-m}N$. Thus we obtain that $\widehat{N}\cong \varprojlim (N+\pi^nM)/\pi^nM$. Moreover, since taking $\pi$-adic completion preserves surjections, $(\widehat{\D^\lambda})^a$ surjects onto $\widehat{N}$. Putting everything together, this gives that the sequence
$$
(\widehat{\D^\lambda})^a\overset{\widehat{\alpha}}{\rightarrow}\widehat{M}\to\widehat{C}\to 0
$$
is also exact. But since $\pi^mC=0$, we have that $\pi^m\widehat{C}=0$ also, and thus $\widehat{C}\otimes_R L=0$. Therefore by applying the exact functor $-\otimes_R L$ to the above exact sequence, we deduce that $\widehat{\alpha}_L:(\widehat{\D_q^\lambda})^a\to\widehat{M_L}$ is surjective, and the result follows by the previous Lemma.
\end{proof}

\begin{cor} Assume that $p$ is a good prime and that $\lambda\in T_P^k$ is regular and dominant. If $\M\in\text{\emph{coh}}(\widehat{\Dflag})$ then $\M$ is generated by its global sections.
\end{cor}

\begin{proof}
By Theorem \ref{defofdhat} there is a lattice $F_0\M$ in $\M$ which is an element of $\coh(\Da^\lambda)$. By Theorem \ref{generators} there is a surjection
$$
\left(\widehat{\D^{\lambda+\mu}(-\mu)}\right)^a\twoheadrightarrow F_0\M
$$
By extending scalars we get a surjection
$$
\left(\widehat{\D^{\lambda+\mu}_q(-\mu)}\right)^a\twoheadrightarrow \M.
$$
But by the Theorem, which applies by the proof of \cite[Theorem 4.12]{QDmod1}, $\widehat{\D^{\lambda+\mu}_q(-\mu)}$ is generated by its global sections. Hence so is $\M$ by the Lemma.
\end{proof}

\subsection{Exactness of global sections}\label{acyclic2} We can now prove that global sections is exact on $\cohDn$. All the results from now on will require the assumption that $p$ is a good prime.

\begin{thm} Assume $p$ is a good prime. Then for any $\lambda\in T_P^k$, $\widehat{\D_q^\lambda}$ is $\Gamma$-acyclic.
\end{thm}

\begin{proof} By Corollary \ref{twists} and Lemma \ref{MLcond}, the projective system $R^i\Gamma(\D^\lambda/\pi^n\D^\lambda)_n$ satisfies the Mittag-Leffler condition for each $i\geq 0$. So by \cite[Proposition 0.13.2.3]{EGAIII.1} applied to $\check{C}^{\text{aug}}(\D^\lambda/\pi^n\D^\lambda)$ and by Proposition \ref{Cech} combined with Remark \ref{ore2}, we get that
$$
H^i\check{C}^{\text{aug}}(\widehat{\D^\lambda})\cong \varprojlim R^i\Gamma(\D^\lambda/\pi^n\D^\lambda)
$$
for all $i\geq 1$. Hence by Corollary \ref{twists}, $H^i\check{C}^{\text{aug}}(\widehat{\D^\lambda})=0$, and so by Theorem \ref{thmB} we have that
$$
R^i\Gamma(\widehat{\D_q^\lambda})\cong H^i\check{C}^{\text{aug}}(\widehat{\D^\lambda})\otimes_R L=0
$$
for $i\geq 1$ as required.
\end{proof}

\begin{cor} Suppose $\lambda$ is regular and dominant. If $\M\in\text{\emph{coh}}(\widehat{\Dflag})$ then $\M$ is $\Gamma$-acyclic and $\Gamma(\M)$ is a finitely presented module over $D:=\Gamma(\widehat{\D_q^\lambda})$.
\end{cor}

\begin{proof} By Corollary \ref{MtoMhat} and since $\Dqnhat$ is Noetherian by Proposition \ref{defn_of_Dq}, we can find a resolution
$$
\F_N\overset{f_N}{\to} \F_{N-1}\overset{f_{N-1}}{\to}\cdots\overset{f_1}{\to} \F_0\overset{f_0}{\to} \M\to 0
$$ 
where the $\F_i$ are all of the form $(\widehat{\D_q^\lambda})^a$. Then by the Theorem each $\F_i$ is $\Gamma$-acyclic. Let $\M_i=\ker{f_i}$. Then by the long exact sequence on cohomology, we get
$$
R^j\Gamma(\M)=R^{j+1}\Gamma(\M_0)=R^{j+2}\Gamma(\M_1)=\cdots=R^{j+N+1}\Gamma(\M_N)=0
$$
for every $j>0$ by Corollary \ref{Cech2}. So $\M$ is $\Gamma$-acyclic.

For the last part, as above we have a resolution
$$
(\widehat{\D_q^\lambda})^a\to(\widehat{\D_q^\lambda})^b\to\M\to 0
$$
for some $a, b\geq 0$. By the above the global section functor is exact on $\cohDn$, so it follows that we have a resolution
$$
D^a\to D^b\to \Gamma(\M)\to 0
$$
making $\Gamma(\M)$ finitely presented as required.
\end{proof}

\subsection{The localisation functor}\label{BB2} By Corollary \ref{acyclic2}, we get a functor
$$
\Gamma:\cohDn\to \text{f.p }D\text{-mod}.
$$
Now let $M$ be a finitely presented (or even finitely generated) $D$-module. Then we define its \emph{localisation} to be the $\Dqnhat$-module
$$
\loc_\lambda(M)=\widehat{\D_q^\lambda}\otimes_D M.
$$
Note that this is a finitely generated $\Dqnhat$-module so is automatically Banach, and in fact a quotient of a finite direct sum of $\widehat{\D_q^\lambda}$. This defines a well-defined $\OqBhat$-comodule structure on $\loc_\lambda(M)$ because $D$ is a trivial subcomodule of $\widehat{\D_q^\lambda}$. This therefore defines a functor
$$
\loc_\lambda:\text{f.p }D\text{-mod}\to\cohDn,
$$
which is a left adjoint to global section. Indeed the adjunction morphisms are given as follows: the map $\loc_\lambda(\Gamma(\M))\to \M$ is given by the $\Dqnhat$-action and is well-defined because $\widehat{\D_q^\lambda}$ represents global sections, while the map $M\to \Gamma(\loc_\lambda(M))$ is just the natural map $m\mapsto 1\otimes m$.

We can now get that the category $\cohDn$ is $D$-affine:

\begin{thm} Suppose that $\lambda\in T_P^k$ is regular and dominant, and assume $p$ is a good prime. Then the functor $\Gamma:\cohDn\to \text{f.p }D\text{-mod}$ is an equivalence of categories with quasi-inverse $\loc_\lambda$.
\end{thm}

\begin{proof} The proof is now standard. We have to show that the two adjunction morphisms are isomorphisms. We first show that the adjunction morphism
$$
\psi_M:M\to\Gamma(\loc_\lambda(M))
$$
is an isomorphism for any finitely presented $D$-module $M$. By definition of $\loc_\lambda$, it is clear that $\psi_M$ is an isomorphism whenever $M$ is a free $D$-module. For a general finitely presented $M$, we have a presentation $D^a\to D^b\to M\to 0$. Since the global sections functor is exact by Corollary \ref{acyclic2} and since $\loc_\lambda$ is right exact, applying $\Gamma\circ \loc_\lambda$ to this presentation and the Five Lemma yields that $\psi_M$ is an isomorphism for any finitely presented module $M$.

It remains to show that the other adjunction morphism
$$
\phi_{\M}:\loc_\lambda(\Gamma(\M))\to\M
$$
is an isomorphism for every $\M\in\cohDn$. By exactness of $\Gamma$ and since every object of $\cohDn$ is generated by its global sections by Corollary \ref{MtoMhat}, the map $\phi_{\M}$ is an isomorphism if and only if $\Gamma(\phi_{\M})$ is an isomorphism.  But the composite
$$
\Gamma(\M)\overset{\psi_{\Gamma(\M)}}{\longrightarrow}\Gamma(\loc_\lambda(\Gamma(\M)))\overset{\Gamma(\phi_{\M})}{\longrightarrow}\Gamma(\M)
$$
is an isomorphism by general properties of adjunctions, and the map $\psi_{\Gamma(\M)}$ is an isomorphism by the above, therefore $\Gamma(\phi_{\M})$ is an isomorphism.
\end{proof}

\subsection{Global sections of $\widehat{\D_q^\lambda}$}\label{global1} We recall from \cite[Proposition 4.8]{QDmod1} the construction of a canonical map $U^{\text{fin}}\to \Gamma(\D^\lambda)$. First, it follows from Lemma \ref{intfinite1} that the canonical surjection $U\to \M_\lambda$ restricts to a surjection $U^{\text{fin}}\to \M_\lambda$. That map is a homomorphism of $\B$-comodules, where $\M_\lambda$ and $U^{\text{fin}}$ are viewed as integrable $U^{\text{res}}(\bo)$-modules via the adjoint action. This gives rise to a surjection
\begin{equation}\label{map1}
p^*(U^{\text{fin}})=\A\otimes_RU^{\text{fin}}\to \A\otimes_R \M_\lambda\cong \D^\lambda.
\end{equation}
But since $U^{\text{fin}}$ is in fact an $\A$-comodule, i.e. an integrable $U^{\text{res}}$ via the adjoint action, it follows from the tensor identity (e.g. \cite[Proposition 2.16]{Andersen}) that
$$
U^{\text{fin}}\cong \ind(U^{\text{fin}})=\Gamma(p^*(U^{\text{fin}})).
$$
This isomorphism is explicitly given by $u\mapsto \sum S^{-1}(u_2)\otimes u_1$ where we use Sweedler's notation for the coaction $U^{\text{fin}}\to U^{\text{fin}}\otimes \A$ induced by the adjoint action of $U^{\text{res}}$ on $U^{\text{fin}}$ (cf \cite[Lemma 3.13(a)]{QDmod1}).

Now, taking global sections in (\ref{map1}) and using this isomorphism, we obtain a map
$$
f:U^{\text{fin}}\to \Gamma(\D^\lambda)
$$
as promised. The above construction can of course be repeated, for any $n\geq 1$, by replacing $U^{\text{fin}}$ and $\D^\lambda$ by $U^{\text{fin}}/\pi^nU^{\text{fin}}$ and $\D^\lambda/\pi^n\D^\lambda$ respectively to obtain maps
$$
f_n: U^{\text{fin}}/\pi^nU^{\text{fin}}\to \Gamma (\D^\lambda/\pi^n\D^\lambda).
$$
By functoriality of the tensor identity, we see that the maps $f_n$ satisfy the property that the diagrams

\begin{minipage}{0.5\textwidth}
\[
\begin{tikzcd}
U^{\text{fin}} \arrow{d}\arrow{r}{f} & \Gamma(\D^\lambda) \arrow{d}\\
U^{\text{fin}}/\pi^nU^{\text{fin}} \arrow{r}{f_n} & \Gamma (\D^\lambda/\pi^n\D^\lambda)
\end{tikzcd}
\]
\end{minipage}\hfill
\begin{minipage}{0.5\textwidth}
\[
\begin{tikzcd}
U^{\text{fin}}/\pi^mU^{\text{fin}} \arrow{d}\arrow{r}{f_m} & \Gamma(\D^\lambda/\pi^m\D^\lambda) \arrow{d}\\
U^{\text{fin}}/\pi^nU^{\text{fin}} \arrow{r}{f_n} & \Gamma (\D^\lambda/\pi^n\D^\lambda)
\end{tikzcd}
\]
\end{minipage}
commute for all $m\geq n$. Thus we may take the inverse limit of the maps $f_n$ to obtain a map
$$
\widehat{f}:=\varprojlim f_n: \widehat{U^{\text{fin}}}\to \Gamma(\widehat{\D^\lambda}).
$$
Note that $I_\lambda:=(\ker(\chi_\lambda))$ is contained in the kernel of the surjection $U^{\text{fin}}\to \M_\lambda$. Moreover, this ideal is preserved under the adjoint action of $U^{\text{res}}$ since it is centrally generated. Thus it is an $\A$-subcomodule of $U^{\text{fin}}$ and hence, by the above construction, is in the kernel of $f:U^{\text{fin}}\to \Gamma(\D^\lambda)$. By the commutativity of the left hand diagram above, it follows that $(I_\lambda+\pi^n U^{\text{fin}})/\pi^n U^{\text{fin}}$ lies in the kernel of $f_n$ for all $n\geq 1$.

Now let $U^\lambda:=U^{\text{fin}}/I_\lambda$ and let $\widehat{U^\lambda_q}:=\widehat{U^\lambda_L}$. Then there is a short exact sequence
$$
0\to \varprojlim (I_\lambda+\pi^n U^{\text{fin}})/\pi^n U^{\text{fin}}\to \widehat{U^{\text{fin}}}\to \widehat{U^\lambda}\to 0
$$
so that we see from the above that $\widehat{f}$ factors through a map $\widehat{U^\lambda}\to \Gamma(\widehat{\D^\lambda})$. Extending scalars this gives a map
$$
\widehat{f_\lambda}:\widehat{U_q^\lambda}\to \Gamma(\widehat{\D_q^\lambda})
$$
Let $V$ denote the unit ball of $\widehat{U_q^\lambda}$, i.e. the image of $\widehat{U^{\text{fin}}}$ in $\widehat{U_q^\lambda}$. In order to study the above map, it will be useful to consider the restriction modulo $\pi$ of $V\to \Gamma(\widehat{\D^\lambda})$, which induces a map
$$
\overline{f_\lambda}:V/\pi V\to U(\g_k)_{\chi_\lambda}
$$
by Corollary \ref{cohomology} whenever $p$ is very good and $\lambda\in T_P^k$. Also, recall the quantum Harish-Chandra isomorphism $\varphi:Z(U)\to (U^0\cap L(2P))^W$ as well as the classical Harish-Chandra isomorphism $\psi: U(\g_k)^{G_k}\to S(\h_k)^W$ from section \ref{HC}.

\begin{prop} Let $\lambda\in T_P^k$ and assume that $p$ is a very good prime. Then $\overline{f_\lambda}$ is surjective. If furthermore $p$ does not divide $|W|$, then it is an isomorphism.
\end{prop}

\begin{proof}
Recall that Corollary \ref{intfinite2} says that $\Ufbar\cong U(\g_k)\otimes_k k(2P/2Q)$. In particular, there is a projection map $\theta:\Ufbar\to U(\g_k)$ which induces a canonical surjection $\Ufbar\to U(\g_k)_{\chi_\lambda}$ which we also denote by $\theta$ by abuse of notation. But the map $\overline{f_\lambda}$ is factored from the map $\Ufbar\to \Gamma(\D^\lambda_k)$ induced by $\widehat{f}:\widehat{U^{\text{fin}}}\to \Gamma(\widehat{\D^\lambda})$, and upon identifying $\Gamma(\D^\lambda_k)$ with $U(\g_k)_{\chi_\lambda}$, via Corollary \ref{cohomology}, this map becomes $\theta$. Thus the surjectivity of $\theta$ implies the surjectivity of $\overline{f_\lambda}$.

So we see that we have a commutative diagram
\[
\begin{tikzcd}
 & U^{\text{fin}}_k \arrow[swap]{ld}{\theta}\arrow{d}\arrow{rd}{\theta} & \\
U(\g_k)_{\chi_\lambda} \arrow[dashrightarrow]{r}{g} & V/\pi V \arrow{r}{\overline{f_\lambda}} & U(\g_k)_{\chi_\lambda}
\end{tikzcd}
\]
It would now suffice to show that the surjection $U^{\text{fin}}_k\to V/\pi V$ factors through $U(\g_k)_{\chi_\lambda}$ via a map $g$ as in the diagram, whenever $p$ does not divide $|W|$. Indeed, any $u\in \ker \overline{f_\lambda}$ lifts to some $v\in \ker \theta$, and if there is such a factorisation it then follows that $u=g(\theta(v))=0$.

Let $\mu\in 2P$. By Theorem \ref{HC} there exists $z_\mu\in Z(U)$ such that $\varphi(z_\mu)=\sum_{\gamma\in W\mu}K_\gamma$. Now $z_\mu-\sum_{\gamma\in W\mu}K_\gamma\in U$ reduces modulo $\pi$ to an element of $(\Ufbar)^{G_k}\cong U(\g_k)\otimes_k k(2P/2Q)$, using Corollary \ref{Gkmod}, and is in the kernel of $\psi\otimes \text{id}$ by commutativity of the square in Theorem \ref{HC}. But since $\psi\otimes \text{id}$ is an isomorphism by \cite[Theorem 9.3]{Jantzen3}, it follows that
$$
z_\mu\equiv \sum_{\gamma\in W\mu}K_\gamma\pmod{\pi}.
$$
Moreover, by Proposition \ref{Weyl}, we have that $K_\gamma\equiv K_\mu\pmod{\pi}$ for all $\gamma\in W\mu$ and so we in fact get
\begin{equation}\label{neweqn3}
z_\mu\equiv |W\mu|\cdot K_\mu \pmod{\pi}.
\end{equation}
Now, by definition of $V$ and of $\chi_\lambda$, the image of $z_\mu\in U^{\text{fin}}$ in $V$ is $\chi_\lambda(z_\mu)=\sum_{\gamma\in W\mu}q^{\langle \lambda, \gamma\rangle}$. Hence, since $q^{\frac{1}{d}}\equiv 1\pmod{\pi}$, its image in $V/\pi V$ is $|W\mu|$. By our assumption on $p$, $|W\mu|\neq 0$ in $k$ and thus from (\ref{neweqn3}) we see that $K_\mu\in \Ufbar$ gets mapped to 1 in $V/\pi V$.

Since $\mu$ was arbitrary and $\Ufbar\cong U(\g_k)\otimes_k k(2P/2Q)$ by Corollary \ref{intfinite2}, it follows that the map $U^{\text{fin}}_k\to V/\pi V$ factors through $U(\g_k)$. To get that it factors further through $U(\g_k)_{\chi_\lambda}$, pick $u\in U(\g_k)^{G_k}$. Since $U(\g_k)^{G_k}\otimes_k k(2P/2Q)\cong Z(U)/\pi Z(U)$ by Theorem \ref{HC}, we may lift $u$ to an element of $Z(U)$. So the image of $u$ in $V$ is equal to $\chi_\lambda(u)$ as required, and we are done.
\end{proof}

We can now prove our main result:

\begin{thm} Let $\lambda\in T_P^k$ and assume that $p$ is a very good prime. Then the map $\widehat{f_\lambda}:\widehat{U_q^\lambda}\to \Gamma(\widehat{\D_q^\lambda})$ is surjective. In particular, $\Gamma(\widehat{\D_q^\lambda})$ is Noetherian. If furthermore $p$ does not divide the order of the Weyl group $W$, then $\widehat{f_\lambda}$ is an isomorphism.
\end{thm}

\begin{proof}
We first show that each map $f_{\lambda, n}:V/\pi^n V\to \Gamma(\D^\lambda/\pi^n \D^\lambda)$ is surjective, respectively an isomorphism, depending on the condition on $p$. Here $f_{\lambda, 1}$ is the map $\overline{f_\lambda}$ from earlier. We proceed by induction on $n$. The case $n=1$ follows from the previous Proposition. In general, the short exact sequences
$$
0\to V/\pi V\overset{\cdot\pi^n}{\to} V\pi^{n+1} V\to V/\pi^nV\to 0
$$
and
$$
0\to \D^\lambda_k\overset{\cdot\pi^n}{\to} \D^\lambda/\pi^{n+1} \D^\lambda\to \D^\lambda /\pi^n\D^\lambda\to 0
$$
give rise to a commutative diagram with exact rows
\[
\begin{tikzcd}
0\arrow{r} &  V \arrow{r}{\cdot\pi^n} \arrow{d}{\overline{f_\lambda}} & V/\pi^{n+1} V\arrow{r} \arrow{d}{f_{\lambda, n+1}} & V/\pi^nV\arrow{r}\arrow{d}{f_{\lambda, n}} & 0\\
0\arrow{r} & \Gamma(\D^\lambda_k)\arrow{r}{\cdot\pi^n} & \Gamma(\D^\lambda/\pi^{n+1} \D^\lambda) \arrow{r} & \Gamma(\D^\lambda /\pi^n\D^\lambda) \arrow{r} & 0
\end{tikzcd}
\]
where we used the acyclicity of $\D^\lambda_k$ (Corollary \ref{cohomology}). The induction hypothesis and the Five Lemma then give that $f_{n+1}$ is surjective, respectively an isomorphism.

If $p$ does not divide $|W|$, we immediately obtain that $f_\lambda=\varprojlim f_{\lambda, n}$ is an isomorphism. So we now assume that $p$ is only a very good prime.

Let $K_n=\ker(f_{\lambda, n})$. Then we have a commutative diagram
\[
\begin{tikzcd}
 & 0\arrow{d} & 0\arrow{d} & 0\arrow{d} & \\
0 \arrow{r} & K_1 \arrow{r}{\cdot\pi^n} \arrow{d} & K_{n+1}\arrow{r}\arrow{d} & K_n \arrow{r}\arrow{d} & 0\\
0\arrow{r} &  V \arrow{r}{\cdot\pi^n} \arrow{d}{\overline{f_\lambda}} & V/\pi^{n+1} V\arrow{r} \arrow{d}{f_{\lambda, n+1}} & V/\pi^nV\arrow{r}\arrow{d}{f_{\lambda, n}} & 0\\
0\arrow{r} & \Gamma(\D^\lambda_k)\arrow{r}{\cdot\pi^n}\arrow{d} & \Gamma(\D^\lambda/\pi^{n+1} \D^\lambda) \arrow{r}\arrow{d} & \Gamma(\D^\lambda /\pi^n\D^\lambda) \arrow{r}\arrow{d} & 0\\
 & 0 & 0 & 0 & 
\end{tikzcd}
\]
where the bottom two rows and all columns are exact. Then a simply diagram chase gives that the top row is exact. In particular, the maps $K_{n+1}\to K_n$ are surjective and so $(K_n)_{n\geq 1}$ satisfies the Mittag-Leffler condition. Passing to the inverse limit and using Proposition \ref{completion1}, we therefore see that the sequence
$$
0\to \varprojlim K_n\to V\overset{\widehat{f_\lambda}}{\to} \Gamma(\widehat{\D^\lambda})\to 0
$$
is exact and $\widehat{f_\lambda}$ is surjective as claimed.

Finally, the Noetherianity of $\Gamma(\widehat{\D_q^\lambda})$ now follows because $\widehat{U^\lambda}$ is Noetherian by Corollary \ref{intfinite2} and Lemma \ref{Noetherian}(i).
\end{proof}

This gives our promised result from the introduction:

\begin{proof}[Proof of Theorem \ref{thmC}]
Put together Theorem \ref{BB2} and Theorem \ref{global1}. Note that the Noetherianity of $\Gamma(\widehat{\D_q^\lambda})$ implies that finitely presented modules over it are the same as finitely generated modules.
\end{proof}

The proof of Theorem \ref{global1} also has the following consequence:

\begin{cor}
Suppose that $p$ is a very good prime. Then $\Gamma(\widehat{\D^\lambda})\cong\widehat{\Gamma(\D^\lambda)}$. Suppose additionally that $\lambda$ is dominant. Then $\D^\lambda\in\Ca$ is $\Gamma$-acyclic.
\end{cor}

\begin{proof}
By commutativity of the square
\[
\begin{tikzcd}
U^{\text{fin}} \arrow{d}\arrow{r}{f} & \Gamma(\D^\lambda) \arrow{d}\\
U^{\text{fin}}/\pi^nU^{\text{fin}} \arrow{r}{f_n} & \Gamma (\D^\lambda/\pi^n\D^\lambda)
\end{tikzcd}
\]
and since $f_n$ is surjective (the map $f_{\lambda, n}$ from the last proof is factored from $f_n$), it follows that the canonical map $\Gamma(\D^\lambda)\to \Gamma (\D^\lambda/\pi^n\D^\lambda)$ is surjective for all $n\geq 1$. But by exactness of
$$
0\to \Gamma(\D^\lambda)\overset{\cdot \pi^n}{\to}\Gamma(\D^\lambda)\to\Gamma(\D^\lambda/\pi^n\D^\lambda)
$$
it follows that $\Gamma(\D^\lambda/\pi^n\D^\lambda)\cong \Gamma(\D^\lambda)/\pi^n\Gamma(\D_\lambda)$ and so, by Proposition \ref{completion1}, that
$$
\Gamma(\widehat{\D^\lambda})\cong\varprojlim \Gamma(\D^\lambda/\pi^n\D^\lambda)\cong \widehat{\Gamma(\D^\lambda)}
$$
as required.

Next, the short exact sequence
$$
0\to \D^\lambda \overset{\cdot \pi}{\to} \D^\lambda\to \D^\lambda_k\to 0
$$
induces a long exact sequence
$$
\cdots\to R^{i-1}\Gamma(\D^\lambda)\to R^{i-1}\Gamma(\D^\lambda_k)\to R^i\Gamma(\D^\lambda)\overset{\cdot \pi}{\to} R^i\Gamma(\D^\lambda)\to R^i\Gamma(\D^\lambda_k)\to\cdots
$$
From this we deduce that multiplication by $\pi$ yields an automorphism of $R^i\Gamma(\D^\lambda)$ for all $i\geq 1$. For $i\geq 2$, this follows immediatly from the acyclicity of $\D^\lambda_k$ (Corollary \ref{cohomology}). For $i=1$, we additionally need the surjectivity of $\Gamma(\D^\lambda)\to \Gamma (\D^\lambda_k)$ obtained above. However, $R^i\Gamma(\D^\lambda)\otimes_R L=R^i\Gamma(\D_q^\lambda)=0$ since $\lambda$ is dominant, by Proposition \ref{base_change} and by the proof of \cite[Theorem 4.12]{QDmod1}. Hence $R^i\Gamma(\D^\lambda)$ is $\pi$-torsion for all $i\geq 1$. From the above, it follows that $R^i\Gamma(\D^\lambda)=0$ for all $i\geq 1$.
\end{proof}

We also believe that large enough twists of $\D^\lambda$ should be acyclic, which would remove the $\pi$-torsion condition in Theorem \ref{twists}. But it is not clear how to prove that, given that we don't necessarily know what the global sections of these twists are modulo $\pi$.

\appendix

\section{Hopf duals of $R$-Hopf algebras and their comodules}

We wish to establish some duality facts to do with Hopf algebras over $R$ which are well known when working over fields but for which we couldn't find references for Hopf algebras over more general commutative rings. Most of our proofs work using the usual arguments but one has to be a bit careful when dealing with torsion.

\subsection{Hopf duals over $R$}\label{hopfdual}
For the entirety of this Section, $H$ will denote a fixed Hopf $R$-algebra. For our purposes, it will be enough to work in the case where $H$ is \emph{torsion-free}. First we wish to define a notion of Hopf dual. Since $H$ has no torsion, it embeds as a sub-Hopf algebra of $H_L=H\otimes_R L$. We will define the Hopf dual to be a sub-Hopf algebra of $(H_L)^\circ$. Let $\mathcal{J}$ denote the set of ideal $I$ in $H$ such that $H/I$ is a finitely generated $R$-module. Moreover, denote by $\mathcal{J}'$ the set of ideals $I$ in $H$ such that $H/I$ is free of finite rank. Finally, write $H^*$ for $\Hom_R(H, R)$. Note that $H^*$ is always torsion-free since $R$ is a domain: if $\pi f=0$ then $\pi f(u)=0$ for all $u\in H$ and so $f(u)=0$ for all $u$.

\begin{definition} We define the Hopf dual of $H$ to be
$$
H^\circ:=\{f\in H^* : f|_I=0 \text{ for some } I\in \mathcal{J}\}.
$$
By the above $H^\circ$ is torsion-free.
\end{definition}

If $n\geq 0$ and $x\in H$ we have for any $f\in H^*$ that $f(x)=0$ if and only if $f(\pi^n x)=0$. Thus if $0\neq f\in H^\circ$ then $f|_I=0$ for some $I\in\mathcal{J}$ where $H/I$ is not torsion. Moreover we then have $f|_{I_L\cap H}=0$ and so by replacing $I$ with $I_L\cap H$ we may in addition assume that $H/I$ is torsion-free. Since $R$ is a PID this shows that
$$
H^\circ=\{f\in H^* : f|_I=0 \text{ for some } I\in \mathcal{J}'\}.
$$
Moreover by extending scalars we may identify $H^\circ$ with an $R$-submodule of $H_L^\circ$. From this it follows by the standard arguments that $H^\circ$ is the algebra of matrix coefficients of $H$-modules which are free of finite rank over $R$. Since this collection of $H$-modules is closed under taking tensor products, direct sums and duals, and we can take dual bases, we have proved

\begin{lem} $H^\circ$ is an sub-Hopf $R$-algebra of $H_L^\circ$. In particular the algebra maps on $H^\circ$ are just the dual maps of the coalgebra maps on $H$ and vice-versa.
\end{lem}

\begin{remark} Some of the above arguments were implicit in Lusztig's work, see \cite[7.1]{Lusztig1}.
\end{remark}

\subsection{$H^\circ$-comodules as $H$-modules}\label{comod}

We now wish to establish some correspondence between comodules over $H^\circ$ and certain $H$-modules. We call an $H$-module $M$ \emph{locally finite} if for all $m\in M$, $Hm$ is finitely generated over $R$.

\begin{prop} Every $H^\circ$-comodule has a canonical structure of a locally finite $H$-module with respect to which every comodule homomorphism is an $H$-modules homomorphism. In other words there is a canonical faithful embedding of categories between the category of $H^\circ$-comodules and the category of locally finite $H$-modules. 
\end{prop}

\begin{proof} This is just the usual argument. If $M$ is an $H^\circ$-comodule with coaction $\rho:M\to M\otimes_R H^\circ$, write $\rho(m)=\sum m_1\otimes m_2$. Then we set
$$
u\cdot m=\sum m_2(u)m_1
$$
for all $u\in H$. It follows from the comodule axioms that this gives a well defined module structure, i.e that $1\cdot m=m$ and that $u\cdot(u'\cdot m)=(uu')\cdot m$ for all $u,u'\in H$ and all $m\in M$. Moreover by definition of the module structure, $H\cdot m$ is finitely generated over $R$ for all $m\in M$. Finally it follows from the definition of the action that any comodule homomorphism is also a module homomorphism.
\end{proof}

Next, we want to show that the functor we just defined is full, i.e that every $H$-module map between two $H^\circ$-comodules is a comodule homomorphism. We first need a technical result. Suppose $M$ is a locally finite $H$-module. Note that we have an $R$-module injection $\phi_M:M\to\Hom_R(H,M)$ given by $\phi_M(m)(u)=um$ for all $u\in H$ and $m\in M$. Moreover we have a map
$$
\theta_M: M\otimes_R H^\circ \to \Hom_R(H,M)
$$
given by $\theta_M(m\otimes f)(u)=f(u)m$. When the $H$-module structure on $M$ arises from an $H^\circ$-comodule structure then we have $\phi_M=\theta_M\circ \rho$. Therefore we can use this expression for $\phi_M$ as an alternative definition of the module structure on $M$. We claim that the map $\theta_M$ is injective. More generally we have the following

\begin{lem} Let $A$ and $B$ be $R$-modules, $A^*=\Hom_R(A,R)$ and suppose $C$ is any $R$-submodule of $A^*$ such that $A^*/C$ has no $\pi$-torsion. Let $M$ be any $R$-module and set
$$
\theta_{M,C}:\Hom_R(B,M)\otimes_R C \to \Hom_R(A\otimes_R B,M)
$$
to be defined by $\theta_{M,C}(g\otimes f)(x\otimes y)=f(x)g(y)$. Then the map $\theta_{M,C}$ is injective.
\end{lem}

\begin{proof} Suppose that $0\neq u=\sum_{i=1}^s g_i\otimes f_i\in \Hom_R(B,M)\otimes_R C$. The $R$-submodule $N$ of $\Hom_R(B,M)$ generated by the $g_i$ is finitely generated, so since $R$ is a PID we can pick a generating set $n_1,\ldots, n_l, t_1,\ldots, t_m$ for $N$ such that $n_1,\ldots, n_l$ are torsion-free while $t_1,\ldots, t_m$ are $\pi$-torsion, and
$$
N=\bigoplus_{i=1}^l Rn_i \oplus \bigoplus_{j=1}^m Rt_j.
$$
For each $1\leq j\leq m$, let $a_j$ be the positive integer such that $Rt_j\cong R/\pi^{a_j}R$.

Now, to show that $\theta_{M,C}(u)\neq 0$ it suffices to show that the restriction of $\theta_{M,C}$ to the span of the $n_i\otimes f_j$ and $t_k\otimes f_j$ is injective. So suppose we are given
$$
v=\sum r_{ij} n_i\otimes f_j +\sum r'_{kj} t_k\otimes f_j\in \ker{\theta_{M,C}}.
$$
Evaluating at $x\otimes y$ we get $\sum_{i,j} r_{ij}f_j(x)n_i(y) + \sum_{k,j} r'_{kj} f_j(x)t_k(y)=0$ for all $x\in A$ and $y\in B$. In particular we have $\sum_{i,j} r_{ij}f_j(x)n_i + \sum_{k,j} r'_{kj} f_j(x)t_k=0$ for any fixed $x\in A$. Since we have a direct sum decomposition of $N$ it follows that
$$
\sum_{j} r_{ij}f_j(x)=0 \quad\text{and}\quad \sum_{j} r'_{kj} f_j(x)\in \pi^{a_k}R
$$
for all $x\in A$ and all $1\leq i\leq l$ and $1\leq k\leq m$. In particular, for all $k$, $\sum_{j} r'_{kj} f_j=\pi^{a_k}g_k$ for some $g_k\in C$ since $A^*/C$ has no $\pi$-torsion.

Therefore we have
$$
\sum_{j} r_{ij}f_j=0 \quad\text{and}\quad \sum_{j} r'_{kj} f_j\in \pi^{a_k}C,
$$
and hence
$$
n_i\otimes \sum_{j} r_{ij}f_j=0=t_k\otimes \sum_{j} r'_{kj} f_j
$$
for all $i,k$, and so $v=0$ as required.
\end{proof}

\begin{cor} Let $M$ be an $R$-module.
\begin{enumerate}
\item The map $\theta_M: M\otimes_R H^\circ \to \Hom_R(H,M)$ is injective.
\item The map $M\otimes_R H^\circ\otimes_R H^\circ\to\Hom_R(H\otimes_R H,M)$ sending $m\otimes f\otimes g$ to $x\otimes y\mapsto  f(x)g(y)m$ is injective.
\end{enumerate}
\end{cor}

\begin{proof}
Let $A=H$ and $C=H^\circ$. From the definition of $H^\circ$ it follows that $A^*/C$ is torsion-free. Then (i) follows immediately from the Lemma by putting $B=R$. For (ii) note that this map is simply the composite
\[
\begin{tikzcd}
M\otimes_R H^\circ\otimes_R H^\circ \arrow{r}{\theta_M\otimes 1} & \Hom_R(H,M)\otimes_R H^\circ\arrow{r}{\varpi} & \Hom_R(H\otimes_R H,M)
\end{tikzcd}
\]
where $\varpi(f\otimes g)(x\otimes y)=g(y)f(x)$. The map $\theta_M\otimes 1$ is injective by (i) and because $H^\circ$ is flat while the map $\varpi$ is injective by putting $B=H$ in the Lemma.
\end{proof}

We can now deduce the result we were aiming for.

\subsection{Proposition}\label{fullyfaithful} \emph{The functor associating any $H^\circ$-comodule to the corresponding $H$-module is a fully faithful embedding.}

\begin{proof} From what we have done already we just need to show that any $H$-module map $f:M\to N$ between two $H^\circ$-comodules is a comodule homomorphism. Write $\rho_M$ and $\rho_N$ for the coactions on $M$ and $N$ respectively, and pick $m\in M$ and $u\in H$. Then we know that $um=\sum m_2(u)m_1$ and we have $uf(m)=\sum m_2(u)f(m_1)$ since $f$ is a module homomorphism. On the other hand by definition of the action on $N$ we have $uf(m)=\sum f(m)_2(u)f(m)_1$. To show that $f$ is a comodule map we need to show that
$$
\sum f(m_1)\otimes m_2=\sum f(m)_1\otimes f(m)_2
$$
or in other words that $\rho_N\circ f=(f\otimes 1)\circ\rho_M$.

Write $\tilde{\rho}_1=\rho_N\circ f$ and $\tilde{\rho}_2=(f\otimes 1)\circ\rho_M$. Moreover recall the map $\phi: M\to\Hom_R(H,M)$ given by $\phi(m)(u)=um$. Then let
$$
\tilde{\phi}=\phi\circ f:M\to\Hom_R(H, N)
$$
so that $\tilde{\phi}(m)(u)=uf(m)$. Then by definition $\tilde{\phi}=\theta_N\circ\tilde{\rho}_1$. On the other hand by our above observation we see that $\tilde{\phi}=\theta_N\circ\tilde{\rho}_2$. Since $\theta_N$ is injective the result follows.
\end{proof}

From now on, if $M$ is a locally finite $H$-module we will say that it is an $H^\circ$-comodule to mean that its $H$-module structure arises from an $H^\circ$-comodule structure.

In order for the above functor to be an isomorphism of categories we therefore just need to show that it is surjective. This may not be true in general, however we can write a very simple necessary and sufficient condition for an isomorphism of categories to hold. Suppose $M$ is a locally finite $H$-module and let $\phi_M: M\to \Hom_R(H,M)$ be given by $\phi_M(m)(x)=x\cdot m$. We have the map $\theta_M:M\otimes_R H^\circ$ as before.

\subsection{Proposition}\label{key_hopf_fact} \emph{A locally finite $H$-module $M$ is an $H^\circ$-comodule if and only if $\phi_M(m)$ belongs to the image of $\theta_M$ for all $m\in M$.}

\begin{proof} If $M$ is a comodule with coaction $\rho$, then by our observation preceding Lemma \ref{comod} we have $\phi_M=\theta_M\circ\rho$ where $\phi_M$ comes from the induced $H$-module structure, and the result is clear. Conversely assume $\phi_M(m)$ belongs to the image of $\theta_M$ for all $m\in M$. Fix $m\in M$. Then there exists $m_1, \ldots, m_n\in M$ and $f_1, \ldots, f_n\in H^\circ$ such that for all $x\in H$, $x\cdot m=\sum_{i=1}^n f_i(x)m_i$ and we define
$$
\rho(m)=\sum_{i=1}^n m_i\otimes f_i,
$$
i.e $\rho(m)$ is the unique element of $M\otimes_R H^\circ$ such that $\theta_M(\rho(m))=\phi_M(m)$. We now have to check that this satisfies the comodule axioms. By definition, the counit on $H^\circ$ is defined by $\varepsilon(f)=f(1)$ and so
$$
(1\otimes\varepsilon)\circ\rho(m)=\sum_{i=1}^n f_i(1)m_i=1\cdot m=m
$$
as required. Finally we aim to show that the following diagram commutes:
\[
\begin{tikzcd}
M \arrow{r}{\rho} \arrow{d}{\rho} & M\otimes_R H^\circ \arrow{d}{1\otimes \Delta}\\
M\otimes_R H^\circ \arrow{r}{\rho\otimes 1} & M\otimes_R H^\circ\otimes_R H^\circ
\end{tikzcd}
\]
By Corollary \ref{comod}(ii), the natural map $M\otimes_R H^\circ\otimes_R H^\circ\to\Hom_R(H\otimes_R H,M)$ is injective. Hence it suffices to show that $(1\otimes\Delta)\circ\rho(m)$ and $(\rho\otimes 1)\circ\rho(m)$ act in the same way on $H\otimes_R H$ for all $m\in M$. But the former sends $x\otimes y$ to $(xy)\cdot m$ while the latter sends $x\otimes y$ to $x\cdot(y\cdot m)$ for any $x,y\in H$, which are clearly equal.
\end{proof}

Since Lemma \ref{comod} was quite general, the same argument as in the above proof shows the following

\begin{lem} Suppose $M$ is a locally finite $H$-module and let $C$ be a subcoalgebra of $H^\circ$ such that $H^*/C$ is torsion-free. If $\phi_M(m)$ belongs to the image of $\theta_{M,C}$ for all $m\in M$ then $M$ is a $C$-comodule.
\end{lem}

\bibliographystyle{abbrv}
\bibliography{bibforme}

\end{document}